\documentclass[11pt]{article}
\usepackage{epsfig}
\usepackage{amssymb,amsmath,amsthm,url}
\usepackage{multirow}
\usepackage{booktabs}
\usepackage{natbib}
\usepackage{setspace}

\usepackage{times,amsfonts,eucal}
\usepackage{algorithm}
\usepackage{algorithmic}
\usepackage{graphicx,ifthen}
\usepackage{wrapfig}
\usepackage{latexsym}
\usepackage{color}
\usepackage{mathrsfs}
\usepackage{bm}
\usepackage{bbm}
\usepackage{srctex}

\usepackage{bm}
\usepackage{natbib}
\newtheorem{thm}{Theorem}[section]
\newtheorem{lemma}{Lemma}[section]

\newtheorem{cor}{Corollary}[section]

\newtheorem{rem}{Remark}[section]
\newtheorem{exa}{Example}[section]
\newtheorem{assu}{Assumption}[section]
\numberwithin{equation}{section}
\numberwithin{thm}{section}
\numberwithin{lemma}{section}
\numberwithin{pro}{section}
\numberwithin{cor}{section}
\numberwithin{definition}{section}
\numberwithin{cons}{section}
\numberwithin{rem}{section}
\numberwithin{exa}{section}
\numberwithin{table}{section}
\numberwithin{figure}{section}

\newcommand{\mR}{\mathbb{R}}
\newcommand{\mZ}{\mathbb{Z}}

\newcommand{\mC}{\mathbb{C}}
\newcommand{\mE}{\mathbb{E}}
\newcommand{\mP}{\mathbb{P}}

\newcommand{\p}{\partial}

\newcommand{\A}{\mathbf{A}}

\newcommand{\C}{\mathbf{C}}
\newcommand{\D}{\mathbf{D}}
\newcommand{\I}{\mathbf{I}}
\renewcommand{\S}{\mathbf{S}}
\newcommand{\U}{\mathbf{U}}
\newcommand{\X}{\mathbf{X}}
\newcommand{\Z}{\mathbf{Z}}

\newcommand{\Y}{\tilde{Y}}
\def\bs{\boldsymbol}

\def\ic{\mathbf{i}}
\def\tr{\mathrm{Tr}}

\allowdisplaybreaks

\begin{document}

\title{Spectral analysis of linear time series in moderately high dimensions
\footnote{This research was partially supported by NSF grants DMR 1035468, DMS 1106690, DMS 1209226, DMS 1305858 and DMS 1407530
}}

\author{
Lili Wang\footnote{Department of Mathematics, Zhejiang University, Hangzhou 310027, China, email: \tt{liliwang@zju.edu.cn}}
\and Alexander Aue\footnote{Department of Statistics, University of California, Davis, CA 95616, USA, emails: \tt{[aaue,debpaul]@ucdavis.edu}}
\and Debashis Paul$^\ddagger$
}
\maketitle

\begin{abstract}
\setlength{\baselineskip}{1.8em}
This article is concerned with the spectral behavior of $p$-dimensional linear processes in the moderately high-dimensional case when both dimensionality $p$ and sample size $n$ tend to infinity so that $p/n\to0$. It is shown that, under an appropriate set of assumptions, the empirical spectral distributions of the renormalized and symmetrized sample autocovariance matrices converge almost surely to a nonrandom limit distribution supported on the real line. The key assumption is that the linear process is driven by a sequence of $p$-dimensional real or complex random vectors with i.i.d. entries possessing zero mean, unit variance and finite fourth moments, and that the $p\times p$ linear process coefficient matrices are Hermitian and simultaneously diagonalizable. Several relaxations of these assumptions are discussed. The results put forth in this paper can help facilitate inference on model parameters, model diagnostics and prediction of future values of the linear process.

\end{abstract}

\setlength{\baselineskip}{1.8em}


\section{Introduction}
\label{sec:intro}

In this article, the spectral properties of a class of multivariate linear time series are studied
through the bulk behavior of the eigenvalues of renormalized and symmetrized sample autocovariance matrices when
both the dimension $p$ and sample size $n$ are large but the dimension increases at a much slower rate compared to the sample size,
so that the dimension-to-sample size ratio $p/n$ converges to zero. The latter asymptotic regime will be referred to as
moderately high-dimensional scenario. Under this framework, the existence of
limiting spectral distributions (LSD) of the matrices $\mathbf{C}_\tau = \sqrt{n/p}(\mathbf{S}_\tau -
\mathbf{\Sigma}_\tau)$ is proved, where $\mathbf{S}_\tau$ is the symmetrized lag-$\tau$ sample autocovariance matrix and
$\mathbf{\Sigma}$ the lag-$\tau$ population autocovariance matrix, for $\tau =0,1,\ldots$
The analysis takes into account both temporal and dimensional correlation and the LSD is described in terms of a kernel that
is determined by the transfer function of a univariate linear time series. The results derived in this paper are
natural extensions of the work of \cite{Bai:Yin:1988}, who proved that the empirical spectral distribution of
normalized sample covariance matrices based on i.i.d.\ observations with zero mean and unit variance
converges to the semi-circle law under the same asymptotic regime. It also extends the work of \cite{Pan:Gao:2009} and
\cite{Wang:Paul:2014} in two different ways, first, by allowing nontrivial temporal dependence among the observation
vectors, and secondly, by describing the LSDs of renormalized sample autocovariance matrices of all lag orders.
We need to impose a certain structural assumption on the linear process, namely, that its coefficient matrices are symmetric (Hermitian for complex-valued data) and simultaneously diagonalizable. However, various ways to relax the latter assumption are discussed.

The results derived in this paper can be seen as natural counterparts of the works of \citet{Liu:Aue:Paul:2015},
who proved the existence of LSDs of symmetrized sample autocovariance matrices under the same structural assumptions
on the linear process but in the asymptotic regime $p,n\to \infty$ such that $p/n\to c\in (0,\infty)$. \cite{Jin:Wang:Bai:Nair:Harding:2014}
derived similar results under the assumption of i.i.d.\ observations, using them for detecting the presence of
factors in a class of dynamic factor models. Under the same asymptotic framework, the existence of the LSD of sample covariance matrices
when the different coordinates of the observed process are i.i.d.\ linear processes has been proved by \citet{Pfaffel:Schlemm:2011} and \citet{Yao:2012}.

The main results in this paper originally formed a part of the Ph.D.\ thesis of the first author \citep{Wang:2014}.
Very recently, we came to know through personal
communication from Arup Bose that \cite{Bhattacharjee:Bose:2015}
proved the existence of the LSD of symmetrized and normalized autocovariance matrices
for an MA($q$) process with fixed $q$, under a weaker assumption on the coefficient matrices
involving existence of limits of averaged traces of polynomials of these matrices, where
the limits satisfy certain requirements associated with a $*$-probability space.
They use free probability theory for their derivations and
therefore their approach is very different from the one presented in this paper, which relies on the
characterization of distributional convergence through the convergence of the corresponding Stieltjes transforms.


The main contribution of this paper is the precise description of the bulk behavior of the eigenvalues
of the matrices $\mathbf{C}_\tau$. These are natural objects to study if one is interested in the understanding
the fluctuations of the sample autocovariance matrices from their population counterparts, since the latter provide
useful information about the various characteristics of the observed process. Under the asmyptotic regime $p,n\to \infty$
with $p/n \to 0$, and under fairly weak regularity conditions,
the symmetrized sample autocovariance matrices converge to the corresponding population autocovariance
matrices in operator norm. However, stronger statements about the quality of the estimates are usually not possible
without imposing further restrictions on the process.
The results stated here provide a way to quantify the fluctuations of the estimated autocovariance matrices
from the population versions, and can be seen as analogous to the standard error bounds in univariate problems.
Indeed, if the quality of estimates is assessed through the Frobenius norm of $\mathbf{C}_\tau$,
or some other measure that can be expressed as a linear functional of the spectral distribution of $\mathbf{C}_\tau$,
the results presented in this paper give a precise description about the asymptotic behavior of such a measure in terms of
integrals of the LSD of $\mathbf{C}_\tau$. Some specific applications of the results are discussed in Section \ref{sec:app}.
A further importance of the results derived here is that they form the building block of further investigations
on the fluctuations of linear spectral statistics of matrices such as $\mathbf{C}_\tau$, thus raising the
possibility of generalizing results such as those obtained recently by \cite{Chen:Pan:2015}.

The rest of the paper is organized as follows. Section \ref{sec:main} gives the main results
are develops intuition. Section \ref{sec:examples} discusses some specific examples to
elucidate the main results. Section \ref{sec:app} discusses a number of potential applications. Sections
\ref{sec:proof:thm_maq}
--\ref{sec:proof:non-gauss} are devoted to describing
the key steps in the proofs of the main results. Further technical details are relegated to the technical Appendix.


\section{Main results}
\label{sec:main}


\subsection{Assumptions}
\label{sec:main:assu}

Let $\mathbb{Z}$, $\mathbb{N}_0$ and $\mathbb{N}$ denote integers, nonnegative integers and
positive integers, respectively. In the following, the linear process $(X_t\colon t\in\mathbb{Z})$
is studied, given by the set of equations
\begin{equation}\label{eq:lin_proc}
X_t=\sum_{\ell=0}^\infty\A_{\ell}Z_{t-\ell},
\qquad t\in\mathbb{Z},
\end{equation}
where $(\A_\ell\colon\ell\in\mathbb{N}_0)$ are coefficient matrices with $\A_0=\I_p$, the $p$-dimensional identity matrix, and $(Z_t\colon t\in\mathbb{Z})$ are innovations for which more specific assumptions are given below. If $\A_\ell=\mathbf{0}_p$, the $p$-dimensional zero matrix, for all $\ell>q$, then one has the $q$th order moving average, MA($q$), process
\begin{equation}\label{eq:expression_maq_process}
X_t=\sum_{\ell=0}^q\A_\ell Z_{t-\ell},
\qquad t\in\mathbb{Z}.
\end{equation}
In the following results will be stated and motivated first for the MA($q$) process and then extended to linear processes. Throughout the following set of conditions are assumed to hold.

\begin{assu}
\label{assu:innovations}
The innovations $(Z_t\colon t\in\mathbb{Z})$ consist of real- or complex-valued entries $Z_{jt}$ which are independent, identically distributed (iid) across time $t$ and dimension $j$ and satisfy
\vspace{-.3cm}
\begin{itemize}
\itemsep-.5ex
\item[{\bf Z1}]
$\mathbb{E}[Z_{jt}]=0$, $\mathbb{E}[|Z_{jt}|^2]=1$ and\/ $\mathbb{E}[|Z_{jt}|^4] < \infty$;
\item[{\bf Z2}]
In case of complex-valued innovations, the real and imaginary parts of $Z_{jt}$ are independent with $\mathbb{E}[\Re(Z_{jt})] =\mathbb{E}[\Im(Z_{jt})] = 0$ and\/ $\mathbb{E}[\Re(Z_{jt})^2] = \mathbb{E}[\Im(Z_{jt})^2] = 1/2$.
\end{itemize}
\end{assu}

\begin{assu}
\label{assu:coef_matrices}
Suppose that
\vspace{-.3cm}
\begin{itemize}
\itemsep-.5ex
\item[{\bf A1}]
$(\A_\ell\colon\ell\in\mathbb{N})$ are Hermitian and simultaneously diagonalizable, that is, there exists a unitary matrix $\U$ such that $\U^* \mathbf{A}_\ell \mathbf{U} = {\bs\Lambda}_\ell$, where ${\bs\Lambda}_\ell$ is a
diagonal matrix with real-valued diagonal entries;
\item[{\bf A2}]
The $j$th diagonal entry of ${\bs\Lambda}_\ell$ is given by $f_\ell(\bs\alpha_j)$, where $\bs\alpha_j \in \mathbb{R}^{m_0}$ for $j=1,\ldots,p$, where $m_0$ is fixed, and
$(f_\ell\colon\ell\in\mathbb{N})$ are continuous functions from $\mathbb{R}^{m_0}$ to $\mathbb{R}$;
\item[{\bf A3}]
As $p \to \infty$, the empirical distribution of $(\bs\alpha_j\colon j=1,\ldots,p)$
converges to a distribution on $\mathbb{R}^{m_0}$ denoted by $F^{\cal A}$;
\item[{\bf A4}]
There exist constants $\bar{a}_{0}=1$ and $(\bar{a}_\ell\colon \ell\in\mathbb{N})$ such that $\|f_\ell\|_\infty \leq \bar{a}_\ell$ for all $\ell\in\mathbb{N}$;
\item[{\bf A5}]
For some $r_0\geq 4$, there are positive constants $L_{j+1}$ such that $\sum_{\ell=0}^{\infty}\ell^j \bar{a}_{\ell}<L_{j+1}$
for $j=0,\ldots,r_0$. The conditions for $j>1$ are only needed for the extension of the results for MA$(q)$ processes to
linear processes; see Section \ref{sec:main:lin_proc}.
\end{itemize}
\end{assu}
\noindent The assumptions on the innovations $(Z_t\colon t\in\mathbb{Z})$ are standard in time series
and high-dimensional statistics contexts. The assumptions on the coefficient matrices $(A_\ell\colon\ell\in\mathbb{N})$
are similar to the ones imposed in \citet{Liu:Aue:Paul:2015} and generalize condition sets previously established
in the literature, for example, the ones in \citet{Pfaffel:Schlemm:2011}, \citet{Yao:2012} and \citet{Jin:Wang:Bai:Nair:Harding:2014}.



\subsection{Result for MA($q$) processes}
\label{sec:main:maq}

The objective of this section is to study the spectral behavior of the the lag-$\tau$ symmetrized
sample autocovariance matrices associated with the MA($q$) process $(X_t\colon t\in\mathbb{Z})$
defined in \eqref{eq:expression_maq_process} in the moderately high dimensional setting
\begin{equation}
\label{eq:mod_dim}
p,n\to\infty\qquad\mbox{such that}\qquad \frac{p}{n}\to 0.
\end{equation}
Extensions to the linear process \eqref{eq:lin_proc} are discussed in Section \ref{sec:main:lin_proc} below.
The symmetrized sample autocovariance matrices are given by the equations
\begin{equation}\label{eq:symmetrized_autocavariance_matrix_lag_tau}
\S_\tau
=\frac{1}{2(n-\tau)}\sum_{t=\tau+1}^{n}\big(X_{t}X_{t-\tau}^{*}+X_{t-\tau}X_{t}^{*}\big),
\qquad \tau\in\mathbb{N}_0,
\end{equation}
where $^*$ signifies complex conjugate transposition of both vectors and matrices. It should be noted
that $\S_0$ is simply the sample covariance matrix. Using the defining equations of the MA($q$) process, one can show that
\[
{\bs\Sigma}_\tau
=\mE[\S_\tau]
=\frac{1}{2}\bigg(\sum_{\ell=0}^{q-\tau}\big[\A_{\ell+\tau}\A_\ell^*+\A_\ell\A_{\ell+\tau}^*\big]\bigg),
\qquad \tau\in\mathbb{N}_0.
\]
Since, under \eqref{eq:mod_dim}, $\S_\tau$ is a consistent estimator for ${\bs\Sigma}_\tau$, one studies appropriately rescaled fluctuations of $\S_\tau$ about its mean ${\bs\Sigma}_\tau$. This leads to the renormalized matrices
\begin{equation}\label{normalized_c}
\C_\tau
=\sqrt{\frac np}\big(\S_\tau-{\bs\Sigma}_\tau\big),
\qquad \tau\in\mathbb{N}_0.
\end{equation}
To study the spectral behavior of $\C_\tau$, introduce its empirical spectral distribution (ESD) $\hat F_\tau$ given by
\[
\hat F_\tau(\lambda)
=\frac 1p\sum_{j=1}^p\mathbb{I}_{\{\lambda_{\tau,j}\leq \lambda\}},
\]
where $\lambda_{\tau,1},\ldots,\lambda_{\tau,p}$ are the eigenvalues of $\C_\tau$. In the RMT
literature, proofs of large-sample results about $\hat F_\tau$ are often based on convergence
properties of Stieltjes transforms \citep{Bai:Silverstein:2010}. The Stieltjes transform of a distribution
function $F$ on the real line is the function
\[
s_F\colon\mathbb{C}^+\to\mathbb{C}^+,\; z\mapsto s_F(z)
=\int\frac{1}{\lambda-z}dF(\lambda),
\]
where $\mathbb{C}^+=\{x+\ic y\colon x\in\mathbb{R},y>0\}$ denotes the upper complex half plane. Note that $s_F$
is analytic on $\mathbb{C}^+$ and that the distribution function $F$ can be obtained from $s_F$ using an inversion formula.

Let $f_0\colon \mathbb{R}^{m_0} \to \mathbb{R}$ be defined as $f_0(\mathbf{a}) = 1$ for all $\mathbf{a} \in \mathbb{R}^{m_0}$.
Define the MA($q$) transfer function
\begin{equation}\label{eq:pre_spectrum}
g(\mathbf{a},\nu) = \sum_{\ell=0}^q f_\ell(\mathbf{a}) e^{\ic \ell \nu},
\qquad \nu \in[0,2\pi], ~\mathbf{a} \in \mathbb{R}^{m_0},
\end{equation}
and the corresponding power transfer function
\begin{equation}\label{eq:spectrum}
\psi(\mathbf{a},\nu) = |g(\mathbf{a},\nu)|^2,
\qquad\nu\in[0,2\pi], ~\mathbf{a} \in \mathbb{R}^{m_0}.
\end{equation}
The effect of the temporal dependence on the spectral behavior of $\C_\tau$ is encoded through the
power transfer function $\psi(\mathbf{a},\nu)$. Keeping $\mathbf{a}$ fixed, it can be seen that
$\psi(\mathbf{a},\nu)$ is up to normalization the spectral density of a univariate MA($q$) process
with coefficients $f_1(\mathbf{a}),\ldots,f_q(\mathbf{a})$. This leads to the following result.

\begin{thm}\label{thm:LSD_variance}
If the MA($q$) process $(X_t\colon t\in\mathbb{Z})$ satisfies {\bf Z1}, {\bf Z2} and {\bf A1}--{\bf A5},
then, with probability one and in the moderately high-dimensional setting \eqref{eq:mod_dim}, $\hat F_\tau$
converges in distribution to a nonrandom distribution $F_\tau$ whose Stieltjes transform $s_\tau$ is given by
\begin{equation}\label{eq:Stieltjes_LSD_covariance}
s_\tau(z)
=  - \int \frac{d F^{\cal A}(\mathbf{b})}{z +\beta_\tau(z,\mathbf{b})},
\qquad z \in \mathbb{C}^+,
\end{equation}
where
\begin{equation}\label{eq:Stieltjes_kernel_covariance}
\beta_\tau(z,\mathbf{a})
= - \int \frac{{\cal R}_\tau(\mathbf{a},\mathbf{b})
dF^{\cal A}(\mathbf{b})}{z + \beta_\tau(z,\mathbf{b})},
\qquad z \in \mathbb{C}^+,~\mathbf{a} \in \mathbb{R}^{m_0},
\end{equation}
and
\begin{equation}
\label{eq:spectral_kernel_covariance}
{\cal R}_\tau(\mathbf{a},\mathbf{b}) = \frac{1}{2\pi} \int_0^{2\pi} \cos^2(\tau
\theta) \psi(\mathbf{a},\theta) \psi(\mathbf{b},\theta) d\theta,
\qquad \mathbf{a},\mathbf{b} \in \mathbb{R}^{m_0}.
\end{equation}
Moreover, $\beta_\tau(z,\mathbf{a})$ is the unique solution to (\ref{eq:Stieltjes_kernel_covariance})
subject to the condition that it is a Stieltjes kernel, that is, for each $\mathbf{a} \in \mathrm{supp}(F^{\cal A})$,
$\beta_\tau(z,\mathbf{a})$ is the Stieltjes transform of a measure on the real line with mass
$\int {\cal R}_\tau(\mathbf{a},\mathbf{b}) dF^{\cal A}(\mathbf{b})$.
\end{thm}
Since it only differs from the spectral density of an MA$(q)$ process by a constant,
it follows that $\psi(\mathbf{a},\theta)$ is strictly positive for all arguments.
Consequently, ${\cal R}_\tau(\mathbf{a},\mathbf{b})$ and $\int{\cal R}_\tau(\mathbf{a},\mathbf{b})
dF^{\cal A}(\mathbf{b})$ are always strictly positive as well. The intuition for the proof of
Theorem \ref{thm:LSD_variance} is given in the next section and will then be completed in Section \ref{sec:proof:thm_maq}.
\begin{rem}\label{rem:MA_q_kernel}
It is easily checked, that for an MA$(q)$ process, the kernel ${\cal R}_\tau(\mathbf{a},\mathbf{b})$ is the
same for all $\tau \geq q+1$. This implies that the Stieltjes transforms $s_\tau(z)$, and hence the LSDs
(limiting spectral distributions) of $\sqrt{n/p}(\mathbf{S}_\tau - \Sigma_\tau)$ are the same for $\tau \geq q+1$.
\end{rem}


\subsection{Intuition for Gaussian MA($q$) processes}
\label{sec:main:intuition}

Assume for now that the innovations $(Z_t\colon t\in\mathbb{Z})$ are complex Gaussian,
the extension to general innovations will be established in the Appendix.
Define the $p\times n$ data matrix $\mathbf{X}=[X_1:\cdots:X_n]$ and the
$p\times n$ innovations matrix $\mathbf{Z}=[Z_1:\cdots:Z_n]$. Using the
$n\times n$ lag operator matrix $\mathbf{L}=[o:e_1:\cdots:e_{n-1}]$,
where $o$ and $e_j$ denote the zero vector and the $j$th canonical unit vector,
respectively, it follows that
\begin{equation}
\mathbf{X}=\sum_{\ell=0}^q\mathbf{A}_\ell\mathbf{Z}\mathbf{L}^\ell + \sum_{\ell=1}^q \mathbf{A}_\ell \mathbf{Z}_{[-q]}\mathbf{L}^{\ell-q},
\end{equation}
where $\mathbf{Z}_{[-q]} = [Z_{-q+1}:\cdots:Z_0: 0:\cdots:0]$ is a $p\times n$ matrix and $\mathbf{L}^{\ell-q} = (\mathbf{L}^{q-\ell})^{-1}$.
In the next step, $\mathbf{L}$ is approximated by the circulant matrix
$\tilde{\mathbf{L}}=[e_n:e_1:\cdots:e_{n-1}]$. As in \citet{Liu:Aue:Paul:2015}, one defines the matrix $\bar{\mathbf{X}}=\sum_{\ell=0}^q\A_\ell\mathbf{Z}\tilde{\mathbf{L}}^\ell$ that differs from
$\mathbf{X}$ only in the first $q$ columns. Let $\mathbf{F}_{n}=\left[e^{\mathbf{i}s\nu_{t}}\right]_{s,t =1}^{n}$, with $\nu_{t}=2\pi t/n$, be a Fourier rotation matrix and  $\tilde{\bs\Lambda}_{n}=\mbox{diag}(e^{\mathbf{i}\nu_{1}},\ldots,e^{\mathbf{i}\nu_{n}})$. Then
\begin{equation}\label{L_decomposition}
\tilde{\mathbf{L}}= \mathbf{F}_{n}\tilde{\bs\Lambda}_{n}\mathbf{F}_{n}^{*}.
\end{equation}
Using this and noticing that $\mathbf{X}$ and $\bar{\mathbf{X}}$ differ by a matrix of rank
$q$, it can be seen that as long as $q$ small compared to $p$, $\S_\tau=(n-\tau)^{-1}\X\D_\tau\X^*$ can be approximated by $\bar{\S}_\tau=(n-\tau)^{-1}\bar{\X}\bar{\D}_\tau\bar{\X}^*$, where  $\D_{\tau} =[\mathbf{L}^{\tau}+(\mathbf{L}^{\tau})^{*}]/2$ and
$\bar{\D}_{\tau} =[\tilde{\mathbf{L}}^{\tau}+(\tilde{\mathbf{L}}^{\tau})^{*}]/2$. Notice next that, due to the assumed
Gaussianity of the innovations, the entries of $\tilde{\Z}=\mathbf{U}^*\Z \mathbf{F}_{n}$ are iid copies of the entries
of $\mathbf{Z}$, with $\mathbf{U}$ denoting the matrix diagonalizing the coefficient matrices $(\A_\ell\colon\ell\in\mathbb{N})$.
Define then $\mathbf{\tilde{S}}_{\tau}=\mathbf{U}^{*}\bar{\mathbf{S}}_{\tau}\mathbf{U}$ and
\begin{equation}\label{eq:transformed_covariances}
\tilde{\mathbf{C}}_\tau = \sqrt{\frac{n}{p}} (\tilde{\mathbf{S}}_\tau - \bs\tilde{\bs\Sigma}_\tau),
\end{equation}
where $\tilde{\bs\Sigma}_{\tau}=\mE[\tilde{\mathbf{S}}_{\tau}]$ is a diagonal matrix. It will
be shown in Section \ref{proof:thm_maq:lsd_ctau} that the LSD of
$\mathbf{C}_\tau^U=\mathbf{U}^{*}\mathbf{C}_\tau\mathbf{U}$ is the same as that of $\tilde{\mathbf{C}}_\tau$.


\subsection{Extensions to linear processes}
\label{sec:main:lin_proc}

In this section, Theorem \ref{thm:LSD_variance} is extended to cover linear processes
as defined in \eqref{eq:lin_proc}. To do so, the continuity condition {\bf A2} is
strengthened to assumption {\bf A6} below. In order to approximate the linear
process with MA$(q)$ models of increasing order, a rate on $q$ is imposed.

\begin{assu}
\label{assu:lin-proc}
The following conditions are assumed to hold.
\vspace{-.3cm}
\begin{itemize}
\itemsep-.5ex
\item[{\bf A6}]
($f_\ell\colon\ell\in\mathbb{N})$ are Lipschitz functions such that $|f_\ell(\mathbf{a})-f_\ell(\mathbf{b})|\leq
C\ell^{r_1}\|\mathbf{a}-\mathbf{b}\|$ for $\mathbf{a},\mathbf{b}\in \mR^{m_0}$ and $\ell\in\mathbb{N}$, where $r_1\leq r_0$ and $r_0$
is as in {\bf A5};
\item[{\bf A7}]
The moving average order $q$ satisfies $q=O(p^{1/4})$.
\end{itemize}
\end{assu}
\noindent Analogously \eqref{eq:pre_spectrum} and \eqref{eq:spectrum} are extended to the linear process transfer function and power transfer function
\begin{equation}
\label{eq:spectrum_lin_proc}
g(\mathbf{a},\nu)=\sum_{\ell=0}^\infty f_\ell(\mathbf{a})e^{\ic\ell\nu}
\qquad\mbox{and}\qquad
\psi(\mathbf{a},\nu)=|g(\mathbf{a},\nu)|^2,
\qquad \nu\in[0,2\pi],~\mathbf{a}\in\mathbb{R}^{m_0},
\end{equation}
respectively. Then, the following result holds.

\begin{thm}
\label{thm:lin_proc}
If the linear process $(X_t\colon t\in\mathbb{Z})$ satisfies {\bf Z1}, {\bf Z2} and {\bf A1}--{\bf A7}, then, the result of Theorem \ref{thm:LSD_variance} is retained if \eqref{eq:spectrum_lin_proc} is used in place of \eqref{eq:pre_spectrum} and \eqref{eq:spectrum}.
\end{thm}

The proof of Theorem \ref{thm:lin_proc} is based on a truncation argument, approximating the
linear process with MA($q$) processes of increasing order $q$. More delicate arguments are needed for this case
as the intuitive arguments outlined in the previous section do not carry over to this case. Indeed conditions
on the approximating MA($q$) processes are needed that ensure that $q$ does not grow too fast or too slow in
order for the LSD of the linear process and its truncated version to be the same. The proof details are
given in Section \ref{sec:proof:lin_proc} below, where it turns out that one can choose $q=O(p^{1/4})$ as specified in {\bf A7}.

As a further generalization, consider the process $(Y_t\colon t\in\mathbb{Z})$ that is obtained from the linear
process $(X_t\colon t\in\mathbb{Z})$ through
\begin{equation}
\label{eq:yt}
Y_t=\mathbf{B}^{1/2}X_t,
\qquad t\in\mathbb{Z},
\end{equation}
where it is assumed that
\begin{itemize}
\item[{\bf A8}] $\mathbf{B}^{1/2}$ is a square root of the nonnegative definite Hermitian matrix $\mathbf{B}$ with $\|\mathbf{B}\|\leq \bar{b}_{0}<\infty$, and there is a nonnegative measurable function $g_{B}$, not identically zero
    on supp$(F^{\cal A})$, such that for each $p$, $\mathbf{U}^{*}\mathbf{B}\mathbf{U}=\mbox{diag}(g_{B}(\bs \alpha_{1}),\cdots, g_{B}(\bs \alpha_{p}))=\mathbf{\Lambda}_{B}$, with $\mathbf{U}$ and $\bs \alpha_1, \cdots, \bs \alpha_p$ as defined in {\bf A1} and {\bf A2}.
\end{itemize}
Observe that the autocovariance matrices of the process $(Y_t\colon t\in\mathbb{Z})$ are given by $\S_\tau^Y=\mathbf{B}^{1/2}\S_\tau\mathbf{B}^{1/2}$ and have expectation $\bs\Sigma^{Y}_{\tau}=\mathbf{B}^{1/2}\bs\Sigma_{\tau}\mathbf{B}^{1/2}$. Assumption {\bf A8} shows that the approximating autocovariance matrix obtained from replacing the lag operator matrix with the corresponding circulant matrix takes on the form
\begin{equation}
\tilde{\S}_\tau^Y
=\frac{1}{n-\tau}
\bigg(\sum_{\ell=0}^{\infty}\sqrt{\mathbf{\Lambda}_{B}}\mathbf{\Lambda}_{\ell}\tilde{\mathbf{Z}}\tilde{\mathbf{\Lambda}}_{n}^{\ell}\bigg)
\bigg(\frac{\tilde{\mathbf{\Lambda}}_{n}^{\tau}+(\tilde{\mathbf{\Lambda}}_{n}^{\tau})^*}{2}\bigg)
\bigg(\sum_{\ell=0}^{\infty}\sqrt{\mathbf{\Lambda}_{B}}\mathbf{\Lambda}_{\ell}\tilde{\mathbf{Z}}\tilde{\mathbf{\Lambda}}_{n}^{\ell}\bigg)^{*}
\end{equation}
with expectation $\bs\tilde{\bs\Sigma}_{\tau}^{Y}=\mbox{diag}(\tilde{\sigma}_{\tau, 1}^{Y},\ldots,\tilde{\sigma}_{\tau, p}^{Y})$ and
\[
\tilde{\sigma}_{\tau, j}^{Y}
=\frac{1}{n-\tau}\sum_{t=1}^{n}g_{B}(\bs \alpha_{j})\cos(\tau \nu_t)\psi(\bs \alpha_{j}, \nu_{t}),
\]
in which $\psi(\bs \alpha_{j}, \nu_{t})$ is defined in (\ref{eq:spectrum_lin_proc}). Following similar arguments as in the finite and infinite order MA cases, it can be shown that the LSD of $\mathbf{C}_{\tau}^{Y}=\sqrt{n/p}(\mathbf{S}_{\tau}^{Y}-\bs\Sigma_{\tau}^{Y})$ is the same as that of $\tilde{\C}_{\tau}^{Y}=\sqrt{n/p}(\tilde{\mathbf{S}}_{\tau}^{Y}-\bs\tilde{\bs\Sigma}_{\tau}^{Y})$. Then, the following theorem is established.


\begin{thm}
\label{thm:LSD_generalized_process_B}
If the process $(Y_t\colon t\in\mathbb{Z})$ defined in \eqref{eq:yt} satisfies {\bf Z1}, {\bf Z2} and {\bf A1}--{\bf A8},
then, with probability one and in the moderately high-dimensional setting \eqref{eq:mod_dim}, $\hat F_\tau^Y$ converges
in distribution to a nonrandom distribution $F_\tau^Y$ whose Stieltjes transform $s_\tau^Y$ is given by
\begin{equation}
\label{eq:Stieltjes_LSD_generalized_process_B}
s_\tau^Y(z)
=  - \int \frac{d F^{\cal A}(\mathbf{a})}{z + \beta_\tau^Y(z,\mathbf{a})},
\qquad z \in \mathbb{C}^+,
\end{equation}
where
\begin{equation}\label{eq:Stieltjes_kernel_generalized_process_B}
\beta_\tau^Y(z,\mathbf{a})
= - \int \frac{g_{B}(\mathbf{a})g_{B}(\mathbf{b}){\cal R}_\tau(\mathbf{a},\mathbf{b})dF^{\cal A}(\mathbf{b})}
{z + \beta_\tau^Y(z,\mathbf{b})},
\qquad z \in \mathbb{C}^+,
~\mathbf{a} \in \mathbb{R}^{m_0},
\end{equation}
and ${\cal R}_\tau(\mathbf{a},\mathbf{b})$ is defined in \eqref{eq:spectral_kernel_covariance}.
Moreover, $\beta_\tau^Y(z,\mathbf{a})$ is the unique solution to (\ref{eq:Stieltjes_kernel_generalized_process_B})
subject to the condition that it is a Stieltjes kernel, that is, for each $\mathbf{a} \in \mathrm{supp}(F^{\cal A})$,
$\beta_\tau^Y(z,\mathbf{a})$ is the Stieltjes transform of a measure on the real line with mass
$g_B(\mathbf{a})  \int g_B(\mathbf{b}) {\cal R}_\tau(\mathbf{a},\mathbf{b}) dF^{\cal A}(\mathbf{b})$
whenever $g_B(\mathbf{a}) > 0$.
\end{thm}


\subsection{Relaxation of commutativity condition}\label{subsec:relax_commutative}

The assumption of commutativity or simultaneous diagonalizability of the coefficients (assumption {\bf A1})
indeed restricts the class of linear processes for which the main result of existence and uniqueness of the
limiting ESD applies. However, this assumption can be relaxed to a milder one in which the coefficients of the
linear processes are only approximately Hermitian and commutative. Two such scenarios are discussed below,
which are natural but by no means exhaustive. In both settings, it is assumed that the linear process
\begin{equation}\label{eq:linear_process_B}
X_t = \sum_{\ell=0}^\infty \mathbf{B}_\ell Z_{t-\ell},
\qquad t \in \mathbb{Z},
\end{equation}
is observed with the standard assumptions {\bf Z1} and {\bf Z2} on the sequence $(Z_t\colon t\in\mathbb{Z})$,
whereas $\mathbf{B}_0=\mathbf{I}_p$ and the sequence $(\mathbf{B}_\ell\colon\ell\in\mathbb{N})$ satisfies the conditions:
\begin{itemize}
\itemsep-.5ex
\item[{\bf B1}] For some $r_0\geq 1$, there are $\bar{b}_0=1$ and $(\bar{b}_\ell\colon\ell\in\mathbb{N})$ such that $\|\mathbf{B}_\ell\|\leq \bar{b}_\ell$ for $\ell\in\mathbb{N}$ and $L_{j+1}^\prime:=\sum_{\ell=0}^\infty \ell^j \bar{b}_\ell < \infty$ for $j=0,\ldots,r_0$.
\item[{\bf B2}] There is a sequence of Hermitian matrices $(\mathbf{A}_\ell\colon\ell\in\mathbb{N})$ approximating the sequence $(\mathbf{B}_\ell\colon\ell\in\mathbb{N})$ and satisfying {\bf A1} -- {\bf A6}.
\end{itemize}
In addition to {\bf B1} and {\bf B2}, it is assumed that the sequence $(\mathbf{A}_\ell\colon\ell\in\mathbb{N})$ satisfies one of the following conditions specifying the approximation property in {\bf B2}:
\begin{itemize}
\itemsep-.5ex
\item[{\bf B3}] For some $1 \leq \beta < 4$, $p^{-1} \sum_{\ell=1}^{\lceil p^{1/\beta}\rceil} \mbox{rank}(\mathbf{B}_\ell -
\mathbf{A}_\ell) \to 0$ under \eqref{eq:mod_dim}.
\item[{\bf B4}] For some $1 \leq \beta < 4$, $\sqrt{n/p} \sum_{\ell=1}^{\lceil p^{1/\beta}\rceil} \parallel
\mathbf{B}_\ell - \mathbf{A}_\ell \parallel \to 0$ under \eqref{eq:mod_dim}.
\end{itemize}
The importance of these conditions is discussed. First, restricting the sums involving $\mathbf{B}_\ell -
\mathbf{A}_\ell$ to first $p^{1/\beta}$ terms is sufficient in view of {\bf B1} ensuring that the process $(X_t\colon t\in\mathbb{Z})$ can be approximated by the truncated process given by $X_t^{q} =\sum_{\ell=0}^{q} \mathbf{B}_\ell Z_{t-\ell}$ with $q = O(p^{1/4})$ without changing the LSD of $\sqrt{n/p}(\mathbf{S}_\tau - \mE[\mathbf{S}_\tau])$. This can be verified by following the derivation in Section \ref{sec:main:lin_proc}. The condition {\bf B3}, on the other hand,  says that the coefficient matrices $(\mathbf{B}_\ell\colon\ell\in\mathbb{N})$ can be seen as
low-rank perturbations of a sequence of Hermitian and commutative matrices $(\mathbf{A}_\ell\colon\ell\in\mathbb{N})$. The condition {\bf B4}, which bounds the norms of differences between the coefficients and
their approximations, is a bit restrictive in the sense that it depends on $n$. Presence of the factor
$\sqrt{n/p}$ suggests that this condition is non-trivial essentially
if $n$ is moderately large compared to $p$.

We state the result in the form of the following corollary.
\begin{cor}
\label{cor:LSD_linear_process_B}
Suppose that the linear process $(X_t\colon t\in\mathbb{Z})$ satisfies conditions {\bf B1}, {\bf B2} and either {\bf B3}
or {\bf B4}, and let $\mathbf{S}_\tau$ denote the lag-$\tau$ symmetrized sample
autocovariance matrix. Then the limiting ESD of the matrix $\sqrt{n/p}(\mathbf{S}_\tau-\mE[\mathbf{S}_\tau])$
exists and its Siteltjes transform $s_\tau(z)$ satisfies \eqref{eq:Stieltjes_LSD_covariance}--\eqref{eq:spectral_kernel_covariance}.
\end{cor}
Proof of Corollary \ref{cor:LSD_linear_process_B} is given in Appendix \ref{sec:proof_cor_linear_process}.

The conditions imposed in Corollary \ref{cor:LSD_linear_process_B} can be used to prove that results
hold for processes $(X_t\colon t\in\mathbb{Z})$ satisfying \eqref{eq:linear_process_B} and whose
coefficient matrices are certain classes of symmetric (Hermitian) Toeplitz matrices. Specifically,
if the matrix $\mathbf{B}_\ell$ is determined by the sequence $(b_{\ell k}\colon k \in \mathbb{Z})$,
satisfying the condition $\sup_{\ell \geq 1} \sum_{|k| \geq m} |k|^s |b_{\ell k}| \to 0$ as $m\to \infty$ for
some $s \geq 1$, and {\bf B1} holds, then the LSDs of the corresponding normalized sample autocovariance
matrices exist under \eqref{eq:mod_dim} provided $n = O(p^{s+1/2})$. In this case, the symmetric (Hermitian)
Toeplitz matrices $\mathbf{B}_\ell$ can be approximated by symmetric (Hermitian) circulant matrices whose
eigenvalues are precisely the symbols associated with the sequence $(b_{\ell k}\colon k \in \mathbb{Z})$
evaluated at the discrete Fourier frequencies $2\pi j/p$, $j=1,\ldots,p$.

\section{Examples}
\label{sec:examples}


In this section, a number of special cases are presented for which the results stated in
Section \ref{sec:main} take on an easier form.
\begin{exa}
Consider the MA(1) process
\[
X_{t} = Z_t+\A Z_{t-1},
\]
with $\A = \mathrm{diag}(\alpha_1, \cdots, \alpha_p)$ for $\alpha_{j}\in\mR$ (thus choosing $m=1$ here). 
Suppose further that $f_{1}(a)=a$. Then, the transfer function \eqref{eq:pre_spectrum} is given by
$g(a, \theta)=1+ae^{\ic \theta}$ and the power transfer function \eqref{eq:spectrum} by 
$\psi(a,\theta) = 1+a^2+2a\cos(\theta)$. This yields the explicit expressions
\[
\mathcal{R}_{\tau}(a,b)=
\begin{cases}
(1+a^2)(1+b^2)+2ab, & \tau=0. \\
(1+a^2)(1+b^2)/2+3ab/2, & \tau=1. \\
(1+a^2)(1+b^2)/2+ab, & \tau\geq 2.
\end{cases}
\]
\end{exa}
\begin{exa}
\label{exa:separable}
Consider the special case of an MA($q$) process with $\mathbf{A}_\ell=\gamma_\ell\I_{p}$, $\ell =1,\ldots,q$, and $f_{\ell}(\bs \alpha_{j})=\gamma_{\ell}$ with ${\bs\alpha}_{j} =\mathbf{1}$ for all $j=1,\ldots,p$. Then $F^{\mathcal{A}}$ is a $\delta$-function at $\mathbf{1}$. Since
\begin{equation*}
g(\mathbf{a},\nu )
=\sum_{\ell=0}^{\infty}f_{\ell}(\mathbf{a})e^{\ic \ell\nu}
=\sum_{\ell=0}^{\infty}\gamma_{\ell}e^{\ic \ell\nu}
=\tilde{g}(\nu)
\end{equation*}
and therefore also $\psi(\mathbf{a},\nu)=\tilde{\psi}(\nu)$ do not depend on $\mathbf{a}$, it follows that
\begin{equation*}
\mathcal{R}_{\tau}(\mathbf{a},\mathbf{1})
=\frac{1}{2\pi}\int_{0}^{2\pi}\cos^2(\tau\nu)(\tilde{\psi}(\nu))^2d\nu
=\bar{\mathcal{R}}_\tau,
\end{equation*}
so that equations \eqref{eq:Stieltjes_kernel_covariance} and \eqref{eq:Stieltjes_LSD_covariance} reduce respectively to $\beta_{\tau}(z,\mathbf{a}) = \beta_{\tau}(z) = \bar{\mathcal{R}}_{\tau}s_{\tau}(z)$ and
\[
s_{\tau}(z) = -\frac{1}{z+\bar{\mathcal{R}}_{\tau}s_{\tau}(z)}.
\]
For\/ $\tau = 0$, the latter equation coincides with that for the Stieltjes transform for the case of independent observations with separable covariance structure discussed in \citet{Wang:Paul:2014}. Indeed, taking in their notation $\A_{p} = \I_{p}$  and $\mathbf{B}_{n}^{1/2} =\mathrm{diag}(\tilde{g}(\nu_1), \cdots, \tilde{g}(\nu_n))$,  equation (2.1) of Theorem 2.1 in \citet{Wang:Paul:2014} reduces to $s(z) = -[z+\bar{b}_2 s(z)]^{-1}$, where
\begin{equation*}
\bar{b}_{2}
=\lim_{n\to \infty} \frac{1}{n}\tr(\mathbf{B}_{n}^2)
=\lim_{n\to \infty}\frac{1}{n}\sum_{t=1}^{n}|\tilde{g}(\nu_{t})|^4
=\frac{1}{2\pi}\int_{0}^{2\pi}\tilde{\psi}(\nu)^2d\nu
= \bar{\mathcal{R}}_{0}.
\end{equation*}
\end{exa}
\begin{exa}
Consider the AR(1) process
\[
X_{t} = \mathbf{A}X_{t-1}+Z_{t},
\]
with $\A = \mathrm{diag}(\alpha_1, \cdots, \alpha_p)$ for $\alpha_j\in\mathbb{R}$ such that $|\alpha_j|<1$. The AR(1) process then admits the linear process representation $X_t=\sum_{\ell=0}^\infty\A^\ell Z_{t-\ell}$. With $f_\ell(a)=a^\ell$, the transfer function \eqref{eq:pre_spectrum} is given by $g(a,\theta)=(1-ae^{\ic\theta})^{-1}$ and the power transfer function \eqref{eq:spectrum} by $\psi(a,\theta)=(1+a^2-2a\cos\theta)^{-1}$.
\end{exa}
\begin{exa} Consider the causal ARMA(1,1) process
\begin{equation*}
\mathbf{\Phi} (L)X_{t}
= \mathbf{\Theta}(L)Z_{t},
\end{equation*}
where $\mathbf{\Phi} (L) = \mathbf{I}-\mathbf{\Phi}_{1}(L)$ and $\mathbf{\Theta}(L) = \mathbf{I}+\mathbf{\Theta}_{1}$ are matrix-valued autoregressive and moving average polynomials in the lag operator $L$ such that $\|\Phi_1\|<\infty $ and $\|\Theta_1\|< \infty$. Then $(X_{t}\colon t \in \mZ)$ can be represented as the linear process
\begin{equation*}
X_{t}
=\mathbf{A}(L) Z_{t},
\end{equation*}
in which $\mathbf{A}(L) = \sum_{\ell=0}^{\infty}\mathbf{A}_{\ell}L^{\ell} = (\mathbf{I}-\mathbf{\Phi}_{1}L)^{-1}(\mathbf{I}+\mathbf{\Theta}_{1}L)$. Assume further that $\mathbf{\Phi}_{1}$ and $\mathbf{\Theta}_{1}$ are simultaneously diagonalizable by $\mathbf{U}$, that is, $\mathbf{U}^{*}\mathbf{\Phi}_{1}\mathbf{U} = \mathrm{diag}(\phi_{1}, \ldots, \phi_{p})$ and $\mathbf{U}^{*}\mathbf{\Theta}_{1}\mathbf{U} = \mathrm{diag}(\theta_1, \ldots, \theta_{p})$. Let $\bs \alpha_{j} = (\phi_{j}, \theta_{j})^{T}\in \mR^2$. Assumption {\bf{A3}} requires that the empirical distribution of $\{\bs \alpha_1, \ldots, \bs \alpha_p\}$ converges weakly to a non-random distribution function defined on $\mR^2$. Note that
\begin{equation*}
\mathbf{U}^{*}\mathbf{A}_{\ell}\mathbf{U}
= \mathrm{diag}(f_{\ell}(\bs \alpha_1), \ldots, f_{\ell}(\bs\alpha_{p})),
\end{equation*}
with $f_{\ell}(\bs \alpha_j) = 1$ and $f_{\ell}(\bs \alpha_j) = (\theta_{j} + \phi_{j})\phi^{\ell-1}_{j}$ for $\ell\in \mathbb{N}$. Thus, the transfer function \eqref{eq:pre_spectrum} is given by
\begin{equation*}
g(\bs \alpha_{j}, \nu)
=\sum_{\ell=0}^{\infty}f_{\ell}(\bs \alpha_{j})e^{\ic \ell \nu}
=1+\sum_{\ell=1}^{\infty}(\theta_{j}+\phi_{j})\phi^{\ell-1}_{j}e^{\ic \ell \nu}
=\frac{1+\theta_{j}e^{\ic \nu}}{1-\phi_{j}e^{\ic \nu}}
\end{equation*}
and the power transfer function \eqref{eq:spectrum} is the squared modulus of the ratio on right-hand side of the last equation.
\end{exa}
\begin{exa}\label{ex:multi_block}
Suppose that  for each $\ell \geq 1$, $\mathbf{A}_\ell$ is a block diagonal matrix with $B$ (a fixed number) diagonal blocks
such that the $b$th block of $\mathbf{A}_\ell$  is of the form
$a_{\ell b} \mathbf{I}_{p_b}$, for $b=1,\ldots,B$, where $\sum_{b=1}^B p_b = p$, and
$\sum_{\ell=1}^\infty \ell^3 \max_{1\leq b \leq B} |a_{\ell b}| < \infty$. Suppose further that for
each $b$, $p_b/p \to \omega_b$ as $p \to \infty$, where $\omega_b > 0$ for all $b$.
In this case, one can take $\boldsymbol{\alpha}_j = b/(m+1)$ if\/ $\sum_{b^\prime=1}^{b-1} p_{b^\prime} +1 \leq j \leq \sum_{b^\prime=1}^{b} p_{b^\prime}$
and define $f_\ell$ to be a function on $[0,1]$ that smoothly interpolates the values $\{(b/(m+1),a_{\ell b})\colon b=1,\ldots,B\}$.
Then, Theorem \ref{thm:lin_proc} applies and the Stieltjes transform $s_\tau(z)$ of the LSD of $\sqrt{n/p}(\mathbf{S}_\tau-\mathbf{\Sigma}_\tau)$ is given by
\begin{equation}\label{eq:s_tau_multi_block}
s_\tau(z) = - \sum_{b=1}^B \omega_b \frac{1}{z+\beta_{\tau,b}(z)},
\qquad z \in \mathbb{C}^+,
\end{equation}
where the functions (Stieltjes transforms) $\beta_{\tau,b}(z)$ are determined by the system of nonlinear equations
\begin{equation}\label{eq:beta_tau_multi_block}
\beta_{\tau,b}(z)
= - \sum_{b^\prime=1}^B \omega_{b^\prime} \frac{\bar{R}_{\tau,bb^\prime}}{z + \beta_{\tau,b^\prime}(z)},
\qquad z \in \mathbb{C}^+,
~~b=1,\ldots,B,
\end{equation}
where
\begin{equation}
\bar{R}_{\tau,bb^\prime}
= \frac{1}{2\pi}\int_{0}^{2\pi} \cos^2(\tau \theta)\tilde\psi_b(\theta) \tilde\psi_{b^\prime}(\theta)d\theta
\end{equation}
with $\tilde\psi_b(\theta) = |1+\sum_{\ell=1}^\infty a_{\ell b} e^{i\ell \theta}|^2$. Note that,
using the notations of Theorem \ref{thm:lin_proc}, $\beta_{\tau,b}(z) \equiv \beta_\tau(z,\mathbf{a})$ for $\mathbf{a} = b/(m+1)$,
and $F^{\cal A}$ is the discrete distribution that associates probability $\omega_b$ to the point $b/(m+1)$, for $b=1,\ldots,B$.
This example illustrates that often the precise description of $f_\ell$'s is not necessary in order for the LSDs to exist.
Numerical methods, such as a fixed point method, for solving (\ref{eq:beta_tau_multi_block}), while ensuring that $\Im(\beta_{\tau,b}(z)) >0$
whenever $z \in \mathbb{C}^+$, are easy to implement, and can be used to compute $s_\tau(z)$ for any given $z$.
\end{exa}


\section{Applications}
\label{sec:app}


The main result (Theorem \ref{thm:lin_proc}) gives a description of the bulk behavior
of the eigenvalues of the matrices $\mathbf{C}_\tau = \sqrt{n/p}(\mathbf{S}_\tau - \mathbf{\Sigma}_\tau)$
under the stated assumptions on the process and the asymptotic regime $p/n \to 0$.
Thus, this result provides a building block for further investigation of the behavior of
spectral statistics of the same matrix. It can also be used to investigate
potential departures from a hypothesized model.

An immediate application of Theorem \ref{thm:lin_proc} is that it provides a way of calculating
an error bound on $\mathbf{S}_\tau$ as an estimate of $\mathbf{\Sigma}_\tau$.
Indeed, if the quality of estimates is assessed through the Frobenius norm of $\mathbf{C}_\tau$,
or some other measure that can be expressed as a linear functional of the spectral distribution of $\mathbf{C}_\tau$,
our result gives a precise description about the asymptotic behavior of such a measure in terms of
integrals of the LSD of $\mathbf{C}_\tau$.  This can be seen as analogous to the standard error bounds
in univariate problems.

One potential application is in the context of model diagnostics.
Using the results for the LSD of the normalized symmetrized autocovariance matrices,
one can check whether the residuals from a time series regression model
have i.i.d.\ realizations. This can be done by graphically comparing the eigenvalue
distributions of $\sqrt{n/p}\S_1^e, \sqrt{n/p}\S_2^e, \ldots$, where $\S_{\tau}^e$
is the lag-$\tau$ symmetrized autocovariance matrix of the residuals obtained from fitting a
time series regression model, with the LSDs of the renormalized autocovariances
of the same orders corresponding to i.i.d.\ data.

Further, these results can also be used to devise a formal test for the hypothesis
$H_{0}\colon X_1, \ldots, X_n$ are i.i.d.\ with zero mean and known covariance versus
$H_{1}\colon X_1, \ldots, X_n$ follow a stationary linear time series model. If
an MA($q_{0}$) process ($q_{0}$ can be $\infty$) is specified, another type of test may be proposed, say,
$H_{0}\colon X_{t}$~{is the given MA}($q_{0}$)~{process (satisfying the assumptions of Theorem \ref{thm:lin_proc})},
versus the alternative that $X_t$ is a different process than the one specified under $H_{0}$. This can be done
through the construction of
a class of test statistics that equal the squared
integrals of the differences between the ESDs of observed renormalized sample covariance and
autocovariance matrices and the corresponding LSDs under $H_{0}$, for certain lag orders.
The LSDs under $H_{0}$ are computable by using the inversion formula of Stieltjes transforms
whenever the Stieltjes transform of the LSDs can be computed numerically.
An example of such a setting is given by Example \ref{ex:multi_block}. The actual numerical calculations
of the LSD can be done along the lines of \cite{Wang:Paul:2014}.
The test of whether a time series follows a given MA$(q_0)$ model, with a fixed $q_0$, can be further facilitated
by making use of the observation in Remark \ref{rem:MA_q_kernel} which shows that if the
process is indeed MA$(q_0)$, then the LSDs of the renormalized lag-$\tau$ symmetrized
sample autocovariances will all be the same for $\tau \geq q_0 + 1$.

Calculation of the theoretical LSD under the null model requires inversion of the corresponding
Stieltjes transform, which is somewhat challenging due to the need for selection
of the correct root, as it is necessary to let the imaginary part of the argument of the Stieltjes transform converge
to zero. A simpler alternative is to compute the differences $|s_{\tau,p}(z)-s_\tau(z)|$ between
the Stieltjes transforms of the ESD and the LSDs for a finite, pre-specified set of $z \in \mathbb{C}^+$, and then combine
them through some norm (like $l_\infty$, $l_1$ or $l_2$) and use the latter as a test statistic.
The null distribution of this statistic can be simulated from a Gaussian ensemble, which can then be
used to determine the critical values of the test.

If the linear process $(X_t\colon t\in\mathbb{Z})$ satisfies all the assumptions
of Theorem \ref{thm:lin_proc} and all the coefficient matrices are determined by a finite
dimensional parameter, then under suitable regularity conditions, it may be possible
to estimate that parameter with error rate $O_P(1/\sqrt{n})$
through the use of method of moments or maximum likelihood (under the working assumption of Gaussianity).
Supposing $\theta$ to be the parameter, assuming that $\Sigma(\theta)$ is twice continuously differentiable and $\frac{\partial^2}{\partial\theta\partial \theta^T}\mathbf{\Sigma}_\tau(\theta)$ has uniformly
bounded norm in a neighborhood of the true parameter $\theta_0$,
and denoting any $\sqrt{n}$-consistent estimate by $\hat\theta$,
it can be shown by a simple application of Lemma \ref{lemma:Levy_norm_inequality}
that the ESD of $\sqrt{n/p}(\mathbf{S}_\tau - \mathbf{\Sigma}_\tau(\hat\theta))$ converges in probability
to the same distribution as the LSD of $\sqrt{n/p}(\mathbf{S}_\tau - \mathbf{\Sigma}_\tau(\theta_0))$. Therefore,
the hypothesis testing framework described in the previous paragraphs is applicable
even if the parameter governing the system is estimated at a suitable precision and plugged into
the expressions for the population autocovariances.

Another interesting application is in analyzing the effects of a linear
filter applied to the observed time series. Linear filters are commonly used to extract
signals  from a time series through modulating its spectral characteristics and
also for predicting future observations. Suppose that $W_{t}=\sum_{\ell=0}^{\infty}c_{\ell}X_{t-\ell}$
where $(X_t\colon t\in\mathbb{Z})$ is the MA$(q)$ process defined in Section \ref{sec:main:assu}
and $(c_\ell\colon \ell\in\mathbb{N}_0)$  a sequence of filter coefficients satisfying
$\sum_{\ell=0}^{\infty}|c_{\ell}|<\infty$. Then, the LSDs of the normalized symmetrized
autocovariances of the filtered process $(W_t\colon t\in\mathbb{Z})$ exist and have the
same structure as that of the process $(X_t\colon t\in\mathbb{Z})$, except that in the
description of their Stieltjes transforms (equations (\ref{eq:Stieltjes_LSD_covariance})
and (\ref{eq:Stieltjes_kernel_covariance})), the spectral density $\psi(\mathbf{a}, \nu)$
is replaced by the function $\tilde \psi(\mathbf{a},\nu;\boldsymbol{c})
= |\sum_{\ell=0}^{\infty}c_\ell e^{\ic \ell\nu}|^2 |g(\mathbf{a}, \nu)|^2$.


\section{Proof of Theorem \ref{thm:LSD_variance}}
\label{sec:proof:thm_maq}


The concern of this paper is in the spectral properties of sample autocovariance matrices.
Since spectral properties are unaffected by this change, in all of the proofs the scaling
factor $1/n$ is preferred over $1/(n-\tau)$ for simplicity of exposition.
Throughout this section, it is assumed that the $Z_{jt}$ are complex-valued and the $\mathbf{A}_\ell$
Hermitian matrices. If the $Z_{jt}$ are real-valued and the $\mathbf{A}_\ell$ real, symmetric
matrices, then the arguments need to be modified very slightly, as indicated in Section 11 of
\cite{Liu:Aue:Paul:2015}. The key arguments in the proof of the real valued case remain the same,
since as in the complex valued case, for Gaussian entries, after appropriate orthogonal transformations,
the data matrix can be assumed to have independent Gaussian entries with zero mean and a variance
profile determined by the spectrum of the process. We omit the details due to space constraints.


\subsection{LSDs of $\mathbf{C}_{\tau}$ and $\tilde{\mathbf{C}}_{\tau}$}
\label{proof:thm_maq:lsd_ctau}

Recall that $\mathbf{C}_\tau$ defined in \eqref{normalized_c} is the renormalized
version of the symmetrized autocovariance matrix $\S_\tau$. In this subsection it
is shown that the LSDs of $\mathbf{C}_\tau^U=\mathbf{U}^{*}\mathbf{C}_{\tau}\mathbf{U}$
and $\tilde{\mathbf{C}}_{\tau}$ coincide, where the latter matrix is the renormalized
version of $\tilde{\S}_\tau$ and defined in \eqref{eq:transformed_covariances}.
Observe that the expectation of $\tilde{\mathbf{S}}_\tau$ is the diagonal matrix
$\tilde{\bs\Sigma}_\tau = \mbox{diag}(\tilde{\sigma}_{\tau,1},\ldots,\tilde{\sigma}_{\tau,p})$ given by
\begin{equation}\label{eq:sigma_tau_j}
\tilde{\sigma}_{\tau,j}
= \frac{1}{n} \sum_{t=1}^n \cos(\tau \nu_t) \psi(\bs\alpha_j,\nu_t),
\qquad j=1,\ldots,p.
\end{equation}
Now write $\mathbf{C}_\tau^U=\sqrt{n/p}(\mathbf{U}^{*}\mathbf{S}_{\tau}\mathbf{U}-\mathbf{\Sigma}_{\tau}^{U})$, where $\mathbf{\Sigma}_{\tau}^{U}=\mathbf{U}^{*}\mathbf{\Sigma}_{\tau}\mathbf{U}
=\mbox{diag}(\sum_{\ell=0}^{q-\tau}f_\ell(\bs \alpha_{j})f_{\ell+\tau}(\bs \alpha_{j}))_{j=1}^{p}$, and
define $\mathbf{C}_\tau^{(1)}=\sqrt{n/p}(\mathbf{U}^*\S_\tau\mathbf{U}-\tilde{\bs\Sigma}_\tau)$.

We first show that $\mathbf{\Sigma}_{\tau}^{U} = \tilde{\mathbf{\Sigma}}_\tau$, which implies equality of the
ESDs of $\mathbf{C}_\tau^U$ and $\mathbf{C}_\tau^{(1)}$.
For each $j=1,\ldots,p$,
\begin{eqnarray}\label{eq:Approximation_error_integral_summation}
\tilde\sigma_{\tau,j} &=&
\frac{1}{n}\sum_{t=1}^{n}\cos(\tau\nu_{t})\psi(\bs{\alpha}_{j},\nu_{t}) \nonumber\\
&=& \frac{1}{n}\sum_{t=1}^{n}\cos(\tau\nu_{t})\sum_{\ell,\ell'=0}^{q}
f_\ell(\bs \alpha_{j})f_{\ell'}(\bs \alpha_{j})e^{\mathbf{i} (\ell-\ell')\nu_t} \nonumber\\
&=& \frac{1}{2n}\sum_{\ell,\ell'=0}^{q}f_\ell(\bs \alpha_{j})f_{\ell'}(\bs \alpha_{j})
\left(\sum_{t=1}^{n} e^{\mathbf{i} (\ell-\ell'+\tau)\nu_t}
+ \sum_{t=1}^{n} e^{\mathbf{i} (\ell-\ell'-\tau)\nu_t}\right)
~=~ \sum_{\ell=0}^{q-\tau} f_\ell(\bs \alpha_{j})f_{\ell+\tau}(\bs \alpha_{j}),
\end{eqnarray}
since $\sum_{t=1}^{n} e^{\mathbf{i} k\nu_t} = n \delta_0(k)$ for $k=0,1,\ldots,n-1$
where $\delta_0$ denotes the Kronecker's delta function. This proves the assertion.

\begin{lemma}
\label{lem:aux:2}
If the conditions of Theorem \ref{thm:LSD_variance} are satisfied, then
$\|F^{\mathbf{C}_\tau^{U}}-F^{\tilde{\C}_\tau}\|\to 0$ almost surely under \eqref{eq:mod_dim},
where $F^{\mathbf{C}_\tau^{U}}$ and $F^{\tilde{\C}_\tau}$ denote the ESDs
of $\mathbf{C}_\tau^{U}$ and $\tilde{\C}_\tau$, respectively, and $\parallel \cdot \parallel$ denotes the sup-norm.
\end{lemma}
\begin{proof}
Recall that $\tilde{\mathbf{C}}_{\tau}=\sqrt{n/p}(\mathbf{U}^{*}(\bar{\mathbf{S}}_{\tau}-\tilde{\bs\Sigma}_{\tau})\mathbf{U})$.
Exploiting the relation between $\mathbf{L}$ and $\tilde{\mathbf{L}}$, it can be shown that $\tilde{\mathbf{S}}_\tau=\mathbf{U}^*\bar{\mathbf{S}}_\tau\mathbf{U}$ can be written as at most $4(q+\tau+1)$ rank-one perturbations of $\mathbf{S}_\tau$. Hence, an application of the rank inequality given in Lemma \ref{lemma:rank_inequality_rectangular} implies that
\begin{equation}
\|F^{\mathbf{C}_\tau^{(1)}}-F^{\tilde{\mathbf{C}}_{\tau}}\|
\leq \frac{1}{p}\mbox{rank}(\tilde{\mathbf{S}}_{\tau}-\mathbf{S}_{\tau})\leq \frac{4(q+\tau+1)}{p}
\to 0
\end{equation}
under \eqref{eq:mod_dim}, which is the assertion since $F^{\mathbf{C}_\tau^{(1)}} = F^{\mathbf{C}_\tau^{U}}$.
\end{proof}

Define the Stieltjes transforms $s_{\tau,p}^U(z)=p^{-1}\tr(\mathbf{C}_{\tau}^{U}-zI)^{-1}$ and $\tilde{s}_{\tau,p}(z)=p^{-1}\tr(\tilde{\mathbf{C}}_{\tau}-zI)^{-1}$. Repeatedly
applying Lemma \ref{lemma:rank_one_inverse} to each of the rank-one perturbation matrices
used in the proof of Lemma \ref{lem:aux:2}, it follows that, for any fixed $z=w+\mathbf{i}v\in \mC^{+}$, $|s^{U}_{p,\tau}(z)-\tilde{s}_{p,\tau}(z)|\leq 4(q+\tau+1)/(vp)$ almost surely.
It is therefore verified that the LSDs of $\mathbf{C}_\tau^U$ and $\tilde{\mathbf{C}}_\tau$ are almost surely identical.


\subsection{Deterministic equation}
\label{proof:thm_maq:det_eq}

In this section a set of deterministic equations is derived that is asymptotically equivalent to the
set of equations determining the Stieltjes transform of the limiting ESD of $\tilde{\mathbf{C}}_\tau$.
The following decomposition will be useful in the proofs. Using assumptions {\bf A1} and {\bf A2} in
combination with \eqref{L_decomposition} and some matrix algebra, it can be shown that
\[
\mathbf{\tilde{S}}_{\tau}
=\mathbf{U}^*\bar{\mathbf{S}}_\tau\mathbf{U}
=\mathbf{V} {\bs\Delta}_\tau\mathbf{V}^*,
\]
where the $p\times n$ matrix $\mathbf{V}$ is defined through its entries
\begin{equation}
\label{eq:transformed_data_matrix}
v_{jt}
= \frac{1}{\sqrt{n}} g(\bs\alpha_j,\nu_t) \tilde Z_{jt},
\qquad j=1,\ldots,p,~t=1,\ldots,n,
\end{equation}
and ${\bs\Delta}_\tau = \mbox{diag}(\cos(\tau\nu_1),\ldots,\cos(\tau\nu_n))$. Let $\mathbf{V}_k$
denote the matrix obtained by replacing the $k$th row of $\mathbf{V}$ with zeros, and let the
$n\times 1$ vector $v_k$ be the $k$th column of the matrix $\mathbf{V}^{*}=(v_{1}: v_{2}:\cdots: v_{p})$.
Let further $\tilde{\bs\Sigma}_{\tau,k}$ be the matrix obtained from $\tilde{\bs\Sigma}_\tau$
by replacing its $k$th diagonal entry with $0$. Denote by $\mathbf{D}_k$, respectively, $\mathbf{D}_{(k)}$
the matrices resulting from $\tilde{\mathbf{C}}_\tau$ from replacing the entries of its $k$th row,
respectively, its $k$th row and $k$th column with zeros, that is,
\[
\mathbf{D}_k
=\sqrt{\frac np}(\mathbf{V}_k\mathbf{\Delta}_\tau\mathbf{V}^*-\tilde{\mathbf{\Sigma}}_{\tau,k})
\qquad\mbox{and}\qquad
\mathbf{D}_{(k)}
=\sqrt{\frac np}(\mathbf{V}_k\mathbf{\Delta}_\tau\mathbf{V}_k^*-\tilde{\mathbf{\Sigma}}_{\tau,k}).
\]
Then,
\begin{equation}
\label{eq:covariance_second_decomp}
\tilde{\mathbf{C}}_\tau
=\mathbf{D}_k+\mathbf{H}_k
=\mathbf{D}_{(k)}+\mathbf{H}_{(k)},
\end{equation}
where $\mathbf{H}_k=e_kh_k^*$ and $\mathbf{H}_{(k)}=\mathbf{H}_k+w_{k}e_k^T$ with $e_k$ being the $k$th canonical unit vector of dimension $p$, $h_k = w_k + \eta_ke_k$,
\begin{equation}
\label{eq:h_k_def}
w_k = \sqrt{\frac np} \mathbf{V}_k\mathbf{\Delta}_\tau v_k
\qquad\mbox{and}\qquad
\eta_k = \sqrt{\frac np} (v_k^*\mathbf{\Delta}_\tau v_k-\tilde{\sigma}_{\tau,k}),
\end{equation}
where $\tilde{\sigma}_{\tau,j}$ is defined in (\ref{eq:sigma_tau_j}),
thereby ensuring that the $k$th entry of $w_k$ is zero and collecting the $k$th diagonal element of
$\tilde{\mathbf{C}}_\tau$ in the term $\eta_k$. Successively replacing rows of $\tilde{\mathbf{C}}_\tau$ with
rows of zeros and noticing that $\tilde{\mathbf{C}}_\tau=\tilde{\mathbf{C}}_\tau^*$ as well as
$\mathbf{H}_k^*=(e_kh_k^*)^*=h_ke_k^T$, the same arguments also yield
\begin{equation}
\label{eq:covariance_rowwise_decomp}
\tilde{\mathbf{C}}_\tau
= \sum_{k=1}^p e_k h_k^*
= \sum_{k=1}^p h_k e_k^T.
\end{equation}

Observe next that, since its $k$th row and column consist of zero entries, $e_k$ is an eigenvector of
$\mathbf{D}_{(k)}$ with eigenvalue $0$. If now, for $z\in\mathbb{C}^+$, $\mathbf{R}_{(k)}(z) =(\mathbf{D}_{(k)}- z \mathbf{I}_p)^{-1}$ denotes the resolvent of $\mathbf{D}_{(k)}$, then
\begin{equation}
\label{eq:resolvent_D_k_eigenvector}
\mathbf{R}_{(k)}(z) e_k
= - \frac{1}{z} e_k,
\end{equation}
that is, $e_k$ is an eigenvector of $\mathbf{R}_{(k)}(z)$ with eigenvalue $-z^{-1}$. Let $\mathbf{R}_k(z) = (\mathbf{D}_{k}- z \mathbf{I}_p)^{-1}$ be the resolvent of $\mathbf{D}_k$. Utilizing (\ref{eq:covariance_second_decomp}) and Lemma \ref{lemma:rank_one_inverse}, it follows that
\[
\mathbf{R}_k(z) e_k
= \mathbf{R}_{(k)}(z) e_k
- \frac{\mathbf{R}_{(k)}(z) w_k e_k^T \mathbf{R}_{(k)}(z) e_k}{1 + e_k^T \mathbf{R}_{(k)}(z) w_k}
= - \frac{1}{z} e_k + \frac{1}{z} \mathbf{R}_{(k)}(z) w_k,
\]
where the second step follows from invoking (\ref{eq:resolvent_D_k_eigenvector}), for the denominator part in the middle expression additionally noticing that $\mathbf{R}_{(k)}(z)=\mathbf{R}_{(k)}^*(z)$ and that $e_k^Tw_k=0$ by construction. Now, all preliminary statements are collected that allow for a detailed study the resolvent and the Stieltjes transform of $\tilde{\mathbf{C}}_\tau$.
\begin{lemma}
\label{lem:aux:3}
Under the assumptions of Theorem~\ref{thm:LSD_variance}, it follows that the
Stieltjes transform $\tilde{s}_{\tau,p}$ of\/ $\tilde{\mathbf{C}}_\tau$ satisfies the equality
\[
\tilde{s}_{\tau,p}(z)
=-\frac 1p\sum_{k=1}^p\frac{1}{z+w_k^*\mathbf{R}_{(k)}(z)w_k-\eta_k}
\]
for any fixed $z\in\mathbb{C}^+$.
\end{lemma}
\begin{proof}
Writing $\mathbf{I}_p + z (\tilde{\mathbf{C}}_\tau - z\mathbf{I}_p)^{-1} = (\tilde{\mathbf{C}}_\tau - z\mathbf{I}_p)^{-1} \tilde{\mathbf{C}}_\tau$, invoking (\ref{eq:covariance_rowwise_decomp}) and Lemma \ref{lemma:rank_one_inverse} implies that
\begin{align}
\label{eq:resolvent_covariance_exact_decomp}
\mathbf{I}_p + z (\tilde{\mathbf{C}}_\tau - z\mathbf{I}_p)^{-1}
&= \sum_{k=1}^p (\tilde{\mathbf{C}}_\tau - z\mathbf{I}_p)^{-1} e_k h_k^* \nonumber\\
&= \sum_{k=1}^p \mathbf{R}_k(z)e_k \left(
1 - \frac{h_k^* \mathbf{R}_k(z) e_k}{1 + h_k^* \mathbf{R}_k(z) e_k}\right) h_k^* \nonumber\\
&= \sum_{k=1}^p \frac{ \mathbf{R}_k(z) e_k h_k^*}{1 + h_k^* \mathbf{R}_k(z) e_k}.
\end{align}
Recall that the Stieltjes transform of $\tilde{\mathbf{C}}_\tau$ is given by $p^{-1}\tr((\tilde{\mathbf{C}}_\tau-z I_p)^{-1})$. Therefore, taking trace on both sides of (\ref{eq:resolvent_covariance_exact_decomp}) and dividing by $p$ leads to
\begin{equation}
\label{eq:estimation_s_n}
\tilde s_{\tau,p}(z)
= \frac{1}{zp}\sum_{k=1}^{p}\left(\frac{h_k^*\mathbf{R}_k(z)e_k}{1+h_k^*\mathbf{R}_k(z)e_k}-1\right)
=-\frac{1}{zp}\sum_{k=1}^p\frac{1}{1+h_{k}^{*}\mathbf{R}_{k}(z) e_k}.
\end{equation}
In order to complete the proof of the lemma, it remains to study $h_k^*\mathbf{R}_k(z) e_k$.  Using Lemma \ref{lemma:rank_one_inverse} on $\mathbf{R}_{k}(z)$ and subsequently first utilizing (\ref{eq:resolvent_D_k_eigenvector}) and then inserting the definition of $w_k$ given in (\ref{eq:h_k_def}), it follows that
\begin{align}
\label{eq:estimation_on_h_k_R_k_e_k}
h_{k}^{*}\mathbf{R}_{k}(z)e_k
&= h_{k}^{*}\mathbf{R}_{(k)}(z)e_k
-h_k^*\frac{\mathbf{R}_{(k)}(z) w_ke_k^T\mathbf{R}_{(k)}(z)e_k}{1+ e_k^T\mathbf{R}_{(k)}(z) w_k} \nonumber \\
&= -\frac{1}{z} h_k^*e_k + \frac{1}{z} h_k^*\mathbf{R}_{(k)}(z)w_k \nonumber \\
&= -\frac{1}{z}\eta_k+\frac{1}{z} w_k^*\mathbf{R}_{(k)}(z)w_k,
\end{align}
where the third step also makes use of $e_{k}^{T}w_{k}=0$. Plugging \eqref{eq:estimation_on_h_k_R_k_e_k} into \eqref{eq:estimation_s_n} finishes the proof.
\end{proof}
In the next auxiliary lemma, the expected value of the Stieltjes transform of $\tilde{\mathbf{C}}_\tau$ is determined.
More generally, equations for the kernel
\begin{equation}
\label{eq:definition_beta_tau}
\tilde\beta_{\tau, p}(z, \mathbf{a})
=\frac{1}{p}\tr((\tilde{\mathbf{C}}_\tau -zI_p)^{-1} \mathbf{\Gamma}_\tau(\mathbf{a}))
\end{equation}
are introduced, where $\mathbf{\Gamma}_\tau(\mathbf{a})=\mathrm{diag}(\mathcal{R}_\tau(\mathbf{a},\bs{\alpha}_k)\colon k=1,\ldots,p)$ with $\mathcal{R}_\tau(\mathbf{a},\bs{\alpha}_k)$ defined in \eqref{eq:spectral_kernel_covariance}. It is a central object of this study and the (approximate) finite-sample companion of the Stieltjes kernel $\beta_\tau(z,\mathbf{a})$ appearing in the statement of Theorem~\ref{thm:LSD_variance}. Its properties will be further scrutinized in Sections~\ref{sec:conv_random_part} and \ref{sec:ex_conv_cont}.
\begin{lemma}
\label{lem:aux:4}
Under the assumptions of Theorem~\ref{thm:LSD_variance}, it follows that the expected value of the Stieltjes transform $\tilde{s}_{\tau,p}$ of\/ $\tilde{\mathbf{C}}_\tau$ satisfies the equality
\begin{equation}
\label{eq:expression_on_expectation_s_tau_n}
\mathbb{E}[\tilde{s}_{\tau,p}(z)]
=-\frac 1p\sum_{k=1}^p\frac{1}{z+\mathbb{E}[\tilde\beta_{\tau,p}(z,\bs{\alpha}_k)]}+\tilde\delta_n
\end{equation}
for any fixed $z\in\mathbb{C}^+$, where the remainder term $\tilde\delta_n$ converges to zero under \eqref{eq:mod_dim}. Moreover,
\begin{equation}
\label{eq:expression_on_expectation_beta_tau_p}
\mE[\tilde\beta_{\tau,p} (z, \mathbf{a})]
=-\frac{1}{p}
\sum_{k=1}^p\frac{\mathcal{R}_{\tau}(\mathbf{a},\bs\alpha_{k})}{z+\mE[\tilde\beta_{\tau,p} (z, \bs{\alpha}_{k})]}
+\delta_{n}
\end{equation}
for any fixed $z\in\mathbb{C}^+$, where the remainder term $\delta_n$ converges to zero under \eqref{eq:mod_dim}.
\end{lemma}
\begin{proof}
The proof of the lemma is given in three parts. In view of the expression for $\tilde{s}_{\tau,p}$
derived in Lemma~\ref{lem:aux:3}, $\mathbb{E}[w_k^*\mathbf{R}_{(k)}(z)w_k]$ is estimated first and
in the second step related to $\tilde\beta_{\tau,p}(z,\mathbf{a})$. The third step is concerned
with the estimation of remainder terms $\delta_n$ and $\tilde\delta_n$.

{\em Step 1:} For $k=1,\ldots,p$, let $\bs\Sigma_{k,v} = \mathrm{Var}(v_k)=n^{-1}\mathrm{diag}(\psi(\bs\alpha_k,\nu_t)\colon t=1,\ldots,n)$ and further $\bs\Xi_{\tau,k} = \mathbf{\Delta}_\tau \bs\Sigma_{k,v} \mathbf{\Delta}_\tau=n^{-1}\mathrm{diag}(\cos^2(\tau\nu_t)\psi(\bs\alpha_k,\nu_t)\colon t=1,\ldots,n)$.
Define
\[
\gamma_{\tau,j}(\bs{a}) := \frac{1}{n}
\sum_{t=1}^n\cos^2(\tau\nu_t)\psi(\bs{a},\nu_t)\psi(\bs\alpha_j,\nu_t),
\]
and observe that $\gamma_{\tau,j}(\bs{a}) = {\cal R}(\bs{a},\bs\alpha_j)$ for all $j=1,\ldots,p$. This follows
calculations similar to those leading to \ref{eq:Approximation_error_integral_summation}. Define
the matrix $\Gamma_{\tau,k}(\bs{a})$ as the one obtained from $\Gamma_{\tau}(\bs{a})$
by replacing its $k$th diagonal entry with zero. Observe next that the definition of $w_k$
in \eqref{eq:h_k_def} implies that it suffices to estimate the following expectation, for which it holds that
\begin{eqnarray}
\label{eq:w_k_R_k_quad_expectation}
&& \frac{n}{p}
\mE \left[v_{k}^{*}\mathbf{\Delta}_{\tau}\mathbf{V}_{k}^{*}\mathbf{R}_{(k)}(z)\mathbf{V}_{k}\Delta_{\tau}v_k\right]
~=~ \frac{n}{p}\mE\left[\tr(\mathbf{\Delta}_{\tau}v_kv_k^*\mathbf{\Delta}_{\tau}
\mathbf{V}_{k}^*\mathbf{R}_{(k)}(z)\mathbf{V}_{k})\right] \nonumber \\
&=&  \frac{n}{p}\tr\left(\mathbf{\Delta}_{\tau}\bs\Sigma_k\mathbf{\Delta}_{\tau}
\mE[\mathbf{V}_{k}^*\mathbf{R}_{(k)}(z)\mathbf{V}_{k}]\right)
~=~ \frac{n}{p}\mE\left[\tr(\mathbf{V}_{k} \bs\Xi_{\tau,k}\mathbf{V}_{k}^* \mathbf{R}_{(k)}(z))\right]\nonumber\\
&=& \frac{n}{p}\sum_{j \neq k} \mE\left[v_j^*\bs\Xi_{\tau,k} v_j (\mathbf{R}_{(k)}(z))_{jj}\right]
~=~ \frac{1}{p} \sum_{j \neq k} \mE \left[\gamma_{\tau,j}(\bs\alpha_k) (\mathbf{R}_{(k)}(z))_{jj}\right] + d_k^{(0)}\nonumber\\
&=& \frac{1}{p}\mE\left[ \tr(\mathbf{R}_{(k)}(z)\mathbf{\Gamma}_{\tau, k}(\bs{\alpha}_{k})) \right]+ d_k^{(0)},
\end{eqnarray}
where independence between $v_k$ and $\mathbf{V}_k$ was used to obtain the second equality and
\begin{equation}\label{eq:d_k_0}
d_k^{(0)}  = \frac{1}{p} \sum_{j \neq k}
\mE\left[(n v_j^* \bs\Xi_{\tau,k} v_j - \gamma_{\tau,j}(\bs\alpha_k))(\mathbf{R}_{(k)}(z))_{jj}\right].
\end{equation}
An application of the Cauchy--Schwarz inequality to the expectation on the right-hand side of \eqref{eq:d_k_0},
subsequently using the fact that $\max_j |(\mathbf{R}_{(k)}(z))_{jj}|\leq \Im(z)^{-1}$ and squaring the resulting estimate, yields that
\[
|d_k^{(0)}|^2
\leq\frac{1}{p\,\Im(z)^2}
\sum_{j \neq k} \mE\left[\left| nv_j^* \bs\Xi_{\tau,k} v_j - \gamma_{\tau,j}(\bs\alpha_k)\right|^2\right]
=\frac{1}{p\,\Im(z)^2}\sum_{j \neq k} \mathrm{Var}\left(nv_j^*\bs\Xi_{\tau,k} v_j\right)
\leq \frac{C^2}{p\,\Im(z)^{2}},
\]
where the equality follows from recognizing that $\mE[nv_j^* \bs\Xi_{\tau,k} v_j]  = \gamma_{\tau,j}(\bs\alpha_k)$ and the inequality from observing that each $nv_j^*\bs\Xi_{\tau,k}v_j$ is a quadratic form in the i.i.d. standard Gaussians $\tilde Z_{j1},\ldots\tilde Z_{jn}$ and has bounded variance. Taking the square root gives
\begin{equation}
\label{eq:d_k_0_bound}
|d_k^{(0)}| \leq \frac{C}{\sqrt{p}\,\Im(z)}
\end{equation}
for some constant $C > 0$.

{\em Step 2:}
Multiplying $\mathbf{\Gamma}_\tau(\mathbf{a})$ to both sides of the equation $\mathbf{I}_p+z(\tilde{\mathbf{C}}_\tau-z\mathbf{I}_p)^{-1}=\tilde{\mathbf{C}}_\tau(\tilde{\mathbf{C}}_\tau-z\mathbf{I}_p)^{-1}$,
then following the arguments that led to \eqref{eq:resolvent_covariance_exact_decomp}, and making use of $\mathbf{\Gamma}_\tau(\mathbf{a})e_k=\mathcal{R}(\mathbf{a},\bs{\alpha}_k)e_k$ gives
\[
\mathbf{\Gamma}_{\tau}(\mathbf{a})+ z\mathbf{\Gamma}_\tau(\mathbf{a})(\tilde{\mathbf{C}}_\tau -z\mathbf{I}_p)^{-1}
=\sum_{k=1}^{p}\mathcal{R}_{\tau}(\mathbf{a},\bs\alpha_{k})e_kh_k^*(\tilde{\mathbf{C}}_\tau -z\mathbf{I}_p)^{-1}
=\sum_{k=1}^{p}\frac{\mathcal{R}_{\tau}(\mathbf{a},\bs\alpha_{k})e_kh_k^*\mathbf{R}_{k}(z)}{1+h_k^*\mathbf{R}_{k}(z)e_k}.
\]
Further taking trace on both sides and invoking (\ref{eq:estimation_on_h_k_R_k_e_k}) yields
\begin{align}
\nonumber
\tilde\beta_{\tau,p}( z, \mathbf{a})
&= -\frac{1}{p}
\sum_{k=1}^p\frac{\mathcal{R}_{\tau}(\mathbf{a},\bs\alpha_{k})}{z+w_k^*\mathbf{R}_{(k)}(z)w_k-\eta_{k}} \\
&=-\frac{1}{p}
\sum_{k=1}^p\frac{\mathcal{R}_{\tau}(\mathbf{a},\bs\alpha_{k})}{z+\mE[\tilde\beta_{\tau,p}(z, \bs{\alpha}_{k})]-\epsilon_k},
\label{eq:expression_on_beta}
\end{align}
where $\epsilon_{k}=\mE[\tilde\beta_{\tau,p}(z, \bs{\alpha}_{k})] - w_k^*\mathbf{R}_{(k)}(z)w_k+\eta_{k}$.
Taking expectation on the left- and right-hand side of \eqref{eq:expression_on_beta} leads to equation
\eqref{eq:expression_on_expectation_beta_tau_p} with the remainder term having the explicit form
\[
\delta_{n}
=-\frac{1}{p}\sum_{k=1}^{p}
\frac{\mathcal{R}_{\tau}(\mathbf{a},\bs\alpha_{k})\mE[\epsilon_{k}]}{(z+\mE \tilde\beta_{\tau,p} (z,\bs \alpha_{k}))^2}
-\frac{1}{p}\sum_{k=1}^{p}\mE\left(\frac{\mathcal{R}_{\tau}(\mathbf{a},\bs\alpha_{k})\epsilon_{k}^{2}}
{(z+\mE[\tilde\beta_{\tau,p} (z, \bs \alpha_{k})])^2(z+\mE[\tilde\beta_{\tau,p} (z, \bs{\alpha}_{k})]-\epsilon_{k})}\right)
= \delta_{n,1}+\delta_{n,2}.
\]
It remains to show that $\delta_n\to 0$ under \eqref{eq:mod_dim}. This will be done in the next step.


{\em Step 3:} To show that $\delta_{n}\to 0$, it suffices to verify that
$\delta_{n,1}\to 0$ and $\delta_{n,2}\to 0$. Note that, since $\tilde\beta_{\tau,p}(z, \bs\alpha_{k})$
is a Stieltjes transform of a measure,
\[
\left| z+\mE [\tilde\beta_{\tau,p}(z, \bs\alpha_{k})]\right|
\geq \Im\left(z+\mE [\tilde\beta_{\tau,p}(z, \bs\alpha_{k})]\right)
\geq \Im(z)+\mE \left[\Im\left(\tilde\beta_{\tau,p}(z, \bs\alpha_{k})\right)\right]
\geq \Im(z)
\]
and since $\eta_k \in \mathbb{R}$, and $\mathbf{w}_{k}^{*}\mathbf{R}_{(k)}(z)\mathbf{w}_{k}$ is a
Stieltjes transform of a measure,
\[
\left| z+\mE [\tilde\beta_{\tau,p}(z, \bs\alpha_{k})]-\epsilon_{k} \right|
= \left|z+\mathbf{w}_{k}^{*}\mathbf{R}_{(k)}(z)\mathbf{w}_{k}-\eta_{k}\right|
\geq \Im(z) + \Im(\mathbf{w}_{k}^{*}\mathbf{R}_{(k)}(z)\mathbf{w}_{k})
\geq \Im(z).
\]
Thus, since moreover $|\mathcal{R}_{\tau}(\mathbf{a},\mathbf{b})|\leq L^2_{1}$ with $L_1$ from {\bf A5},
it only needs to be shown that $\max_k |\mE[\epsilon_{k}]|\to 0$ and $\max_k \mE[|\epsilon_{k} - \mE[\epsilon_{k}]|^2] \to 0$.


Let $\tilde{\mathbf{R}}(z)=(\tilde{\mathbf{C}}_{\tau}-z\mathbf{I})^{-1}$. Since $\mE[\eta_{k}]=0$,
it follows from (\ref{eq:w_k_R_k_quad_expectation}) and (\ref{eq:definition_beta_tau}) that
\begin{align}
\label{eq:estimation_expectation_epsilon_k}
|\mE[\epsilon_{k}]|
&=\left|\frac{1}{p} \mE[\tr(\tilde{\mathbf{R}}(z)\mathbf{\Gamma}_{\tau}(\bs\alpha_k))]
-\frac 1p \mE [\tr(\mathbf{R}_{(k)}(z)\mathbf{\Gamma}_{\tau,k}(\bs\alpha_k))]
- d_k^{(0)}\right| \nonumber \\
&\leq \frac 1p \left|\mE[\tr(\tilde{\mathbf{R}}(z)\mathbf{\Gamma}_{\tau}(\bs\alpha_k))]
-\mE[\tr(\mathbf{R}_{(k)}(z)\mathbf{\Gamma}_{\tau}(\bs\alpha_k))]\right| \nonumber \\
& \qquad + \frac 1p \left|\mE[\tr(\mathbf{R}_{(k)}(z)\{\mathbf{\Gamma}_{\tau}(\bs\alpha_k)
-\mathbf{\Gamma}_{\tau,k}(\bs\alpha_k)\})]\right|
+ |d_k^{(0)}|\nonumber\\
&= d_k^{1,1}+ d_k^{1,2} 
+ |d_k^{(0)}|,
\end{align}
where $\mathbf{\Gamma}_{\tau,k}(\bs\alpha_k)=\mathbf{\Gamma}_\tau(\bs\alpha_k)-\mathcal{R}_\tau(\bs\alpha_k,\bs\alpha_k)e_ke_k^T$.
Arguments as the more general ones leading to (\ref{eq:bound_trace_resolvent_H}), imply that $\max_k d_k^{1,1} \leq 6qL_1^2(p\Im(z))^{-1}$.
Since $\|\mathbf{R}_{(k)}(z) \| \leq (\Im(z))^{-1}$ and ${\mathcal R}_\tau(\bs\alpha_k,\bs\alpha_k)$ is uniformly bounded,
it follows that $\max_k d_k^{1,2} \leq  L_1^2(p\Im(z))^{-1}$.
Together with (\ref{eq:d_k_0_bound}) and
(\ref{eq:estimation_expectation_epsilon_k}), these guarantee that $\max_k |\mE[\epsilon_{k}]|\to 0$ and thus $|\delta_{n,1}|\leq L_1^2(\Im(z))^{-2}\max_k|\mE[\epsilon_{k}]|\to 0$.

Observe next that, by (\ref{eq:w_k_R_k_quad_expectation}),
\begin{align*}
\mE\left[|\epsilon_{k}-\mE [\epsilon_{k}]|^{2}\right]
&=\mE\Big[\Big|-w_k^*\mathbf{R}_{(k)}(z)w_k
+ \frac{1}{p} \mE [\tr(\mathbf{R}_{(k)}(z)\mathbf{\Gamma}_{\tau,k}(\bs \alpha_{k}))]
+ d_k^{(0)} +\eta_{k} \Big|^{2}\Big] \nonumber \\
&\leq 3\mE\Big[\Big|-w_k^*\mathbf{R}_{(k)}(z)w_k
+ \frac{1}{p} \tr(\mathbf{R}_{(k)}(z)\mathbf{\Gamma}_{\tau,k}(\bs \alpha_{k})) + \eta_{k}\Big|^{2}\Big] \nonumber \\
& \qquad + 3\mE\Big[\Big|\frac{1}{p} \tr(\mathbf{R}_{(k)}(z)\mathbf{\Gamma}_{\tau,k}(\bs \alpha_{k}))
- \mE \Big[\frac{1}{p} \tr(\mathbf{R}_{(k)}(z)\mathbf{\Gamma}_{\tau,k}(\bs \alpha_{k}))\Big]\Big|^2\Big]
+ 3|d_k^{(0)}|^2 \nonumber \\
&= d_k^{2,1}+d_k^{2,2}+3|d_k^{(0)}|^2,
\end{align*}
where
\begin{align*}
d_k^{2,1}
&\leq 6\mE\Big[\Big|-w_k^*\mathbf{R}_{(k)}(z)w_k+\frac{1}{p} \tr(\mathbf{R}_{(k)}(z)
\mathbf{\Gamma}_{\tau,k}(\bs \alpha_{k}))\Big|^2\Big] + 6\mE[|\eta_{k}|^2] \\
&= 6 d_k^{2,3}+6\mE[|\eta_{k}|^2].
\end{align*}
Now, $\max_k \mE[|\eta_{k}|^2]<Cp^{-1}$ for some $C > 0$ as proved in Section \ref{subsec:aux:deterministic:1}. It is shown in Sections \ref{subsec:estimation_on_d_4} and \ref{subsec:bound_on_d_31} that $\max_k d_k^{2,2}\to 0$ and $\max_k d_k^{2,3} \to 0$, respectively. Consequently, $\max_k \mE[|\epsilon_k - \mE[\epsilon_k]|^2] \to 0$ and hence also $\delta_{n,2}\to 0$.

{\em Step 4:} Using the expression for $\tilde s_{\tau,p}(z)$ derived in Lemma \ref{lem:aux:3}, relation
\eqref{eq:expression_on_expectation_s_tau_n} can be obtained from similar arguments as in Steps 1--3 of this proof.
In particular, it can be shown that $\tilde\delta_{n}\to 0$.
\end{proof}


\subsection{Convergence of random part}
\label{sec:conv_random_part}

In this section, it is shown that, almost surely $s_{\tau,p}(z)-\mE [s_{\tau,p}(z)]\to 0$ and
$\beta_{\tau, p}(z, \mathbf{a})-\mE[\beta_{\tau, p}(z,\mathbf{a})]\to 0$ for any $z\in \mathbb{C}^{+}$
when the entries of $\mathbf{Z}$ are i.i.d. standardized random variables with arbitrary distributions.
The concentration inequalities on $s_{\tau,p}(z)$ and $\beta_{\tau, p}(z, \mathbf{a})$ are derived
by using the McDiarmid's inequality given in Lemma \ref{lemma:mcdiarmid} and the proof of almost
sure convergence is obtained through the use of the Borel--Cantelli lemma. To apply the McDiarmid
inequality, treat $\mathbf{C}_{\tau}$ as a function of the independent rows of $\mathbf{Z}$, say,
$\mathbf{z}_1^*,\ldots,\mathbf{z}_p^*$. Let
\[
\mathbf{Z}_{(j)}
= \mathbf{Z} - e_je_j^T \mathbf{Z}
= \mathbf{Z} - e_j\mathbf{z}_j^*,
\qquad j=1,\ldots,p,
\]
where $\mathbf{Z}=[\mathbf{z}_1^*:\cdots:\mathbf{z}_p^*]^*$. Let further $\mathbf{X}_{(j)}$ be the
$p\times n$ matrix obtained from the original data matrix $\mathbf{X}$ with the $j$th row removed, that is,
\[
\mathbf{X}_{(j)}
= \sum_{\ell=0}^{q}\mathbf{A}_{\ell}\mathbf{Z}_{(j)}\mathbf{L}^{\ell}.
\]
Define $\mathbf{S}_{\tau}^{(j)}=n^{-1}\mathbf{X}_{(j)}\mathbf{D}_\tau\mathbf{X}_{(j)}^*$ and $\mathbf{C}_{\tau}^{(j)}=\sqrt{n/p} (\mathbf{S}_{\tau}^{(j)}-\mathbf{\Sigma}_{\tau})$, where $\mathbf{D}_\tau=[\mathbf{L}^\tau+(\mathbf{L}^\tau)^*]/2$.
It follows then from the relation
\begin{align*}
\mathbf{S}_{\tau}
&=\frac{1}{n}\bigg(\sum_{\ell=0}^{q}\mathbf{A}_{\ell}(\Z_{(j)}+e_j\mathbf{z}_j^*)\mathbf{L}^{\ell}\bigg)
\mathbf{D}_\tau \bigg(\sum_{\ell=0}^{q}\mathbf{A}_{\ell}(\Z_{(j)}+e_j\mathbf{z}_j^*)\mathbf{L}^{\ell}\bigg)^* \\
&=\mathbf{S}_{\tau}^{(j)}+\frac{1}{n}\bigg(
\sum_{\ell=0}^{q}a_{j\ell}y_{j\ell}^{*}\mathbf{D}_\tau\mathbf{X}_{(j)}^{*}
+\sum_{\ell=0}^{q}\mathbf{X}_{(j)}\mathbf{D}_\tau y_{j\ell}a_{j\ell}^{*}
+\sum_{\ell,\ell^{\prime}=0}^{q}a_{j\ell}y_{j\ell}^{*}\mathbf{D}_\tau y_{j\ell^{\prime}}a_{j\ell^{\prime}}^{*}\bigg),
\end{align*}
where $a_{j\ell}=\mathbf{A}_{\ell}e_j$, $y_{j\ell}^{*}=\mathbf{z}_j^*\mathbf{L}^{\ell}$, that
\begin{equation}
\label{eq:expression_relation_C_tau_and_C_tau_i}
\mathbf{C}_{\tau}
=\mathbf{C}_{\tau}^{(j)}
+\sum_{\ell=0}^{q}a_{j\ell}{\zeta}_{j\ell}^{*}
+\sum_{\ell=0}^{q}\zeta_{j\ell}a_{j\ell}^{*}
+\sum_{\ell,\ell^{\prime}=0}^{q}\omega_{\ell,\ell^{\prime}}^{j}a_{j\ell}a_{j\ell^{\prime}}^{*},
\end{equation}
making use of the notations $\zeta_{j\ell}=(np)^{-1/2}y_{j\ell}^{*}\mathbf{\Delta}\mathbf{X}_{(j)}^{*}$ and $\omega_{\ell,\ell^{\prime}}^j=(pn)^{-1/2}y_{j\ell}^{*}\mathbf{\Delta}y_{j\ell^{\prime}}$.
The following lemma will be instrumental in determining the convergence of the random part.

\begin{lemma}
\label{lem:aux:5}
Under the assumptions of Theorem \ref{thm:LSD_variance}, it follows that
\[
\mathrm{diff}_{\tau,j}(\mathbf{H})
=\frac{1}{p}\left|\tr\big((\mathbf{C}_{\tau}-z\mathbf{I})^{-1}\mathbf{H}\big)
-\frac{1}{p}\tr\big((\mathbf{C}^{(j)}_{\tau}-z\mathbf{I})^{-1}\mathbf{H}\big)\right|
\leq \frac{3(q+1)\|\mathbf{H}\|}{p\,\Im(z)},
\]
where $\mathbf{H}$ is an arbitrary $p\times p$ Hermitian matrix with $\|\mathbf{H}\|$ bounded.
\end{lemma}
\begin{proof}
First observe that $\sum_{\ell,\ell^{\prime}=0}^{q}\omega_{\ell,\ell^{\prime}}^{j}a_{j\ell}a_{j\ell^{\prime}}^{*}$ is a Hermitian matrix
of rank $q+1$ and hence we can write it as
$\sum_{\ell=0}^{q}\tilde\omega_{j\ell} b_{j\ell}b_{j\ell}^{*}$, where each $\tilde\omega_{j\ell} \in \{-1,+1\}$
and observe that $a_{j\ell}\zeta_{\j\ell}^{*}+\zeta_{j\ell}a_{j\ell}^{*}=u_{j\ell}u_{j\ell}^{*}-v_{j\ell}v_{j\ell}^{*}$ where $u_{j\ell}=2^{-1/2}(\zeta_{j\ell}+a_{j\ell})$ and $v_{j\ell}=2^{-1/2}(\zeta_{j\ell}-a_{j\ell})$. Define the matrices $\mathbf{D}_{1j}=\mathbf{C}_{\tau}^{(j)}+\sum_{\ell=0}^{q}u_{j\ell}u_{j\ell}^{*}$ and $\mathbf{D}_{2j}=\mathbf{D}_{1j}-\sum_{\ell=0}^{q}v_{j\ell}v_{j\ell}^{*}$, and notice that it then follows from (\ref{eq:expression_relation_C_tau_and_C_tau_i}) that $\mathbf{C}_{\tau}=\mathbf{D}_{2j}+\sum_{\ell=0}^{q}
\tilde\omega_{j\ell}  b_{j\ell}b_{j\ell}^{*}$. Therefore,
 \begin{align*}
\mathrm{diff}_{\tau,j}(\mathbf{H})
\leq& \frac 1p \Big|\tr\big((\mathbf{C}_{\tau}-z\mathbf{I})^{-1}\mathbf{H}\big)
-\tr\big((\mathbf{D}_{2j}-z\mathbf{I})^{-1}\mathbf{H}\big)\Big| \\
&+\frac 1p \Big|\tr\big((\mathbf{D}_{2j}-z\mathbf{I})^{-1}\mathbf{H}\big)
-\tr\big((\mathbf{D}_{1j}-z\mathbf{I})^{-1}\mathbf{H}\big)\Big| \\
&+\frac 1p \Big|\tr\big((\mathbf{D}_{1j}-z\mathbf{I})^{-1}\mathbf{H}\big)
-\tr\big((\mathbf{C}^{(j)}_{\tau}-z\mathbf{I})^{-1}\mathbf{H}\big)\Big| \\
=&K_{j1}+K_{j2}+K_{j3}.
 \end{align*}
In the following an estimate for $K_{j2}$ is given. For $1\leq k\leq q+1$, let then $\mathbf{T}_{j}^{(k)}=\mathbf{D}_{2j}+\sum_{\ell=0}^{k-1}v_{j\ell}v_{j\ell}^{*}$, so that $\mathbf{T}_{j}^{(0)}=\mathbf{D}_{2j}$ and $\mathbf{T}_{j}^{(q+1)}=\mathbf{D}_{1j}$. An application of Lemmas \ref{lemma:rank_one_inverse} and \ref{lemma:quad_bound} implies that
\begin{align*}
K_{j2}
&=\frac 1p\sum_{k=1}^{q+1}\left|\tr\big((\mathbf{T}_{j}^{(k)}-z\mathbf{I})^{-1}\mathbf{H}\big)
-\tr\big((\mathbf{T}_{j}^{(k-1)}-z\mathbf{I})^{-1}\mathbf{H}\big)\right| \\
&\leq \frac{1}{p}\sum_{k=1}^{q+1} \left|
\frac{v_{jk}^{*}(\mathbf{T}_{j}^{(k-1)}-z\mathbf{I})^{-1}\mathbf{H}(\mathbf{T}_{j}^{(k-1)}-z\mathbf{I})^{-1}v_{jk}}
{1+v_{jk}^{*}(\mathbf{T}_{j}^{(k-1)}-z\mathbf{I})^{-1}v_{jk}}\right|
\leq \frac{(q+1)\|\mathbf{H}\|}{p\Im(z)}.
 \end{align*}
Estimates for $K_1$ and $K_3$ can be obtained in a similar way, leading to the bound $(q+1)(p\Im(z))^{-1}\|\mathbf{H}\|$
in each case.
This proves the lemma.
\end{proof}

Lemma \ref{lem:aux:5} gives the bound $\mathrm{diff}_{\tau,j}(\mathbf{I}_p)\leq 3(q+1)(p\Im(z))^{-1}$ and $\mathrm{diff}_{\tau,j}(\mathbf{\Gamma}_\tau(\mathbf{a}))\leq 3(q+1)(p\Im(z))^{-1}L_1^2$. Let $\mathrm{diff}^\prime_{\tau,j}$
be defined as $\mathrm{diff}_{\tau,j}$ with $\mathbf{C}_\tau$ replaced with $\mathbf{C}_\tau^\prime$, where the
latter matrix in turn is obtained from the former replacing its $j$th's row $\mathbf{z}_j^*$ with an independent
copy $(\mathbf{z}_j^\prime)^*$. From Lemma \ref{lem:aux:5} it follows then that
\[
\frac{1}{p}\left|\tr\big((\mathbf{C}_{\tau}-z\mathbf{I})^{-1}\big)
-\tr\big((\mathbf{C}^{\prime}_{\tau}-z\mathbf{I})^{-1}\big)\right|
\leq \frac{6(q+1)}{p\Im(z)}
\]
and
\begin{equation}
\label{eq:bound_trace_resolvent_H}
\frac{1}{p}\left|\tr\big((\mathbf{C}_{\tau}-z\mathbf{I})^{-1}\Gamma_{\tau}(\mathbf{a})\big)
-\tr\big((\mathbf{C}^\prime_{\tau}-z\mathbf{I})^{-1}\Gamma_{\tau}(\mathbf{a})\big)\right|
\leq \frac{6(q+1){L_1}^2}{p\Im(z)}.
\end{equation}
Recognizing that $s_{\tau,p}(z)=p^{-1}\tr((\mathbf{C}_{\tau}-z\mathbf{I})^{-1})$ and $\beta_{\tau, p}(z, \mathbf{a})=p^{-1}\tr((\mathbf{C}_{\tau}-z\mathbf{I})^{-1}\Gamma_{\tau}(\mathbf{a}))$ and applying the McDiarmid's inequality
(Lemma \ref{lemma:mcdiarmid}) yields that, for any $\epsilon>0$,
\begin{equation}\label{eq:applying_mcdiarmid_s_n_tau_0}
\mP\left(|s_{\tau,p}(z)-\mE [s_{\tau,p}(z)]|>\epsilon\right)
\leq 4\exp\left(-\frac{p\Im(z)\epsilon^2}{18(q+1)^2}\right)
\end{equation}
and
\begin{equation}
\label{eq:mc-diarmid-2}
\mP\left(|\beta_{\tau, p}(z, \mathbf{a}) -\mE [\beta_{\tau, p}(z, \mathbf{a})]|>\epsilon\right)
\leq 4\exp\left(-\frac{p\Im(z)\epsilon^2}{18(q+1)^2{L_{1}}^2}\right).
\end{equation}
Now the Borel--Cantelli lemma implies that $|s_{\tau,p}(z)-\mE [s_{\tau,p}(z)]|\to 0$ and $|\beta_{\tau,p}(z, \mathbf{a})-\mE[\beta_{\tau, p}(z,\mathbf{a})]|\to 0$ almost surely under \eqref{eq:mod_dim}. Moreover, it can be readily seen that
these almost sure convergence results also hold for $\tilde s_{\tau,p}$ and $\tilde\beta_{\tau,p}$.


\subsection{Existence, uniqueness and continuity of the solution}
\label{sec:ex_conv_cont}

This section provides a proof of the existence of a unique solution $s_\tau(z)$ and $\beta_{\tau}(z, \mathbf{a})$, for $\mathbf{a}\in \mathrm{supp}(F^{\mathcal{A}})$ and $z\in\mathbb{C}^+$, to the set of equations (\ref{eq:Stieltjes_LSD_covariance})--(\ref{eq:spectral_kernel_covariance}). Assuming that these solutions exist,
it can be shown that $\tilde s_{\tau,p}(z)\stackrel{a.s.}{\longrightarrow} s_{\tau}(z)$ and
$\tilde\beta_{\tau, p}(z, \mathbf{a})\stackrel{a.s.}{\longrightarrow} \beta_{\tau}(z, \mathbf{a})$
for any $\mathbf{a}\in \mathrm{supp}(F^{\mathcal{A}})$ and $z\in \mC^{+}$. In view of the results
derived in Section \ref{sec:conv_random_part} and Lemma \ref{lem:aux:4},
it suffices to show that for every sequence $\{p_j \colon j \in \mathbb{N}\}$ there exists a further
subsequence $\{\tilde p_j\colon j \in \mathbb{N}\}$ such that $\mathbb{E}(\tilde\beta_{\tau,\tilde p_j}(z,\mathbf{a}))$ converges
to a limit $\beta_\tau(z,\mathbf{a})$ satisfying (\ref{eq:Stieltjes_LSD_covariance})--(\ref{eq:spectral_kernel_covariance}).
The verification is based on a diagonal subsequence argument and the
Arzel\`{a}--Ascoli theorem.


\begin{lemma}
Let $\{p_j \colon j\in\mathbb{N}\}$ denote a subsequence of the integers $\mathbb{N}$ and define
$\rho_{\tau, p_{j}}(z,\mathbf{a})=\mE[\tilde\beta_{\tau, p_{j}}(z,\mathbf{a})]$. Then the following statements hold. \\[-1cm]
\begin{itemize}
\item[($a$)] There is a further subsequence $\{\tilde p_{j} \colon j\in\mathbb{N}\}$ such that
$\rho_{\tau, \tilde p_{j}}(z,\mathbf{a})$ convergences uniformly in $\mathbf{a}\in\mathrm{supp}(F^{\mathcal{A}})$
and pointwise in $z\in \mC^{+}$ to a limit $\rho_{\tau}(z, \mathbf{a})$ which is analytic in $z$ and continuous in $\mathbf{a}$;

\item[($b$)] The limit $\rho_{\tau}(z,\mathbf{a})$ in (a) coincides with $\beta_{\tau}(z,\mathbf{a})$ and is the Stieltjes
transform of a  measure on the real line with mass $\int \mathcal{R}_{\tau}(\mathbf{a},\mathbf{b})dF^{\mathcal{A}}(\mathbf{b})$
satisfying (\ref{eq:Stieltjes_kernel_covariance}).
 \end{itemize}
\end{lemma}

\begin{proof}
\textit{Step 1:}\ 
Define $\mathcal{F}=\{\rho_{\tau, p_{j}(\mathbf{a})}(\cdot,\mathbf{a})\colon \mathbf{a}\in \mathrm{supp}(F^{\mathcal{A}})\}$.
For any compact set $K\subset \mC^{+}$,
\[
|\rho_{\tau,p_{j}(\mathbf{a})}(z,\mathbf{a})|\leq L_{1}^2/\min_{z\in K}\Im(z)=M(K).
\]
Let $\{\mathbf{a}_1,\mathbf{a}_2,\ldots\}$ be an enumeration of the dense subset
$\mathrm{supp}(F^{\mathcal{A}})\cap \mathbb{Q}^m$ of $\mathrm{supp}(F^{\mathcal{A}})$. An application of Lemma \ref{lem:Uniform_convergence_analytic_limit} yields that for any $\mathbf{a}_\ell$ there exists a further
subsequence $\{p_{j}(\mathbf{a}_\ell) \colon  j\in\mathbb{N}\}$ such that $\cdots\subset\{p_{j}(\mathbf{a}_\ell)\}\subset\{p_{j}(\mathbf{a}_{\ell-1})\}\subset\cdots\subset \{p_{j}(\mathbf{a}_1)\}$
such that $\rho_{\tau, p_{j}(\mathbf{a}_{\ell})}(z, \mathbf{a}_{\ell})$ converges
uniformly on compact subsets of $\mC^+$ to a limit denoted by $\rho_\tau(z,\mathbf{a}_\ell)$, which is an
analytic function of $z\in\mC^{+}$ for each $\ell\in\mathbb{N}$. Choosing the diagonal subsequence
$\{p_j(\mathbf{a}_j)\colon\mathbb{N}\}$, it follows that
\[
\rho_{\tau,p_j(\mathbf{a}_j)}(z,\mathbf{a}_\ell)
\to \rho_\tau(z,\mathbf{a}_\ell)
\qquad (j\to\infty)
\]
for all $\ell\in\mathbb{N}$ uniformly on compact subsets of $\mC^+$. Note that the limit is defined on $\mC^{+}\times (\mathrm{supp}(F^{\mathcal{A}})\cap \mathbb{Q}^m)$.

\textit{Step 2:}\ It is shown in Appendix \ref{sec:app:equicont} that, for any fixed $z\in \mC^{+}$ and
subsequence $\{p_j\}$, $\{\rho_{\tau, p_j}(z, \mathbf{a})\}$ are equicontinuous functions.
Since $\rho_{\tau, p_\ell(\mathbf{a}_\ell)}(z, \mathbf{a})$ converges pointwise to $\rho_{\tau}(z, \mathbf{a})$
on the dense subset $\mathrm{supp}(F^{\mathcal{A}})\cap \mathbb{Q}^{m}$ of $\mathrm{supp}(F^{\mathcal{A}})$,
the Arzel\`{a}-Ascoli theorem (Lemma \ref{lemma:Arzela_Ascoli_theorem}) implies that $\rho_{\tau, p_{\ell}(\mathbf{a}_{\ell})}(z, \mathbf{a})$ uniformly converges to a limit, a continuous function of $\mathbf{a}\in \mathrm{supp}(F^{\mathcal{A}})$,
that coincides with $\rho_{\tau}(z, \mathbf{a})$ for $\mathbf{a}\in \mathrm{supp}(F^{\mathcal{A}})\cap \mathbb{Q}^{m}$.
Thus, the limit $\rho_{\tau}(z, \mathbf{a})$ is now defined on $\mC^{+}\times\mathrm{supp}(F^{\mathcal{A}})$ and
is analytic in $z\in\mC^{+}$. From \eqref{eq:expression_on_expectation_beta_tau_p} it follows that the
limit $\rho_{\tau}(z, \mathbf{a})$ coincides with $\beta_{\tau}(z,\mathbf{a})$ for $\mathbf{a}\in \mathrm{supp}(F^{\mathcal{A}})$.

\textit{Step 3:}\ It remains to show that $\beta_{\tau}(z,\mathbf{a})$ is the Stieltjes transform of a measure on the real line with mass $m_{\tau}(\mathbf{a}):=\int\mathcal{R}_{\tau}(\mathbf{a},\mathbf{b})dF^{\mathcal{A}}(\mathbf{b})$. This is equivalent to showing that $(m_{\tau}(\mathbf{a}))^{-1}\beta_{\tau}(z,\mathbf{a})$ is the Stieltjes transform of a Borel probability measure.
The proof relies on the Lemma \ref{lem:tightness}, stated below. From the definition of $\tilde\beta_{\tau,p}(z,\mathbf{a})$
and the fact that $\Gamma_\tau(\mathbf{a})$ is a positive definite matrix with bounded norm, it follows that
$(m_{\tau,p}(\mathbf{a}))^{-1}\tilde\beta_{\tau,p}(z,\mathbf{a})$ is the Stieltjes transform of a probability measure
$\mu_{p,\mathbf{a}}$ where $m_{\tau,p}(\mathbf{a}) = p^{-1}\tr(\Gamma_\tau(\mathbf{a}))$. The measure $\mu_{p,\mathbf{a}}$ is
such that $\mu_{p,\mathbf{a}}((x,\infty)) \leq \| \Gamma_\tau(\mathbf{a})\| (m_{\tau,p}(\mathbf{a}))^{-1}
F^{\tilde{\mathbf{C}}_\tau}((x,\infty))$ for all $x$. Now, by the tightness of the sequence $\{F^{\tilde{\mathbf{C}}_\tau}\}$ (by Lemma
\ref{lem:tightness}), it follows that $\{\mu_{p,\mathbf{a}}\}$ is a tight sequence of probability measures. Now, by Step 2 and
the conclusion in Section \ref{sec:conv_random_part}, it follows  there is a subsequence $\{p_\ell\}$ such that
the Stieltjes transform of $(m_{\tau,p_\ell}(\mathbf{a}))^{-1}\tilde\beta_{\tau,p_\ell}(z,\mathbf{a})$ converges almost surely
to $(m_{\tau}(\mathbf{a}))^{-1}\beta_\tau(z,\mathbf{a})$ for each $z\in \mathbb{C}^+$. The conclusion that
$(m_{\tau}(\mathbf{a}))^{-1}\beta_{\tau}(z,\mathbf{a})$ is the Stieltjes transform of a Borel probability measure
then follows from Lemma \ref{lemma:prob_measure_convergence}.
\end{proof}

\begin{lemma}\label{lem:tightness}
Under the conditions of Theorem \ref{thm:lin_proc}, $F^{\mathbf{C}_\tau}$ is a tight sequence.
\end{lemma}
It should be noted that Lemma \ref{lem:tightness}, together with $s_{\tau,p}(z)\stackrel{a.s.}{\longrightarrow} s_{\tau}(z)$
for $z \in \mathbb{C}^+$, proves the existence of the LSD of $\mathbf{C}_\tau$. The proof of Lemma \ref{lem:tightness} is given
in Appendix \ref{sec:proof_lemma_tightness}.


Next, we prove the uniqueness of the solutions $\beta(z,\mathbf{a})$ under the constraint that the solutions belong to the class of
Stieltjes kernels that are analytic on $\mC^{+}$ for all $\mathbf{a}\in \mathrm{supp}(F^{\mathcal{A}})$. First, we verify the uniqueness of the solution for $z\in \mC^{+}(v_0)=\{z\in \mC^{+}\colon\Im(z)>v_{0}\}$ for sufficiently large $v_0 > 0$. At the same time, continuity of the solution with respect to $F^{\cal A}$ is verified. Accordingly, let $\beta_{\tau}(z, \mathbf{a})$ satisfy (\ref{eq:Stieltjes_kernel_covariance}) for any $\mathbf{a}\in \mathrm{supp}(F^{\mathcal{A}})$. In view of establishing the continuous dependence of $\beta_{\tau}(z, \mathbf{a})$, and hence $s_{\tau}(z)$, on $F^{\mathcal{A}}$, on $F^{\mathcal{A}}$ and the kernel ${\cal R}_\tau$, suppose that there is a possibly different distribution $F^{\bar{\mathcal{A}}}$ and a possibly different kernel $\bar{\cal R}_\tau$ (but having the same properties as
${\cal R}_\tau$) such that $\bar{\beta}_{\tau}(z, \mathbf{a})$ satisfies
\[
\bar{\beta}_{\tau}(z, \mathbf{a})
= - \int \frac{\bar{\cal R}_\tau(\mathbf{a},\mathbf{b})dF^{\bar{\mathcal{A}}}(\mathbf{b})}{z + \bar{\beta}_{\tau}(z, \mathbf{b})}, \qquad \mathbf{a}
\in \mathbb{R}^{m_0},
\]
and is a Stieltjes transform of a measure for all $\mathbf{a} \in \mathrm{supp}(F^{\bar{\mathcal{A}}})$.
Note that, by the defining equations and the continuity of ${\cal R}_\tau(\mathbf{a},\mathbf{b})$,
and $\bar{\cal R}_\tau(\mathbf{a},\mathbf{b})$, the functions
$\beta(z,\mathbf{a})$ and $\bar\beta(z,\mathbf{a})$ are continuous in $\mathbf{a}$
for all $z \in \mC^{+}$. Also,
\begin{eqnarray}\label{eq:uniqueness_on_beta_n}
\beta_{\tau}(z, \mathbf{a}) -\bar{\beta}_{\tau}(z, \mathbf{a})
&=& \int \frac{{\cal R}_\tau(\mathbf{a},\mathbf{b})(\beta_{\tau}(z,\mathbf{b})-\bar{\beta}_{\tau}(z,\mathbf{b}))
dF^{\mathcal{A}}(\mathbf{b})}{(z + \beta_{\tau}(z, \mathbf{b}))(z + \bar{\beta}_{\tau}(z, \mathbf{b}))}
- \int \frac{({\cal R}_\tau(\mathbf{a},\mathbf{b}) - \bar{\cal R}_\tau(\mathbf{a},\mathbf{b}))dF^{\cal A}(\mathbf{b})}
{z + \bar{\beta}_{\tau}(z, \mathbf{b})}\nonumber\\
&&
-\int \frac{\bar{\cal R}_\tau(\mathbf{a},\mathbf{b})d(F^{\mathcal{A}}(\mathbf{b})-F^{\bar{\mathcal{A}}})}{z +\bar{\beta}_{\tau}(z, \mathbf{b})}.
\end{eqnarray}
Define
\begin{equation}
\|\beta_{\tau}(z,\cdot)-\bar{\beta}_{\tau}(z,\cdot)\|^2_{\mathcal{A}}=\int |\beta(z,\mathbf{a})-\bar{\beta}_{\tau}(z,\mathbf{a})|^2 dF^{\mathcal{A}}(\mathbf{a}).
\end{equation}
Then, by Cauchy--Schwarz inequality,
\begin{align}
\big|\beta_{\tau}(z,\mathbf{a})
&-\bar{\beta}_{\tau}(z,\mathbf{a})\big|^2\nonumber\\
&\leq 3\left|\int\frac{\mathcal{R}_{\tau}(\mathbf{a},\mathbf{b})(\beta_{\tau}(z,\mathbf{b})-\bar{\beta}_{\tau}(z,\mathbf{b}))
dF^{\mathcal{A}}(\mathbf{b})}{(z+\beta_{\tau}(z,\mathbf{b}))(z+\bar{\beta}_{\tau}(z,\mathbf{b}))}\right|^2
+ r_\tau^{(1)}(\mathbf{a}) + r_\tau^{(2)}(\mathbf{a})\nonumber\\
&\leq  3 \left[\int |\beta_{\tau}(z,\mathbf{b})-\bar{\beta}_{\tau}(z,\mathbf{b})|^2dF^{\mathcal{A}}(\mathbf{b})\right]
\left[\int \frac{\mathcal{R}^2_{\tau}(\mathbf{a},\mathbf{b})dF^{\mathcal{A}}(\mathbf{b})}
{|z+\beta_{\tau}(z,\mathbf{b})|^2|z+\bar{\beta}_{\tau}(z,\mathbf{b})|^2}\right]+ r_\tau^{(1)}(\mathbf{a}) + r_\tau^{(2)}(\mathbf{a}),
\label{eq:beta_diff_norm_inequality}
\end{align}
where
\[
r_{\tau}^{(1)}(\mathbf{a}) = 3 \left| \int \frac{({\cal R}_\tau(\mathbf{a},\mathbf{b}) - \bar{\cal R}_\tau(\mathbf{a},\mathbf{b}))dF^{\cal A}(\mathbf{b}) }{z + \bar{\beta}_{\tau}(z, \mathbf{b})}\right|^2
\leq \frac{3}{v^2}\|{\cal R}_\tau - \bar{\cal R}_\tau\|_\infty^2,
\]
where $\|{\cal R}_\tau - \bar{\cal R}_\tau\|_\infty
= \sup_{\mathbf{a},\mathbf{b} \in \mathbb{R}^{m_0}}|{\cal R}_{\tau}(\mathbf{a},\mathbf{b}) - \bar{\cal R}_{\tau}(\mathbf{a},\mathbf{b})|$,
and
\[
r_{\tau}^{(2)}(\mathbf{a})
=3 \left|\int\frac{\bar{\cal R}_{\tau}(\mathbf{a},\mathbf{b})d(F^{\bar{\mathcal{A}}}-F^{\mathcal{A}})(\mathbf{b})}
{z+\bar{\beta}_{\tau}(z,\mathbf{b})}\right|^2
\leq \frac{6(L_1^4 + \|{\cal R}_\tau - \bar{\cal R}_\tau\|_\infty^2)}{v^2} \| F^{\mathcal{A}}- F^{\bar{\mathcal{A}}}\|_{TV}^2,
\]
where $\parallel \cdot \parallel_{TV}$ denotes the total variation distance.
Taking $v_{0}= \max\{1, \sqrt{2} L_1\}$, if follows for $v > v_{0}$ that
\[
\int\frac{\mathcal{R}_{\tau}^2(\mathbf{a},\mathbf{b})dF^{\mathcal{A}}(\mathbf{b})}
{\left|z+\beta_{\tau}(z,\mathbf{b})\right|^{2}\left|z+\bar{\beta}_{\tau}(z,\mathbf{b})\right|^{2}}\leq \frac{L_1^4}{v^4} < \frac{1}{4}.
\]
Therefore, by (\ref{eq:beta_diff_norm_inequality}), for $v > v_0$,
\begin{eqnarray}\label{eq:beta_diff_norm_bound}
\parallel \beta_\tau(z,\cdot) - \bar\beta_\tau(z,\cdot)\parallel_{\mathcal{A}}^2 &\leq& 4\int (r_\tau^{(1)}(\mathbf{a})+r_\tau^{(2)}(\mathbf{a})) dF^{\mathcal{A}}(\mathbf{a}) \nonumber\\
&\leq& \frac{12}{v^2}\left(\|{\cal R}_\tau - \bar{\cal R}_\tau\|_\infty^2 +
2(L_1^4+\|{\cal R}_\tau - \bar{\cal R}_\tau\|_\infty^2) \| F^{\mathcal{A}} - F^{\bar{\mathcal{A}}}\|_{TV}^2\right).
\end{eqnarray}
If $F^{\mathcal{A}}=F^{\bar{\mathcal{A}}}$, and ${\cal R}_\tau = \bar{\cal R}_\tau$,
(\ref{eq:beta_diff_norm_bound}) and the continuity of $\beta_{\tau}(z, \mathbf{a})$
and $\bar{\beta}_{\tau}(z,\mathbf{a})$ in $\mathbf{a}$ imply that $\beta_{\tau}(z, \mathbf{a})=\bar{\beta}_{\tau}(z,\mathbf{a})$ for $z\in\mC^{+}(v_{0})$ and $\mathbf{a} \in \mbox{supp}(F^{\mathcal{A}})$.
Then, since both are analytic functions on $\mC^{+}$ for every fixed $\mathbf{a}\in \mbox{supp}(F^{\mathcal{A}})$,
the uniqueness of the solution in $z\in \mC^{+}$ follows.
Moreover, (\ref{eq:beta_diff_norm_bound}) proves the continuous dependence of the solution $\beta_\tau(z,\cdot)$ on
on ${\cal R}_\tau$ and $F^{\mathcal{A}}$, with respect to the topology of uniform convergence and
that of total variation norm, respectively. From this, similar properties for $s_\tau$ are easily deduced.

\section{Proof of Theorem \ref{thm:lin_proc}}
\label{sec:proof:lin_proc}


In this section, the results are extended to the setting that $q$ is not fixed, but
tends to infinity at certain rate. In fact, $q=O(p^{1/4})$ is an appropriate choice. This
rate plays a crucial role in two places of the derivations. First in verifying properties
(such as continuity) of the solution and then in transitioning from the Gaussian to the
non-Gaussian case. The latter situation requires the $1/4$ power, while the former can
be worked out under the weaker assumption that $q=o(p^{1/2})$. It is shown here that the
LSD of the truncated process is the same as that of the linear process almost surely. Denote then by
\begin{equation}
\mathbf{S}^\mathrm{tr}_{\tau}=\frac{1}{2n}\left(\sum_{t=\tau+1}^{n}X^\mathrm{tr}_{t}{X^\mathrm{tr}_{t-\tau}}^{*} +\sum_{t=\tau+1}^{n}X^\mathrm{tr}_{t-\tau}{X^\mathrm{tr}_{t}}^{*}\right)
\end{equation}
 the symmetrized auto-covariance matrix for the truncated process
 $X^\mathrm{tr}_{t} =\sum_{\ell=0}^{q}\mathbf{A}_{\ell}Z_{t-\ell}$, $t\in \mZ$.
 Let $L(F,G)$ denote the Levy distance between distribution function $F$ and $G$, defined by
\begin{equation*}
L(F,G) = \inf\{\epsilon>0\colon F(x-\epsilon)-\epsilon\leq G(x)\leq F(x+\epsilon)+\epsilon\}.
\end{equation*}
In view of Lemma \ref{lem:levy_distance_normal_matrices}, the aim is to show that
\begin{equation}
\label{eq:inequality_levy_distance_truncated}
L^{3}(F^{\mathbf{C}_{\tau}},F^{\mathbf{C}^\mathrm{tr}_{\tau}})
\leq \frac{1}{p}\tr(\mathbf{C}_{\tau}-\mathbf{C}^\mathrm{tr}_{\tau})^2 \to 0
\qquad\mbox{a.s.}
\end{equation}
To this end, define $\bar{X}_t=X_t-X_t^\mathrm{tr}=\sum_{\ell=q+1}^\infty \mathbf{A}_\ell Z_{t-\ell}$ and notice that
\begin{align*}
\mathbf{S}_{\tau}-\mathbf{S}_{\tau}^\mathrm{tr}
=&
\frac{1}{2n}\sum_{t=1}^{n-\tau}(X_{t}X^*_{t+\tau} + X_{t+\tau}X^{*}_{t})
-\frac{1}{2n}\sum_{t=1}^{n-\tau}(X^\mathrm{tr}_{t}{X^\mathrm{tr}_{t+\tau}}^*+X_{t+\tau}^\mathrm{tr}{X^\mathrm{tr}_{t}}^{*}) \\
=&
\frac{1}{2n}\sum_{t=1}^{n-\tau}(\bar{X}_{t}{X^\mathrm{tr}_{t+\tau}}^{*}+X^\mathrm{tr}_{t+\tau}\bar{X}^{*}_{t})
+\frac{1}{2n}\sum_{t=1}^{n-\tau}(X^\mathrm{tr}_{t}\bar{X}^{*}_{t+\tau}+\bar{X}_{t+\tau}{X^\mathrm{tr}_{t}}^{*}) \\
& ~~~~~~~~~~~~~
+\frac{1}{2n}\sum_{t=1}^{n-\tau}(\bar{X}_{t}\bar{X}^{*}_{t+\tau}+\bar{X}_{t+\tau}\bar{X}^{*}_{t}) \\
=& \mathbf{S}_{\tau,1}+\mathbf{S}_{\tau,2}+\mathbf{S}_{\tau,3}.
\end{align*}
Therefore,
\begin{equation}\label{eq:S_trunc_diff_norm_sum}
\|\mathbf{C}_{\tau}-\mathbf{C}^\mathrm{tr}_{\tau}\|^{2}_{F}
\leq 3 \left(
\frac{n}{p}\|\mathbf{S}_{\tau,1}-\mE [\mathbf{S}_{\tau,1}]\|^{2}_{F}
+\frac{n}{p}\|\mathbf{S}_{\tau,2}-\mE [\mathbf{S}_{\tau,2}]\|^{2}_{F}
+\frac{n}{p}\|\mathbf{S}_{\tau,3}-\mE [\mathbf{S}_{\tau,3}]\|^{2}_{F}
\right).
\end{equation}
Hence, to prove that (\ref{eq:inequality_levy_distance_truncated}) holds, it suffices to show that
\begin{equation}\label{eq:series_summation}
\sum_{p=1}^{\infty}\frac{n}{p^2}\mE\big[\|\mathbf{S}_{\tau,i}-\mE [\mathbf{S}_{\tau,i}]\|^{2}_{F}\big]< \infty,
\qquad i=1,2,3,
\end{equation}
due to the Borel-Cantelli lemma. The corresponding detailed calculations are performed in Appendix \ref{app:lin_proc}.


\section{Extension to non-Gaussian settings}
\label{sec:proof:non-gauss}


In this section, it is shown that Theorem \ref{thm:LSD_variance} and Theorem \ref{thm:LSD_generalized_process_B}
extend beyond the Gaussian setting. In order to show this, Lindeberg's replacement strategy as developed in
\cite{Chatterjee:2006} is applied to a process consisting of truncated, centered and rescaled versions of the
original innovation entries $Z_{tj}$. To formally define this transformation, let $\epsilon_p>0$ be such that $\epsilon_p\to 0$,
$p^{1/4}\epsilon_p\to\infty$ and $\mathbb{P}(|Z_{11}|\geq n^{1/4}\epsilon_p)\leq n^{-1}\epsilon_p$. The existence
of such an $\epsilon_p$ follows from {\bf Z1} and {\bf Z2}. Let then
$\breve{Z}_{tj}^{c}=Z_{tj}^{c}I_{\{|Z_{tj}^{c}|\leq n^{1/4}\epsilon_{p}\}}$ denote the truncated innovations and $\hat{Z}_{tj}^{c}=(\breve{Z}_{tj}^{c}-\mE[\breve{Z}_{tj}^{c}])/(2\mbox{sd}(\breve{Z}_{ij}^{c}))$ the standardized versions
where $c \in \{\mathbf{R},\mathbf{I}\}$ with the superscripts $\mathbf{R}$ and $\mathbf{I}$ denoting the real and imaginary parts.
Let further $\hat X_t=\sum_{\ell=0}^q\mathbf{A}_\ell\hat Z_{t-\ell}$, $t\in\mathbb{Z}$, and define the
autocovariance matrix of $(\hat X_t\colon t\in\mathbb{Z})$ be defined by
\[
\hat{\mathbf{C}}_{\tau}:=\sqrt{\frac{n}{p}}(\hat{\mathbf{S}}_{\tau}-\mE [\hat{\mathbf{S}}_{\tau}]),
\]
where
\begin{equation}
\hat{\mathbf{S}}_{\tau}
 = \frac{1}{2(n-\tau)}\bigg(
 \sum_{t=\tau+1}^{n}\hat{X}_t\hat{X}_{t-\tau}^{*}+\sum_{t=\tau+1}^{n}\hat{X}_{t-\tau}\hat{X}_t^{*}
 \bigg).
\end{equation}
The LSD of the auto-covariance matrix of $\mathbf{C}_\tau$ is the same as that of $\hat{\mathbf{C}}_\tau$,
since, according to \cite{Bai:Yin:1988} and \cite{Liu:Aue:Paul:2015}, an application of a rank inequality
and Bernstein's inequality implies that
\[
\sup_{x}\big|F^{\mathbf{C}_{\tau}}(x)-F^{\hat{\mathbf{C}}_{\tau}}(x)\big|
\to 0
\qquad\mbox{a.s.}
\]
For notational simplicity, the truncated, centered and rescaled variables are therefore henceforth
still denoted by $Z_{jt}$ (correspondingly, $X_{jt}$) and it is assumed that they are i.i.d.\ with $|Z_{11}| \leq n^{1/4}\epsilon_p$,
$\mathbb{E}[Z_{11}] = 0$, $\mathbb{E}[|Z_{11}|^2] = 1$, the real and imaginary parts are independent
with equal variance, and $\mathbb{E}[|Z_{11}|^4] = \mu_4$ for some
finite constant $\mu_4$.

Consider now the process $(X_t^\prime\colon t\in\mathbb{Z})$ given by
\begin{equation}
X^\prime_{t}=\sum_{\ell=0}^{q}\mathbf{A}_{\ell}W_{t-\ell},
\qquad t\in\mathbb{Z},
\end{equation}
with the innovations $(W_t\colon t\in\mathbb{Z})$ consisting of i.i.d.\ real- or complex-valued (not necessarily Gaussian) entries $W_{jt}$ satisfying
\begin{itemize}
\itemsep-.2ex
\item[{\bf T1}] $\mathbb{E}[W_{jt}] = 0$, $\mathbb{E}[|W_{jt}|^2] = 1$ and $\mathbb{E}[|W_{jt}|^4] \leq C$ for some finite constant $C>0$;
\item[{\bf T2}] In case of complex-valued innovations, the real and imaginary parts of $W_{jt}$ are independent with $\mathbb{E}[\Re(W_{jt})]
=\mathbb{E}[\Im(W_{jt})] = 0$ and $\mathbb{E}[\Re(W_{jt})^2]=\mathbb{E}[\Im(W_{jt})^2] = 1/2$;
\item[{\bf T3}] $|W_{jt}|\leq n^{1/4}\epsilon_{p}$ with $\epsilon_p > 0$ such that $\epsilon_{p}\to 0$ and $p^{1/4}\epsilon_{p} \to \infty$;
\item[{\bf T4}] The $W_{jt}$ are independent of the $Z_{tj}$ defined in Theorem \ref{thm:LSD_variance}.
\end{itemize}
It is assumed that the coefficient matrices $(\mathbf{A}_{\ell}\colon\ell\in\mathbb{N})$ satisfy conditions
{\bf A1}--{\bf A5}.  Define the lag-$\tau$ auto-covariance matrix of $(X_{t}^\prime\colon t\in\mathbb{Z})$ by
\begin{equation}
\mathbf{S}^\prime_{\tau}
=\frac{1}{2(n-\tau)}\bigg(
\sum_{t=\tau+1}^{n}X_t^\prime{X_{t-\tau}^\prime}^{*}+\sum_{t=\tau+1}^{n}X_{t-\tau}^\prime{X_{t}^\prime}^{*}
\bigg),
\end{equation}
so that the corresponding renormalized lag-$\tau$ auto-covariance matrix is given by
\begin{equation*}
\mathbf{C}_{\tau}^\prime = \sqrt{\frac{n}{p}}(\mathbf{S}_{\tau}^\prime-\mE[\mathbf{S}_{\tau}^\prime])
\end{equation*}
and the lag-$\tau$ Stieltjes transform by
$s_{\tau,p}^\prime(z)=\frac{1}{p}\tr(\mathbf{C}_{\tau}^\prime-zI)^{-1}$, $z\in \mC^{+}$. We denote the Stieltjes
transform of $\mathbf{C}_\tau$, defined in terms of the bounded (after trunctation and normalization)
$Z_{jt}$'s, by $s_{\tau,p}$. Since we have proved the existence and uniqueness
of LSD in the case where $Z_{jt}$'s are i.i.d. standard Gaussian, it follows that for all $z\in \mC^+$, $s_{\tau,p}(z)$ converges a.s.
to the Stieltjes transform of the LSD determined by (\ref{eq:Stieltjes_LSD_covariance}) and (\ref{eq:Stieltjes_kernel_covariance}).
Thus, proving that the results hold for non-Gaussian innovations means showing that
(i) $s_{\tau,p}^\prime(z)-\mE [s_{\tau,p}^\prime(z)]\to 0$ a.s.\ and
(ii) $\mE[s_{\tau,p}(z)-s_{\tau,p}^\prime(z)]\to 0$ for all $z\in\mC^{+}$ under \eqref{eq:mod_dim}.
Since (\ref{eq:applying_mcdiarmid_s_n_tau_0}) has been derived without invoking Gaussianity of the innovations,
(i) follows readily. To show that (ii) holds requires an application of the Linderberg principle developed
in \cite{Chatterjee:2006}. This task is equivalent to verifying that the difference
\begin{equation}\label{eq:difference_gaussian_stieltjes_nongaussian_stieltjes}
\mE\left(\frac{1}{p}\tr(\mathbf{C}_\tau-zI)^{-1}\right)-\mE\left(\frac{1}{p}\tr(\mathbf{C}_{\tau}^\prime-zI)^{-1}\right)
\end{equation}
tends to zero. The arguments for (ii) to hold are provided in Appendix \ref{app:non-gauss}.


\appendix


\section{Technical lemmas}
\label{sec:tech_lemmas}

\begin{lemma}
\label{lemma:rank_one_inverse}
Supposing that $\mathbf{A}$ is invertible and $c^*\mathbf{A}^{-1} b \neq - 1$, it holds
\[
(\mathbf{A} + bc^*)^{-1}
= \mathbf{A}^{-1} - \frac{\mathbf{A}^{-1} bc^* \mathbf{A}^{-1}}{1+c^*\mathbf{A}^{-1}b}.
\]
\end{lemma}
\begin{lemma}[\citet{McDiarmid:1989} Inequality]
\label{lemma:mcdiarmid}
Let $X_{1},\ldots,X_{m}$ be independent random variables taking values in $\mathcal{X}$.
Suppose that $f\colon\mathcal{X}^{m}\to\mR$ is a function of $X_{1},\ldots,X_{m}$ satisfying, for all $x_{1},\ldots, x_{m}$ and $x^{\prime}_{j}$,
\[
|f(x_1,\ldots,x_{j},\ldots, x_{m}) - f(x_1,\ldots,x_{j}^{\prime},\ldots,x_{m})|\leq c_{j}.
\]
Then, for all $\epsilon>0,$
\begin{equation*}
\mP\left(|f(X_{1},\ldots, X_{m})-\mE [f(X_{1}, \ldots, X_{m})]|>\epsilon\right)
\leq 2 \exp\left(-\frac{2\epsilon^{2}}{\sum_{j=1}^{m}c_{j}^{2}}\right).
\end{equation*}
\end{lemma}
\begin{lemma}[\citet{Silverstein:Bai:1995}, Lemma 2.6]
\label{lemma:quad_bound}
Let $z \in \mathbb{C}^+$ with $v = \Im(z)$. Let $\mathbf{A}$ and $\mathbf{B}$ be $n \times n$ matrices with $\mathbf{A}$ Hermitian, and let $r \in \mathbb{C}^n$. Then,
\[
\left|\tr\left(\{(\mathbf{A}-z\mathbf{I})^{-1} - (\mathbf{A}+rr^* - z\mathbf{I})^{-1}\}\mathbf{B}\right)\right|
=\left|\frac{r^* (\mathbf{A}-z\mathbf{I})^{-1}\mathbf{B}(\mathbf{A}-z\mathbf{I})^{-1}r}
{1+r^*(\mathbf{A}-z\mathbf{I})^{-1}r}\right|
\leq \frac{\|\mathbf{B}\|}{v}.
\]
\end{lemma}
\begin{lemma}[\citet{Silverstein:Bai:1995}, Lemma 8.10]
\label{lemma:moments_of_quadratic_forms}
Let $\mathbf{A}$ be an $n\times n$ non-random matrix and $X=(X_1, \ldots, X_n)^{T}$ be a random vector of
independent entries. Assume that $\mE [X_{j}]=0,$ $\mE[|X_{j}|^2]=1$ and $\mE[|X_{j}|^{\ell}]\leq \nu_{\ell}$.
Then, for any integer $\alpha\geq 2$,
\begin{equation*}
\mE\left[|X^{*}\mathbf{A}X-\tr(\mathbf{A})|^{\alpha}\right]
\leq C_{\alpha}\left(\nu_{2\alpha}\tr((\mathbf{A}\mathbf{A}^{*})^{\alpha/2})
+(\nu_{4}\tr(\mathbf{A}\mathbf{A}^*))^{\alpha/2})\right),
\end{equation*}
where $C_{\alpha}$ is a constant depending on $\alpha$ only, and for any real function $f$ on
$\mathbb{R}$,  $\tr(f(\mathbf{A}^{*}\mathbf{A})) = \sum_{i=1}^n f(\lambda_i(\mathbf{A}^*\mathbf{A}))$
where $\lambda_i(\mathbf{A}^*\mathbf{A})$ is the $i$-th largest eigenvalue.
\end{lemma}
\begin{lemma}[\citet{Bai:Silverstein:2010}, Theorem A.43]
\label{lemma:rank_inequality}
Let $\mathbf{A}$ and $\mathbf{B}$ be two $p\times p$ Hermitian matrices. Then, $\|F^\mathbf{A} - F^\mathbf{B}\|\leq \frac{1}{p}\mbox{rank}(\mathbf{A}-\mathbf{B})$, where $\parallel f \parallel$ means $\sup_x|f(x)|$.
\end{lemma}
\begin{lemma}[\citet{Bai:Silverstein:2010}, Theorem A.44]
\label{lemma:rank_inequality_rectangular}
Let $\mathbf{A}$ and $\mathbf{B}$ be two $p\times n$ complex matrices with ESD's $F^{\mathbf{A}}$ and $F^\mathbf{B}$. Then,
\[
\|F^{\mathbf{A}\mathbf{A}^{*}}-F^{\mathbf{B}\mathbf{B}^{*}}\|
\leq \frac{1}{p}\mathrm{rank}(\mathbf{A}-\mathbf{B}).
\]
More generally, if $\mathbf{C}$ and $\mathbf{D}$ are Hermitian matrices of orders $p\times p$ and $n\times n$ respectively, then,
\[
\|F^{\mathbf{C}+\mathbf{A}\mathbf{D}\mathbf{A}^{*}}-F^{\mathbf{C}+\mathbf{B}\mathbf{D}\mathbf{B}^{*}}\|
\leq \frac{1}{p}\mathrm{rank}(\mathbf{A}-\mathbf{B}).
\]
\end{lemma}
\begin{lemma}[\citet{Bai:Silverstein:2010}, Corollary A.40]
\label{lem:levy_distance_normal_matrices}
Let $\mathbf{A}$ and $\mathbf{B}$ be two $n\times n$ normal matrices with ESD's $F^{\mathbf{A}}$ and $F^\mathbf{B}$. Then, $L^{3}(F^\mathbf{A},F^\mathbf{B}) \leq n^{-1}\tr\big((\mathbf{A}-\mathbf{B})(\mathbf{A}-\mathbf{B})^{*}\big)$,
where $L(F,G)$ denotes the L\'{e}vy distance between distribution functions $F$ and $G$.
\end{lemma}
\begin{lemma}[\citet{Bai:Silverstein:2010}, Theorem A.45]
\label{lemma:Levy_norm_inequality}
Let $\mathbf{A}$ and $\mathbf{B}$ be two $p\times p$ Hermitian matrices. Then, $L(F^\mathbf{A}, F^\mathbf{B})\leq \|\mathbf{A}-\mathbf{B}\|$.
\end{lemma}
\begin{lemma}[\citet{Geronimo:Hill:2003}, Lemma 3]
\label{lem:Uniform_convergence_analytic_limit}
Let $\mathcal{F}$ be a family of functions analytic in an open connected set $\mathcal{D}$.
If for each compact set $K$ in $\mathcal{D}$ there is a constant $M(K)$ such that
\begin{equation}
|f(z)|\leq M(K)~~ \text{for all}~f\in \mathcal{F}~\text{and}~z\in K,
\end{equation}
then every sequence in $\mathcal{F}$ has a subsequence that converges uniformly on compact subsets
of $\mathcal{D}$ to a function analytic in $\mathcal{D}$.
\end{lemma}
\begin{lemma}[Arzela--Ascoli]
\label{lemma:Arzela_Ascoli_theorem}
A sequence of continuous functions on a compact support converges uniformly to a continuous function if
they are equicontinuous and converge pointwise on a dense subset of the support.
\end{lemma}
\begin{lemma}[\citet{Liu:Aue:Paul:2015}, Lemma S.13]
\label{lemma:prob_measure_convergence}
Suppose that $(P_{n})$ is a tight sequence of Borel probability measures with corresponding
Stieltjes transforms $(s_{n}(z))$. If $s_{n}(z)\to s(z)$ for all $z\in\mC^{+}$, then
$\lim_{v\to\infty}\mathbf{i}vs(\mathbf{i}v)=-1$ and thus $s(z)$ is a Stieltjes transform of a Borel probability measure.
\end{lemma}
\begin{lemma}[\citet{Geronimo:Hill:2003}, Theorem 1]
\label{lemma:sufficient_nessessary_condition_limit_stieltjes}
Suppose that $(P_{n})$ are real Borel probability measures (with mass 1) with corresponding Stieltjes
transforms $(s_{n}(z))$. If\/ $\lim_{n\to\infty}s_{n}(z)=s(z)$ for all $z$ with $\Im(z)>0$, then there
exists a Borel probability measure $P$ with Stieltjes transform $s_{P}=s$ if and only if
\begin{equation}
\lim_{v\to \infty}\mathbf{i}vs(\mathbf{i}v)=-1
\end{equation}
in which case $P_n\to P$ in distribution.
\end{lemma}

\section{Proof of Corollary \ref{cor:LSD_linear_process_B}}
\label{sec:proof_cor_linear_process}

In view of Lemma \ref{lemma:rank_inequality_rectangular} and a truncation argument analogous to
that in Section \ref{sec:proof:lin_proc}, without loss of generality,
attention can be restricted to the matrix
\[
\mathbf{C}_{\tau}^{B} = \sqrt{\frac np}\bigg(\bar{\mathbf{S}}_{\tau}^{B}
- \frac{1}{2}\sum_{\ell=0}^{q_p-\tau} (\mathbf{B}_\ell \mathbf{B}_{\ell+\tau}^*+\mathbf{B}_{\ell+\tau} \mathbf{B}_\ell^*)\bigg),
\]
where $\bar{\mathbf{S}}_{\tau}^{B} = \frac{1}{n-\tau}\bar{\mathbf{X}}^{B}\mathbf{D}_\tau(\bar{\mathbf{X}}^{B})^*$ with $\mathbf{D}_\tau = \frac{1}{2}(\mathbf{L}^\tau + (\mathbf{L}^{\tau})^T)$,
$\bar{\mathbf{X}}^{B} = \sum_{\ell=0}^{q_p} \mathbf{B}_\ell \mathbf{Z} \tilde{\mathbf{L}}^\ell$
and $q_p = \lceil p^{1/4}\rceil \leq p^{1/\beta}$ since $\beta \in [0,4)$. Define $\bar{\mathbf{X}}^{A} = \sum_{\ell=0}^{q_p} \mathbf{A}_\ell \mathbf{Z} \tilde{\mathbf{L}}^\ell$ and $\mathbf{C}_{\tau}^{A}=\sqrt{n/p}(\bar{\mathbf{S}}_{\tau}^A
- \sum_{\ell=0}^{q_p-\tau} \mathbf{A}_\ell \mathbf{A}_{\ell+\tau})$ where $\bar{\mathbf{S}}_{\tau}^A
= \frac{1}{n-\tau}\bar{\mathbf{X}}^A\mathbf{D}_\tau(\bar{\mathbf{X}}^A)^*$. It suffices to show that
the distance between the ESDs of $\mathbf{C}_\tau^{B}$ and $\mathbf{C}_\tau^{A}$ converge to zero almost surely under conditions
{\bf B3} or {\bf B4} and {\bf B1}, {\bf B2}, {\bf A1}--{\bf A5}.

First, to prove the result under condition {\bf B3}, by Lemma \ref{lemma:rank_inequality},
it suffices to show that
\begin{equation}\label{eq:rank_diff_C_asymmetric}
\frac{1}{p}\mbox{rank}\left(\mathbf{C}_{\tau}^{B} - \mathbf{C}_{\tau}^{A} \right) \to 0 ~~\mbox{a.s.}
\end{equation}
In this direction, first note that,
\begin{align*}
\frac{1}{p}\mbox{rank}(\bar{\mathbf{S}}_{\tau}^{B} - \bar{\mathbf{S}}_{\tau}^{A})
\leq& \frac{2}{p}\mbox{rank}(\bar{\mathbf{X}}^B - \bar{\mathbf{X}}^A) \\
\leq& \frac{2}{p}\sum_{\ell=0}^{q_p} \mbox{rank}(\mathbf{B}_\ell-\mathbf{A}_\ell) \leq
\frac{2}{p}\sum_{\ell=0}^{\lceil p^{1/\beta}\rceil} \mbox{rank}(\mathbf{B}_\ell-\mathbf{A}_\ell) ~\to~0,
\end{align*}
where the last condition is by {\bf B3}. Also,
\begin{align*}
& \frac{1}{p} \mbox{rank}\left(\frac{1}{2}\sum_{\ell=0}^{q_p-\tau} (\mathbf{B}_\ell \mathbf{B}_{\ell+\tau}^*
+\mathbf{B}_{\ell+\tau} \mathbf{B}_\ell^*) - \sum_{\ell=0}^{q_p - \tau} \mathbf{A}_\ell \mathbf{A}_{\ell +\tau}\right) \\
&\leq \frac{2}{p}  \mbox{rank}(\sum_{\ell=0}^{q_p - \tau} (\mathbf{B}_\ell - \mathbf{A}_\ell)\mathbf{B}_{\ell+\tau}^*
+ \sum_{\ell=0}^{q_p - \tau} \mathbf{A}_\ell (\mathbf{B}_{\ell+\tau} - \mathbf{A}_{\ell+\tau})^*)
\leq \frac{4}{p}\sum_{\ell=0}^{q_p} \mbox{rank}(\mathbf{B}_\ell-\mathbf{A}_\ell) \to 0.
\end{align*}
Combining the last two displays, (\ref{eq:rank_diff_C_asymmetric}) follows.

Now, to prove the result under {\bf B4}, by Lemma \ref{lemma:Levy_norm_inequality},
it suffices to show that
\begin{equation}\label{eq:norm_diff_C_asymmetric}
\| \mathbf{C}_{\tau}^{B} - \mathbf{C}_{\tau}^{A}\| \to 0 ~~\mbox{a.s.}
\end{equation}
where $\| \cdot \|_F$ denotes the Frobenius norm. As a first step, note that since
$\| \mathbf{L}\| \leq 1$, where $\mathbf{L}$ is the lag operator,
\begin{align*}
\sqrt{\frac{n}{p}} \| \bar{\mathbf{S}}_{\tau}^{B} - \bar{\mathbf{S}}_{\tau}^A \|
\leq&  \frac{\sqrt{n}}{\sqrt{p}(n-\tau)} \max\{\| \bar{\mathbf{X}}^{B} \|, \| \bar{\mathbf{X}}^{A}\|\}
\| \bar{\mathbf{X}}^{B} -  \bar{\mathbf{X}}^{A}\| \\
\leq& \frac{\sqrt{n}}{\sqrt{p}(n-\tau)}\|\mathbf{Z}\|^2\sum_{\ell=0}^{q_p} (\| \mathbf{B}_\ell \|+\|\mathbf{A}_\ell\|)
\sum_{\ell=0}^{q_p}  \|\mathbf{B}_\ell - \mathbf{A}_\ell\| \\
\leq& \frac{n}{n-\tau} \|\frac{1}{n}  \mathbf{Z}\mathbf{Z}^*\| \left(\sum_{\ell=0}^\infty \bar{b}_\ell + \sum_{\ell=0}^\infty \bar{a}_\ell\right)
 \sqrt{\frac{n}{p}}  \sum_{\ell=0}^{\lceil p^{1/\beta}\rceil}  \|\mathbf{B}_\ell - \mathbf{A}_\ell\| ~\to~0 ~~\mbox{a.s.}
\end{align*}
Here the last line follows from assumptions {\bf A4}, {\bf A5}, {\bf B1},
the fact that $\|\frac{1}{n} \mathbf{Z}\mathbf{Z}^*\| \leq 1+\epsilon$ a.s.
for large $n$, for any given $\epsilon > 0$, and assumption {\bf B4}. Next,
\begin{align*}
 \sqrt{\frac{n}{p}} &\|\frac{1}{2}\sum_{\ell=0}^{q_p-\tau} (\mathbf{B}_\ell \mathbf{B}_{\ell+\tau}^*+\mathbf{B}_{\ell+\tau} \mathbf{B}_\ell^*) - \sum_{\ell=0}^{q_p - \tau} \mathbf{A}_\ell \mathbf{A}_{\ell +\tau} \| \\
&\leq (\max_{\ell}\|\mathbf{B}_\ell\| + \max_{\ell}\|\mathbf{A}_\ell\|) \sqrt{\frac{n}{p}}
\sum_{\ell=0}^{q_p}  \|\mathbf{B}_\ell - \mathbf{A}_\ell\| \to 0,
\end{align*}
again by {\bf A4}, {\bf A5} {\bf B1} and {\bf B4}. Combining the last two displays, (\ref{eq:norm_diff_C_asymmetric}) is obtained.



\section{Proof of Lemma \ref{lem:tightness}}\label{sec:proof_lemma_tightness}

In view of Lemma \ref{lem:aux:2} and the truncation arguments in
Sections \ref{sec:proof:lin_proc} and \ref{sec:proof:non-gauss}, it suffices
to show that $(F^{\bar{\mathbf{C}}_\tau})$ is a tight sequence,
where $\bar{\mathbf{C}}_\tau = \sqrt{n/p}(\bar{\mathbf{S}}_\tau - \mathbb{E}[\bar{\mathbf{S}}_\tau])$ and
$\bar{\mathbf{S}}_\tau = \frac{1}{n} \bar{\mathbf{X}} \tilde{\mathbf{D}}_\tau \bar{\mathbf{X}}^*$ where $\bar{\mathbf{X}}
= \sum_{\ell=0}^{q_p} \mathbf{A}_\ell \mathbf{Z} \tilde{\mathbf{L}}^\ell$ and $\tilde{\mathbf{D}}_\tau
= \frac{1}{2}(\tilde{\mathbf{L}}^\tau + \tilde{\mathbf{L}}^{-\tau})$, with $q_p = \lceil p^{1/4}\rceil$ and the
$Z_{tj}$ satisfying {\bf Z1}, {\bf Z2} and $|Z_{tj}|\leq n^{1/4}\epsilon_p$ where $\epsilon_p > 0$ is
such that $\epsilon_p \to 0$ and $p^{1/4}\epsilon_p \to \infty$.
This is established by showing that $\frac{1}{p} \tr(\bar{\mathbf{C}}_\tau^2) = \int x^2 dF^{\bar{\mathbf{C}}_\tau}$ is bounded almost surely.
which in turn is shown by verifying that $p^{-1}\mathbb{E}[\tr(\bar{\mathbf{C}}_\tau^2)]$ is bounded from above and
\begin{equation}\label{eq:sum_var_trace_C_tau_2}
\sum_{p\geq 1} \mathbb{E}\bigg[\bigg(\frac{1}{p} \tr(\bar{\mathbf{C}}_\tau^2) - \frac{1}{p} \mathbb{E}\big[\tr(\bar{\mathbf{C}}_\tau^2)\big]\bigg)^2\bigg]
< \infty,
\end{equation}
the result whereby follows from the Borel-Cantelli lemma.

Define, $\mathbf{E}^{\ell\ell'} = \tilde{\mathbf{L}}^\ell \tilde{\mathbf{D}}_\tau \tilde{\mathbf{L}}^{-\ell}$,
$\mathbf{U}^{\ell\ell'} = \mathbf{Z}\mathbf{E}^{\ell\ell'}\mathbf{Z}^*- \tr(\mathbf{E}^{\ell\ell'})\mathbf{I}_p$
and $\mathbf{G}^{\ell\ell'} = \mathbf{A}_\ell\mathbf{A}_{\ell'}$. Observe that
\begin{equation}\label{eq:trace_E_ell_ell_prime}
\tr(\mathbf{E}^{\ell\ell'})
= \frac{1}{2}\tr(\tilde{\mathbf{L}}^{\ell-\ell'+\tau}) + \frac{1}{2}\tr(\tilde{\mathbf{L}}^{\ell-\ell'-\tau})
= \frac{n}{2}(\delta_0(\ell'-\ell-\tau) + \delta_0(\ell-\ell'-\tau)).
\end{equation}
Then,
\[
\bar{\C}_\tau
= \frac{1}{\sqrt{np}}\sum_{\ell_1=0}^{q_p}\sum_{\ell_2=0}^{q_p}
 \mathbf{A}_{\ell_1} (\mathbf{Z}\mathbf{E}^{\ell_1\ell_2}\mathbf{Z}^* - \mathbb{E}[\mathbf{Z}\mathbf{E}^{\ell_1\ell_2}\mathbf{Z}^*])\mathbf{A}_{\ell_2}
= \frac{1}{\sqrt{np}} \sum_{\ell_1=0}^{q_p}\sum_{\ell_2=0}^{q_p} \mathbf{A}_{\ell_1}\mathbf{U}^{\ell_1\ell_2}\mathbf{A}_{\ell_2}
\]
and hence
\begin{align}
\label{eq:trace_C_bar_tau_2}
\frac{1}{p} \tr(\bar{\C}_\tau^2) =& \frac{1}{np^2}\sum_{\ell_1=0}^{q_p}\sum_{\ell_2=0}^{q_p}\sum_{\ell_3=0}^{q_p}\sum_{\ell_4=0}^{q_p}
\tr(\mathbf{U}^{\ell_1\ell_2} \mathbf{A}_{\ell_2}\mathbf{A}_{\ell_3} \mathbf{U}^{\ell_3\ell_4} \mathbf{A}_{\ell_4}\mathbf{A}_{\ell_1})
\nonumber\\
=& \frac{1}{np^2}\sum_{\ell_1=0}^{q_p}\sum_{\ell_2=0}^{q_p}\sum_{\ell_3=0}^{q_p}\sum_{\ell_4=0}^{q_p}
\sum_{i_1=1}^p\sum_{i_2=1}^p\sum_{i_3=1}^p\sum_{i_4=1}^p \mathbf{U}_{i_1i_2}^{\ell_1\ell_2} \mathbf{G}_{i_2i_3}^{\ell_2\ell_3}
\mathbf{U}_{i_3i_4}^{\ell_3\ell_4} \mathbf{G}_{i_4i_1}^{\ell_4\ell_1}.
\end{align}
Let the $i$-th row of $\mathbf{Z}$, written as an $1\times n$ vector,
be denoted by $\mathbf{z}_i^* = (Z_{i1},\cdots,Z_{in}) = (\mathbf{z}_i^R)^T + \mathbf{i}(\mathbf{z}_i^I)^T$
where $\mathbf{z}_i^R = (\Re(Z_{i1}),\ldots,\Re(Z_{in}))$ and $\mathbf{z}_i^I = (\Im(Z_{i1}),\ldots,\Im(Z_{in}))$.
Thus, $\mathbf{z}_i = \mathbf{z}_i^R - \mathbf{i} \mathbf{z}_i^I$, and hence
\begin{align}
\label{eq:U_ell_ell_prime_i_i_prime}
\mathbf{U}_{ii'}^{\ell\ell'} =& \mathbf{z}_i^* \mathbf{E}^{\ell\ell'}\mathbf{z}_{i'} - \tr(\mathbf{E}^{\ell\ell'}) \delta_0(i-i')\nonumber\\
=& \left((\mathbf{z}_i^R)^T \mathbf{E}^{\ell\ell'} \mathbf{z}_{i'}^R - \frac{1}{2}\tr(\mathbf{E}^{\ell\ell'}) \delta_0(i-i') \right)
+ \left((\mathbf{z}_i^I)^T \mathbf{E}^{\ell\ell'} \mathbf{z}_{i'}^I - \frac{1}{2}\tr(\mathbf{E}^{\ell\ell'}) \delta_0(i-i')\right) \\
&
- \mathbf{i} (\mathbf{z}_i^R)^T \mathbf{E}^{\ell\ell'} \mathbf{z}_{i'}^I + \mathbf{i}  (\mathbf{z}_{i}^I)^T
\mathbf{E}^{\ell\ell'} \mathbf{z}_{i'}^R\nonumber\\
=& u_{ii'}^{\ell\ell'} + v_{ii'}^{\ell\ell'} - \mathbf{i} x_{ii'}^{\ell\ell'} + \mathbf{i} y_{ii'}^{\ell\ell'},
\end{align}
say. Notice that $x_{ii'}^{\ell\ell'} = y_{i'i}^{\ell\ell'}$. Further,
expectation of each of the terms in the last line of (\ref{eq:U_ell_ell_prime_i_i_prime}) is zero, which follows from
\begin{equation}\label{eq:z_i_quad_expectation}
\mathbb{E}[\mathbf{z}_i^R(\mathbf{z}_i^R)^T] = \mathbb{E}[\mathbf{z}_i^I(\mathbf{z}_i^I)^T]
= \frac{1}{2}\mathbf{I}_p,
\qquad \mathbb{E}[\mathbf{z}_i^R(\mathbf{z}_i^I)^T] = \mathbf{0}_{p\times p}
\end{equation}
and the independence of $\mathbf{z}_i$. It can then deduced that
\[
\mathbb{E}[\mathbf{U}_{i_1i_2}^{\ell_1\ell_2}\mathbf{U}_{i_1i_2}^{\ell_3\ell_4}] = 0
\qquad\mbox{and}\qquad
\mathbb{E}[\mathbf{U}_{i_1i_2}^{\ell_1\ell_2}\mathbf{U}_{i_2i_1}^{\ell_3\ell_4}]
= \tr(\mathbf{E}^{\ell_1\ell_2}\mathbf{E}^{\ell_3\ell_4}),
\qquad i_1\neq i_2,
\]
while
\[
\mathbb{E}[\mathbf{U}_{ii}^{\ell_1\ell_2}\mathbf{U}_{ii}^{\ell_3\ell_4}]
= \tr(\mathbf{E}^{\ell_1\ell_2}\mathbf{E}^{\ell_3\ell_4}) + (\mu_4-2)\sum_{j=1}^n
\mathbf{E}_{jj}^{\ell_1\ell_2}\mathbf{E}_{jj}^{\ell_3\ell_4},
\]
where $\mu_4 = \mathbb{E}|Z_{11}|^4$.
Note also that
$\sum_{j=1}^n \mathbf{E}_{jj}^{\ell_1\ell_2}\mathbf{E}_{jj}^{\ell_3\ell_4}$ is zero except when either
$|\ell_1-\ell_2|=\tau$  or $|\ell_3-\ell_4|=\tau$.

Now $\frac{1}{p} \mathbb{E}[\tr(\bar{\C}_\tau^2)]$ can be computed. Recalling that
that $\mathbb{E}[\mathbf{U}_{ii'}^{\ell\ell'}] = 0$, and using (\ref{eq:trace_C_bar_tau_2})
and (\ref{eq:z_i_quad_expectation}), and independence of $\mathbf{z}_i$, it follows that
\begin{align*}
& \frac{1}{p} \mathbb{E}[\tr(\bar{\C}_\tau^2)] \\
&= \frac{1}{np^2}
\sum_{\ell_1=0}^{q_p}\sum_{\ell_2=0}^{q_p}\sum_{\ell_3=0}^{q_p}\sum_{\ell_4=0}^{q_p}
\sum_{i_1=1}^p \sum_{i_2\neq i_1}^p \left( \mathbb{E}[\mathbf{U}_{i_1i_2}^{\ell_1\ell_2}\mathbf{U}_{i_1i_2}^{\ell_3\ell_4}]
\mathbf{G}_{i_2i_1}^{\ell_2\ell_3}\mathbf{G}_{i_2i_1}^{\ell_4\ell_1} +
\mathbb{E}[\mathbf{U}_{i_1i_2}^{\ell_1\ell_2}\mathbf{U}_{i_2i_1}^{\ell_3\ell_4}] \mathbf{G}_{i_2i_2}^{\ell_2\ell_3}\mathbf{G}_{i_1i_1}^{\ell_4\ell_1}\right) \\
&\quad + \sum_{\ell_1=0}^{q_p}\sum_{\ell_2=0}^{q_p}\sum_{\ell_3=0}^{q_p}\sum_{\ell_4=0}^{q_p}
\sum_{i=1}^p \mathbb{E}[\mathbf{U}_{ii}^{\ell_1\ell_2}\mathbf{U}_{ii}^{\ell_3\ell_4}]
\mathbf{G}_{ii}^{\ell_2\ell_3}\mathbf{G}_{ii}^{\ell_4\ell_1}.
\end{align*}
From this, (\ref{eq:trace_E_ell_ell_prime}), the calculations above and recalling {\bf A4} and {\bf A5}, it follows that
$\frac{1}{p}\mathbb{E}[\tr(\bar{\C}_\tau^2)]$ is bounded from above.

A more involved calculation, involving the computation of
$\mathbb{E}[\mathbf{U}_{i_1i_2}^{\ell_1\ell_2}\mathbf{U}_{i_3i_4}^{\ell_3\ell_4}\mathbf{U}_{i_5i_6}^{\ell_5\ell_6}\mathbf{U}_{i_7i_8}^{\ell_7\ell_8}]$,
where the indices $i_k$ are paired, with several applications of Lemma \ref{lemma:moments_of_quadratic_forms}
(for dealing with terms where the same index $i_k$ appears at least six times),
proves that $\mathbb{E}[(\frac{1}{p}\tr({\C}_\tau^2) - \frac{1}{p}\mathbb{E}(\tr({\C}_\tau^2)))^2]
\leq M/p^2$ for some finite constant $M$, which implies (\ref{eq:sum_var_trace_C_tau_2}) and completes the proof.
Detailed calculations are omitted due to space constraints.




\section{Auxiliary results for Section \ref{proof:thm_maq:det_eq}}
\label{sec:aux:deterministic}


\subsection{Estimation of $\mE[|\eta_{k}|]$}
\label{subsec:aux:deterministic:1}

In order to verify $\mE[|\eta_{k}|]\to 0$, it suffices to show that $\mE[|\eta_{k}|^2]\to 0$, since $(\mE[|\eta_{k}|])^2 \leq \mE[|\eta_{k}|^2]$. Indeed,
\begin{align*}
\mE[|\eta_{k}|^2]
&= \mE\bigg[\bigg|\sqrt{\frac{n}{p}}(v_{k}^{*}\mathbf{\Delta}_{k}v_k-\tilde\sigma_{\tau, k})\bigg|^{2}\bigg] \\
&= \frac{1}{np} \mE\bigg[\bigg|\sum_{t=1}^{n}\cos(\tau \nu_{t})\psi(\bs{\alpha}_{k}, \nu_{t})|\tilde{Z}_{kt}|^2
- \sum_{t=1}^{n}\cos(\tau \nu_{t})\psi(\bs{\alpha}_{k}, \nu_{t})\bigg|^{2}\bigg] \\
&= \frac{1}{np}\mE\bigg[\bigg\{\bigg(\sum_{t=1}^{n}\cos(\tau \nu_{t})\psi(\bs{\alpha}_{k}, \nu_{t})\bigg)
(|\tilde{Z}_{kt}|^2-1)\bigg\}^2\bigg] \\
&=\frac{1}{np}\mE \bigg[\sum_{t=1}^{n}\sum_{t^\prime=1}^{n}\cos(\tau \nu_{t})\cos(\tau \nu_{t^\prime})\psi(\bs{\alpha}_{k}, \nu_{t})\psi(\bs{\alpha}_{k}, \nu_{t^\prime})(|\tilde{Z}_{kt}|^2-1)(|\tilde{Z}_{kt^\prime}|^2-1)\bigg] \\
&=\frac{1}{np} \sum_{t=1}^{n}\cos^2(\tau \nu_{t})\psi^2(\bs{\alpha}_{k}, \nu_{t}) \mE\big[|\tilde{Z}_{kt}|^2-1\big]^2 \\
&=\frac{\mu_4-1}{np} \sum_{t=1}^{n}\cos^2(\tau \nu_{t})\psi^2(\bs{\alpha}_{k}, \nu_{t}) \\
&\leq\frac{C}{p}
\end{align*}
for some constant $C>0$, where $\mu_4=E[|\tilde Z_{kt}|^4]$.


\subsection{Estimation of $\max_kd_k^{2,2}$}
\label{subsec:estimation_on_d_4}

Let $\mathbf{J}_{k}=\frac{1}{p} \tr(\mathbf{R}_{(k)}(z)\mathbf{\Gamma}_{\tau,k}^n(\bs \alpha_{k}))-\mE[p^{-1} \tr(\mathbf{R}_{(k)}(z)\mathbf{\Gamma}_{\tau,k}^n(\bs \alpha_{k}))]$. It follows from \eqref{eq:mc-diarmid-2}, that, for all $y>0$,
\[
\mP(|\mathbf{J}_{k}|^2>y)\leq 4 \exp\bigg(-\frac{c_0py}{q^2}\bigg),
\]
for some $c_0 > 0$. Thus, setting $\tilde{c}_0 = c_0 pq^{-2}$,
\begin{align*}
\mE[|\mathbf{J}_{k}|^2]
&= \int y\mP(|\mathbf{J}_{k}|^2>y)dy \\
&\leq\frac{4}{\tilde{c}_0^2}\int_{0}^{\infty} \tilde{c}_0^2 y \exp(-\tilde{c}_0y)dy = \frac{4}{\tilde{c}_0^2}\Gamma(2) = \frac{4q^4}{c_0^2p^2},
\end{align*}
where $\Gamma(\cdot)$ is the Gamma function. The right-hand side goes to zero if $p\to\infty$. This continues to hold even if the the MA order $q$ grows at a rate satisfying $q^2 = o(p)$. Consequently, $\max_k d_k^{2,2}\to 0$ under \eqref{eq:mod_dim}, as required.


\subsection{Estimation of $\max_k d_k^{2,3}$}
\label{subsec:bound_on_d_31}

Throughout this subsection the following fact is repeatedly used:
\[
|\psi( \mathbf{a}, \nu_{t})|
=\bigg|\sum_{\ell,\ell^\prime=0}^{\infty} f_\ell(\mathbf{a})f_{\ell^\prime}(\mathbf{a})e^{\ic (\ell-\ell^\prime)\nu_t}\bigg|
\leq \sum_{\ell=0}^{\infty}\bar{a}_{\ell}\sum_{\ell^{\prime}=0}^{\infty}\bar{a}_{\ell^{\prime}}
\leq L^2_{1},
\]
which holds, since $|f_{\ell}(\mathbf{a})|\leq \|\mathbf{A}_{\ell}\|\leq \bar{a}_{\ell}$ and $\sum_{\ell=0}^{\infty}\bar{a}_{\ell}\leq L_1$ by assumptions {\bf A4} and {\bf A5}.

Let $w_k=\sqrt{n/p}\mathbf{V}_{k}\Delta_{\tau}v_k$ and recall that $\mathbf{V}_{k}^*=[v_1,\ldots,v_{k-1},0, v_{k+1}, \ldots,v_p]$ and $v_k = n^{-1/2}\mathcal{G}_k\tilde Z_k$, where $\mathcal{G}_k = \mathrm{diag}(g(\bs \alpha_{k}, \nu_{t})\colon t=1,\ldots,n)$ and $\tilde{Z}_k$ is the $n\times 1$ column vector with entries $\tilde{Z}_{kt}$. Thus,
\begin{align*}
w_k^*\mathbf{R}_{(k)}(z)w_k
&=\frac{n}{p}v_k^*\Delta_{\tau}\mathbf{V}_{k}^{*}\mathbf{R}_{(k)}(z)\mathbf{V}_{k}\Delta_{\tau}v_k \\
&=\frac{n}{p}\tr\left(\mathbf{R}_{(k)}(z)\mathbf{V}_k\Delta_{\tau}v_kv_k^*\Delta_{\tau} \mathbf{V}_k^*\right) \\
&=\frac{1}{p}\tr\left(\mathbf{R}_{(k)}(z)\mathbf{V}_k\Delta_{\tau}\mathbf{G}_{k}\Delta_{\tau}\mathbf{V}_k^*\right)
+ R_{k}^{(1)},
\end{align*}
where
\begin{align}
\label{eq:R_k_1}
R_{k}^{(1)}
&=\frac{1}{p}\tr\left(\Delta_\tau\mathbf{V}_k^*\mathbf{R}_{(k)}(z)\mathbf{V}_k\Delta_\tau(nv_kv_k^*-\mathbf{G}_k)\right) \nonumber \\
&= \frac{1}{p}\sum_{j \neq k} (\mathbf{R}_{(k)}(z))_{jj} v_j^* \Delta_\tau \mathbf{U}_k \Delta_\tau v_j,
\end{align}
and $\mathbf{U}_k = n v_kv_k^* - \mathbf{G}_k$ and $\mathbf{G}_{k} = n\mathbf{\Sigma}_{k,v} = n \mathbb{E}[v_k v_k^*] = \mathrm{diag}(\psi(\bs\alpha_k,\nu_t)\colon t=1,\ldots,n)$. Define the $(j,j^\prime)$th element of $\mathbf{Q}_{k}=\mathbf{V}_{k}\Delta_{\tau}\mathbf{G}_{k}\Delta_{\tau}\mathbf{V}_{k}^{*}$ as $\mathbf{Q}_{k}(j,j^\prime)=v_j^*\Delta_{\tau}\mathbf{G}_{k}\Delta_{\tau} v_{j^\prime}$, and notice that $\mathbf{Q}_{k}(k,k)=0$. Then,
\begin{align}
\label{eq:R_k_decomp}
R_k
&= w_k^* \mathbf{R}_{(k)}(z) w_k - \frac{1}{p} \tr(\mathbf{R}_{(k)}(z) \mathbf{\Gamma}_{\tau,k}^n(\bs\alpha_k)) \nonumber \\
&= \frac{1}{p}\tr\left(\mathbf{R}_{(k)}(z)\left\{\mathbf{Q}_{k}-\mathbf{\Gamma}_{\tau,k}^n(\bs{\alpha_{k}})\right\}\right) + R_k^{(1)}\nonumber\\
&= R_k^{(2)} + R_k^{(1)},
\end{align}
where
\begin{align}
\label{eq:R_k_2_decomp}
R_k^{(2)}
&=\frac{1}{p}\tr\left(\mathbf{R}_{(k)}(z)\left\{\mathbf{Q}_{k}-\mathbf{\Gamma}_{\tau,k}^n(\bs{\alpha_{k}})\right\}\right) \nonumber \\
&=\frac{1}{p}\sum_{j\neq k} (\mathbf{R}_{(k)}(z))_{jj}
(v_j^{*}\mathbf{\Delta}_{\tau}\mathbf{G}_{k}\mathbf{\Delta}_{\tau}v_j - \gamma_{\tau, k}^n(\bs \alpha_{j}))
+ \frac{1}{p}\sum_{j\ne j^\prime \ne k}(\mathbf{R}_{(k)}(z))_{j^\prime j}(v_j^*\mathbf{\Delta}_{\tau}\mathbf{G}_{k}\mathbf{\Delta}_{\tau}v_{j^\prime}) \nonumber \\
&= R_k^{2,1} + R_k^{2,2}
\end{align}
with
\begin{align*}
R_k^{2,1}
&=\frac{1}{p}\sum_{j\ne k}\left(\mathbf{R}_{(k)}(z)\right)_{jj}
\bigg(\frac{1}{n}\sum_{t=1}^{n}\cos^2(\tau\nu_t)\psi(\bs \alpha_{k}, \nu_t)\psi(\bs \alpha_{j},\nu_t)(|\tilde{Z}_{jt}|^2-1)\bigg), \\
R_k^{2,2}
&=\frac{1}{p}\sum_{j\ne j^\prime\ne k} \left(\mathbf{R}_{(k)}(z)\right)_{j^\prime j}
\bigg(\frac{1}{n}\sum_{t=1}^{n}\cos^2(\tau\nu_t)\psi(\bs \alpha_{k}, \nu_t)g(\bs \alpha_{j^\prime}, \nu_t)\overline{g(\bs \alpha_{j}, \nu_t)}\tilde{Z}_{j^\prime t}\overline{\widetilde{Z}}_{jt}\bigg).
\end{align*}
Using independence of $\tilde{Z}_{j1},\ldots,\tilde{Z}_{jn}$, $\mathbb{E}[\tilde{Z}_{jt}]=0$ and $\mathbb{E}[|\tilde{Z}_{jt}|^2]=1$, it follows that
\begin{equation}
\label{eq:E_R_k_21_bound}
\mathbb{E}\big[|R_k^{2,1}|^2\big]
= \frac{1}{n^2p^2}\sum_{j\neq k} |(\mathbf{R}_{(k)}(z))_{jj}|^2 \sum_{t=1}^n (c_{\tau,jk}(\nu_t))^2
\mathbb{E}\big[(|\tilde{Z}_{it}|^2-1)^2\big]
\leq \frac{C}{\Im(z)^2} \frac{1}{np},
\end{equation}
where $c_{\tau,jk}(\nu) = \cos^2(\tau\nu)\psi(\bs \alpha_{k}, \nu)\psi(\bs \alpha_{j},\nu)$ and the inequality results from $\| \mathbf{R}_{(k)}(z)\| \leq \Im(z)^{-1}$. Similarly,
\begin{align}
\label{eq:E_R_k_22_bound}
\mathbb{E}\big[|R_k^{2,2}|^2\big]
=& \frac{1}{n^2p^2}\sum_{j\neq j^\prime \neq k} |(\mathbf{R}_{(k)}(z))_{j^\prime j}|^2 \sum_{t=1}^n |c_{\tau,j^\prime jk}(\nu_t)|^2\mathbb{E}[|\tilde{Z}_{j^\prime t}|^2]\mathbb{E}[|\tilde{Z}_{jt}|^2] \nonumber\\
&+ \frac{1}{n^2p^2}\sum_{j\neq j^\prime \neq k} (\mathbf{R}_{(k)}(z))_{j^\prime j}\overline{(\mathbf{R}_{(k)}(z))_{jj^\prime}} \sum_{t=1}^n c_{\tau,j^\prime jk}(\nu_t)\overline{c_{\tau,jj^\prime k}(\nu_t)} \mathbb{E}[|\tilde{Z}_{j^\prime t}|^2]\mathbb{E}[|\overline{\tilde{Z}}_{jt}|^2] \nonumber\\
\leq& \frac{C}{\Im(z)^2} \frac{1}{n},
\end{align}
where $c_{\tau,j^\prime jk}(\nu) = \cos^2(\tau\nu)\psi(\bs \alpha_{k}, \nu_t)g(\bs \alpha_{j^\prime}, \nu)\overline{g(\bs \alpha_{j}, \nu)}$, making also use of the fact that $\mathbb{E}(\overline{\tilde Z}_{jt})^2 = 1$ since the real and imaginary parts of $\tilde Z_{jt}$ are independent $N(0,1/2)$ random variables.

Next, consider
\begin{align}
\label{eq:E_R_k_1_expansion}
\mathbb{E}[|R_k^{(1)}|^2]
=& \frac{1}{p^2} \sum_{j \neq k} |(\mathbf{R}_{(k)}(z))_{jj}|^2
\mathbb{E}\big[(v_j^*\mathbf{\Delta}_\tau \mathbf{U}_k \mathbf{\Delta}_\tau v_j)^2\big] \nonumber \\
&+ \frac{1}{p^2} \sum_{j \neq j^\prime \neq k} (\mathbf{R}_{(k)}(z))_{jj} \overline{(\mathbf{R}_{(k)}(z))_{j^\prime j^\prime}} \mathbb{E}\big[v_j^*\mathbf{\Delta}_\tau \mathbf{U}_k \mathbf{\Delta}_\tau v_j v_{j^\prime}^*\mathbf{\Delta}_\tau \mathbf{U}_k \mathbf{\Delta}_\tau v_{j^\prime}\big]
\nonumber\\
=& \frac{1}{n^2p^2} \sum_{j \neq k} |(\mathbf{R}_{(k)}(z))_{jj}|^2
\mathbb{E}\big[(\tilde Z_j^*{\mathcal{G}}_j^* \mathbf{\Delta}_\tau \mathbf{U}_k \mathbf{\Delta}_\tau {\mathcal{G}}_j\tilde Z_j)^2\big] \nonumber\\
& + \frac{1}{n^2p^2} \sum_{j \neq j^\prime \neq k} (\mathbf{R}_{(k)}(z))_{jj} \overline{(\mathbf{R}_{(k)}(z))_{j^\prime j^\prime}} \mathbb{E}\big[\tr({\mathcal{G}}_j^*\mathbf{\Delta}_\tau \mathbf{U}_k \mathbf{\Delta}_\tau {\mathcal{G}}_j)\tr({\mathcal{G}}_{j^\prime}^*\mathbf{\Delta}_\tau \mathbf{U}_k \mathbf{\Delta}_\tau {\mathcal{G}}_{j^\prime})\big].
\end{align}
Define $\mathbf{B}_{jk,\tau} = \mathcal{G}_j^* \mathbf{\Delta}_\tau \mathbf{U}_k \mathbf{\Delta}_\tau \mathcal{G}_j$ and the $(s,t)$th element of $\mathbf{B}_{jk,\tau}$ by $b_{jk,\tau}(s,t)$ for $1\leq s,t \leq n$. Observe that $\mathbf{B}_{jk,\tau}^* = \mathbf{B}_{jk,\tau}$. Also,
\begin{align*}
\tr(\mathbf{B}_{jk,\tau})
&= \tilde Z_k^* \mathcal{G}_k^*\mathbf{\Delta}_\tau \mathcal{G}_j \mathcal{G}_j^* \mathbf{\Delta}_\tau \mathcal{G}_k \tilde Z_k - \tr(\mathcal{G}_k^*\mathbf{\Delta}_\tau  \mathcal{G}_j \mathcal{G}_j^* \mathbf{\Delta}_\tau \mathcal{G}_k) \\
&= \sum_{t=1}^n c_{\tau,jk}(\nu_t) (|\tilde Z_{kt}|^2 - 1) \\
&= \sum_{t=1}^n b_{jk,\tau}(t,t).
\end{align*}
Thus, for $j\neq j^\prime\neq k$,
\begin{align}
\label{eq:E_tr_B_i_tr_B_j}
\mathbb{E}
&\big[\tr(\mathcal{G}_j^*\mathbf{\Delta}_\tau \mathbf{U}_k \mathbf{\Delta}_\tau \mathcal{G}_j)
\tr(\mathcal{G}_{j^\prime}^*\mathbf{\Delta}_\tau \mathbf{U}_k \mathbf{\Delta}_\tau \mathcal{G}_{j^\prime})\big]
\nonumber\\
&= \mathbb{E}\big[\tr(\mathbf{B}_{jk,\tau})\tr(\mathbf{B}_{j^\prime k,\tau})\big] \nonumber\\
&= \mathbb{E}\bigg[
\sum_{t=1}^n c_{\tau,jk}(\nu_t) (|\tilde Z_{kt}|^2 - 1)
\sum_{t^\prime=1}^n c_{\tau,j^\prime k}(\nu_{t^\prime}) (|\tilde Z_{kt^\prime}|^2 - 1)
\bigg] \nonumber\\
&= \sum_{t=1}^n c_{\tau,jk}(\nu_t)c_{\tau,j^\prime k}(\nu_t) \mathbb{E}\big[(|\tilde Z_{kt}|^2 - 1)^2\big]
\leq C n
\end{align}
for some $C > 0$ uniformly in $j,j^\prime,k$ using the independence of the $\tilde Z_{k1},\ldots,\tilde Z_{k,n}$ and that $\mathbb{E}[|\tilde Z_{kt}|^2] = 1$ for all $t$. Utilizing the same arguments, for $j\neq k$,
\begin{align}
\label{eq:E_Z_i_B_ik_Z_i_square}
\mathbb{E}&
\big[(\tilde Z_j^* \mathbf{B}_{jk,\tau} \tilde Z_j)^2 \big| \mathbf{U}_k\big] \nonumber\\
&= \mathbb{E}\bigg[\bigg(\sum_{t=1}^n\sum_{t^\prime=1}^n b_{jk,\tau}(t,t^\prime) \tilde Z_{jt^\prime} \overline{\tilde Z_{jt}}\bigg)^2\Big|\mathbf{U}_k\bigg] \nonumber\\
&= \sum_{t=1}^n\sum_{t^\prime=1}^n \left[b_{jk,\tau}^2(t,t^\prime) \mathbb{E}\big[(\tilde Z_{jt^\prime})^2(\overline{\tilde Z_{jt}})^2\big]
+ b_{jk,\tau}(t,t^\prime) b_{jk,\tau}(t^\prime,t)\mathbb{E}\big[|\tilde Z_{jt^\prime}|^2|\tilde Z_{jt}|^2\big]\right] \nonumber\\
&= \sum_{t=1}^n b^2_{jk,\tau}(t,t)\mathbb{E}[|\tilde Z_{jt}|^4-1]
+ \sum_{t=1}^n\sum_{t^\prime=1}^n b_{jk,\tau}(t,t^\prime) b_{jk,\tau}(t^\prime,t) \mathbb{E}[|\tilde Z_{jt^\prime}|^2]\mathbb{E}[|\tilde Z_{jt}|^2] \nonumber\\
&= (\mu_4 -1) \sum_{t=1}^n b^2_{jk,\tau}(t,t)  + \tr((\mathbf{B}_{jk,\tau})^2),
\end{align}
where $\mu_4 = \mathbb{E}[|\tilde Z_{jt}|^4]$, noting that the last step makes use of $\mathbf{B}_{jk,\tau}^* = \mathbf{B}_{jk,\tau}$. Now,
\begin{align}
\mathbb{E}\bigg[\sum_{t=1}^n (b_{jk,\tau}(t,t))^2 \bigg]
&= \sum_{t=1}^n c^2_{\tau,jk}(\nu_t) \mathbb{E}\big[(|\widetilde Z_{kt}|^2 - 1)^2\big] \nonumber\\
&= (\mu_4 - 1) \sum_{t=1}^n c^2_{\tau,jk}(\nu_t) \leq C n,
\end{align}
for some $C > 0$ uniformly in $j$ and $k$.

Finally, define $\mathbf{F}_{jk,\tau} = \mathcal{G}_k^*\mathbf{\Delta}_\tau \mathcal{G}_j\mathcal{G}_j^* \mathbf{\Delta}_\tau \mathcal{G}_k$. Observe that $\mathbf{F}_{jk,\tau}$ is a diagonal matrix with $(t,t)$th element
$c_{\tau,jk}(\nu_t)$ for $t=1,\ldots,n$. Thus,
\begin{align}
\label{eq:E_tr_B_ik_square}
\mathbb{E}[\tr((\mathbf{B}_{jk,\tau})^2)]
&= \mathbb{E}[\tr(\mathbf{U}_k \mathbf{\Delta}_\tau \mathcal{G}_j\mathcal{G}_j^* \mathbf{\Delta}_\tau \mathbf{U}_k \mathbf{\Delta}_\tau \mathcal{G}_j\mathcal{G}_j^*\mathbf{\Delta}_\tau)]
\nonumber\\
&= \mathbb{E}\left[\tr\left((\mathcal{G}_k \tilde Z_k \tilde Z_k^* \mathcal{G}_k^*
- \mathcal{G}_k \mathcal{G}_k^*)
 \mathbf{\Delta}_\tau \mathcal{G}_j\mathcal{G}_j^* \mathbf{\Delta}_\tau (\mathcal{G}_k \tilde Z_k \tilde Z_k^* \mathcal{G}_k^* - \mathcal{G}_k \mathcal{G}_k^*) \mathbf{\Delta}_\tau \mathcal{G}_j \mathcal{G}_j^* \mathbf{\Delta}_\tau \right)\right] \nonumber\\
&= \mathbb{E}\left[\left(\tilde Z_k^* \mathcal{G}_k^* \mathbf{\Delta}_\tau \mathcal{G}_j \mathcal{G}_j^* \mathbf{\Delta}_\tau \mathcal{G}_k \tilde Z_k\right)^2\right]
- \tr\left[\left( \mathcal{G}_k^* \mathbf{\Delta}_\tau \mathcal{G}_j \mathcal{G}_j^* \mathbf{\Delta}_\tau
\mathcal{G}_k\right)^2\right] \nonumber\\
&= \mathbb{E}\bigg[\bigg(\sum_{t=1}^n c_{\tau,jk}(\nu_t)|\tilde Z_{kt}|^2 \bigg)^2\bigg]
- \sum_{t=1}^n c^2_{\tau,jk}(\nu_t) \nonumber\\
&= \sum_{t=1}^n c^2_{\tau,jk}(\nu_t) \mathbb{E}[|\widetilde Z_{kt}|^4 - 1]
+ \sum_{t\neq t^\prime} c_{\tau,jk}(\nu_t)c_{\tau,jk}(\nu_{t^\prime})
\mathbb{E}[|\tilde Z_{kt}|^2] \mathbb{E}[|\tilde Z_{kt^\prime}|^2] \nonumber\\
&= (\mu_4 - 1) \sum_{t=1}^n c^2_{\tau,jk}(\nu_t) + \bigg(\sum_{t=1}^n c_{\tau,jk}(\nu_t)\bigg)^2 \nonumber\\
&\leq C n^2,
\end{align}
for some $C > 0$ uniformly in $j$ and $k$.

Combining (\ref{eq:E_R_k_1_expansion})--(\ref{eq:E_tr_B_ik_square}) leads to
\begin{equation}
\label{eq:E_R_k_1_bound}
\mathbb{E}[|R_k^{(1)}|^2] \leq \frac{C_1}{\Im(z)^2} \frac{1}{p} +  \frac{C_2}{\Im(z)^2} \frac{1}{n}
\end{equation}
for some constants $C_1, C_2 > 0$ uniformly in $k$. Thus, noticing that $d_k^{2,3} = \mathbb{E}[|R_k|^2]$ and combining \eqref{eq:R_k_decomp}--\eqref{eq:E_R_k_22_bound} and \eqref{eq:E_R_k_1_bound}, it follows that
\[
\max_k d_k^{2,3} \leq \frac{1}{\Im(z)^2} \left(\frac{C_1^\prime}{p} + \frac{C_2^\prime}{n}\right)
\]
for some $C_1^\prime,C_2^\prime > 0$, so that $\max_k d_k^{2,3}$ is asymptotically negligible.


\section{Equicontinuity of $\beta_{\tau,p}(z,\mathbf{a})$}
\label{sec:app:equicont}


In this subsection,  it is verified that $\beta_{\tau,p}(z,\mathbf{a})$ is uniformly equicontinuous in $\mathbf{a}$.
Observe that
\begin{eqnarray*}
\mathcal{R}_{\tau}(\mathbf{a}_{1}, \mathbf{b})- \mathcal{R}_{\tau}(\mathbf{a}_{2}, \mathbf{b})&=&\frac{1}{2\pi}\int_{0}^{2\pi}\cos^2(\tau\nu)\psi(\mathbf{b},\nu)(\psi(\mathbf{a}_1,\nu)-\psi(\mathbf{a}_2,\nu))d\nu\\
&=&\frac{1}{4\pi}\int_{0}^{2\pi}\psi(\mathbf{b},\nu)(\psi(\mathbf{a}_1,\nu)-\psi(\mathbf{a}_2,\nu))d\nu\\
&+& \frac{1}{4\pi}\int_{0}^{2\pi}\cos(2\tau\nu)\psi(\mathbf{b},\nu)(\psi(\mathbf{a}_1,\nu)-\psi(\mathbf{a}_2,\nu))d\nu\\
&=& K_4+K_5.
\end{eqnarray*}
Recall that, by {\bf A6}, for each $\ell \geq 1$, $f_{\ell}$ is a Lipschitz function satisfying, for any $\mathbf{a}_{1}, \mathbf{a}_{2}\in \mathbb{R}^{m_0}$,
\begin{equation}
|f_{\ell}(\mathbf{a}_{1})-f_{\ell}(\mathbf{a}_{2})|\leq C \ell^{r_0}\|\mathbf{a}_{1}-\mathbf{a}_{2}\|
\end{equation}
for some $C > 0$ and some integer $r_0 \geq 4$ as in {\bf A5}. Therefore,
\begin{align*}
K_4
&=
\frac{1}{4\pi}\!\int_{0}^{2\pi}\!\!\!
\bigg(\!
\sum_{\ell_{1}=0}^{\infty}\sum_{\ell_{1}^{\prime}=0}^{\infty}
\big[
f_{\ell_1}(\mathbf{a}_{1})f_{\ell_{1}^{\prime}}(\mathbf{a}_1)-f_{\ell_1}(\mathbf{a}_{2})f_{\ell_{1}^{\prime}}(\mathbf{a}_2)
\big]
e^{\mathbf{i}(\ell_1-\ell_1^{\prime})\nu}
\!\bigg)\!\!
\bigg(\!
\sum_{\ell_{2}=0}^{\infty}\sum_{\ell_{2}^{\prime}=0}^{\infty}
f_{\ell_{2}}(\mathbf{b})f_{\ell_{2}^{\prime}}(\mathbf{b})e^{\mathbf{i}(\ell_2-\ell_2^{\prime})\nu}
\!\bigg)
d\nu\\
&=
\frac{1}{4\pi}\int_{0}^{2\pi}
\sum_{\ell_{1}=0}^{\infty}\sum_{\ell_{1}^{\prime}=0}^{\infty}\sum_{\ell_{2}=0}^{\infty}\sum_{\ell_{2}^{\prime}=0}^{\infty}
f_{\ell_1}(\mathbf{a}_{1})\big[f_{\ell_{1}^{\prime}}(\mathbf{a}_{1})-f_{\ell_{1}^{\prime}}(\mathbf{a}_{2})\big]
f_{\ell_2}(\mathbf{b})f_{\ell_{2}^{\prime}}(\mathbf{b})e^{\mathbf{i}(\ell_1-\ell_1^{\prime}+\ell_2-\ell_2^{\prime})\nu}
d\nu \\
&\quad+\frac{1}{4\pi}\int_{0}^{2\pi}
\sum_{\ell_{1}=0}^{\infty}\sum_{\ell_{1}^{\prime}=0}^{\infty}\sum_{\ell_{2}=0}^{\infty}\sum_{\ell_{2}^{\prime}=0}^{\infty}
f_{\ell_1^{\prime}}(\mathbf{a}_{2})\big[f_{\ell_{1}}(\mathbf{a}_{1})-f_{\ell_{1}}(\mathbf{a}_{2})\big]
f_{\ell_2}(\mathbf{b})f_{\ell_{2}^{\prime}}(\mathbf{b})e^{\mathbf{i}(\ell_1-\ell_1^{\prime}+\ell_2-\ell_2^{\prime})\nu}
d\nu \\
&=
\sum_{\ell_{1}=m}^{\infty}\sum_{\ell^\prime_2=m}^{\infty}\sum_{m=0}^{\infty}
f_{\ell_1}(\mathbf{a}_{1})
\big[
f_{\ell_1-m}(\mathbf{a}_1)-f_{\ell_1-m}(\mathbf{a}_1)
\big]
f_{\ell^\prime_2}(\mathbf{b})f_{\ell^\prime_2-m}(\mathbf{b}) \\
&\quad+
\sum_{\ell_{1}=m}^{\infty}\sum_{\ell^\prime_2=m}^{\infty}\sum_{m=0}^{\infty}
f_{\ell_1-m}(\mathbf{a}_{1})\big[f_{\ell_1}(\mathbf{a}_1)-f_{\ell_1}(\mathbf{a}_1)\big]f_{\ell^\prime_2}(\mathbf{b})f_{\ell^\prime_2-m}(\mathbf{b})
\end{align*}
and
\begin{align*}
K_5
&=\frac{1}{4\pi}\int_{0}^{2\pi}\cos(2\tau\nu)
\bigg(
\sum_{\ell_{1}=0}^{\infty}\sum_{\ell_{1}^{\prime}=0}^{\infty}
\big[
f_{\ell_1}(\mathbf{a}_{1})f_{\ell_{1}^{\prime}}(\mathbf{a}_1)-f_{\ell_1}(\mathbf{a}_{2})f_{\ell_{1}^{\prime}}(\mathbf{a}_2)
\big]
e^{\mathbf{i}(\ell_1-\ell_1^{\prime})\nu}
\bigg) \\
& \qquad\qquad \times\bigg(
\sum_{\ell_{2}=0}^{\infty}\sum_{\ell_{2}^{\prime}=0}^{\infty}
f_{\ell_{2}}(\mathbf{b})f_{\ell_{2}^{\prime}}(\mathbf{b})e^{\mathbf{i}(\ell_2-\ell_2^{\ell})\nu}
\bigg)d\nu\\
&=\frac{1}{4\pi}\int_{0}^{2\pi}\cos(2\tau\nu)
\sum_{\ell_{1}=0}^{\infty}\sum_{\ell_{1}^{\prime}=0}^{\infty}\sum_{\ell_{2}=0}^{\infty}\sum_{\ell_{2}^{\prime}=0}^{\infty}
f_{\ell_1}(\mathbf{a}_{1})
\big[
f_{\ell_{1}^{\prime}}(\mathbf{a}_{1})-f_{\ell_{1}^{\prime}}(\mathbf{a}_{2})
\big]
f_{\ell_2}(\mathbf{b})f_{\ell_{2}^{\prime}}(\mathbf{b})e^{\mathbf{i}(\ell_1-\ell_1^{\prime}+\ell_2-\ell_2^{\prime})\nu}
d\nu \\
&\quad+
\frac{1}{4\pi}\int_{0}^{2\pi}\cos(2\tau\nu)
\sum_{\ell_{1}=0}^{\infty}\sum_{\ell_{1}^{\prime}=0}^{\infty}\sum_{\ell_{2}=0}^{\infty}\sum_{\ell_{2}^{\prime}=0}^{\infty}
f_{\ell_1^{\prime}}(\mathbf{a}_{2})
\big[f_{\ell_{1}}(\mathbf{a}_{1})-f_{\ell_{1}}(\mathbf{a}_{2})\big]
f_{\ell_2}(\mathbf{b})f_{\ell_{2}^{\prime}}(\mathbf{b})e^{\mathbf{i}(\ell_1-\ell_1^{\prime}+\ell_2-\ell_2^{\prime})\nu}d\nu \\
&=
\sum_{\ell_{1}=m+\tau}^{\infty}\sum_{\ell^\prime_2=m+\tau}^{\infty}\sum_{m=0}^{\infty}
f_{\ell_1}(\mathbf{a}_{1})
\big[f_{\ell_1-m-\tau}(\mathbf{a}_1)-f_{\ell_1-m-\tau}(\mathbf{a}_1)\big]
f_{\ell^\prime_2}(\mathbf{b})f_{\ell^\prime_2-m-\tau}(\mathbf{b}) \\
&\quad+
\sum_{\ell_{1}=m+\tau}^{\infty}\sum_{\ell^\prime_2=m+\tau}^{\infty}\sum_{m=0}^{\infty}
f_{\ell_1-m-\tau}(\mathbf{a}_{1})
\big[f_{\ell_1}(\mathbf{a}_1)-f_{\ell_1}(\mathbf{a}_1)\big]
f_{\ell^\prime_2}(\mathbf{b})f_{\ell^\prime_2-m-\tau}(\mathbf{b}).
\end{align*}
Therefore, by the Lipschitz properties of the $f_{\ell}$'s,
\begin{align*}
|K_4|&
\leq
\sum_{\ell_{1}=m}^{\infty}\sum_{\ell_2^\prime=m}^{\infty}\sum_{m=0}^{\infty}
|f_{\ell_1}(\mathbf{a}_{1})||f_{\ell_1-m}(\mathbf{a}_1)-f_{\ell_1-m}(\mathbf{a}_1)|
|f_{\ell_2^\prime}(\mathbf{b})||f_{\ell_2^\prime-m}(\mathbf{b})|\\
&\quad+
\sum_{\ell_{1}=m}^{\infty}\sum_{\ell_2^\prime=m}^{\infty}\sum_{m=0}^{\infty}
|f_{\ell_1-m}(\mathbf{a}_{1})||f_{\ell_1}(\mathbf{a}_1)-f_{\ell_1}(\mathbf{a}_1)|
|f_{\ell_2^\prime}(\mathbf{b})||f_{\ell_2^\prime-m}(\mathbf{b})| \\
&\leq C \sum_{\ell_{1}=m}^{\infty}\sum_{\ell_2^\prime=m}^{\infty}\sum_{m=0}^{\infty}
|\ell_1-m|^{r_0}|f_{\ell_1}(\mathbf{a}_{1})||f_{\ell_2^\prime}(\mathbf{b})|
|f_{\ell_2^\prime-m}(\mathbf{b})|\|\mathbf{a}_{1}-\mathbf{a}_2\| \\
&\quad+
C\sum_{\ell_{1}=m}^{\infty}\sum_{\ell_2^\prime=m}^{\infty}\sum_{m=0}^{\infty}
|\ell_1|^{r_0} |f_{\ell_1-m}(\mathbf{a}_{1})||f_{\ell_2^\prime}(\mathbf{b})|
|f_{\ell^\prime_2-m}(\mathbf{b})|\|\mathbf{a}_1-\mathbf{a}_2\| \\
&=K_{41}+K_{42}.
\end{align*}
Using {\bf A4} and {\bf A5}, one obtains the  bound
\begin{align*}
K_{42}&=
C\sum_{\ell_{1}=0}^{\infty}\sum_{\ell_2^\prime=0}^{\infty}\sum_{m=0}^{\infty}
|(\ell_1-m)-(\ell_2^\prime-m)+\ell_2^\prime|^{r_0} |f_{\ell_1-m}(\mathbf{a}_{1})|
|f_{\ell_2^\prime-m}(\mathbf{b})||f_{\ell^\prime_2}(\mathbf{b})|\|\mathbf{a}_{1}-\mathbf{a}_2\| \\
&\leq
C3^{r_0-1}\|\mathbf{a}_{1}-\mathbf{a}_2\|
\sum_{\ell_{1}=0}^{\infty}\sum_{\ell_2^\prime=0}^{\infty}\sum_{m=0}^{\infty}
\left[
|\ell_1-m|^{r_0}\bar{a}_{\ell_1-m}\bar{a}_{\ell_2\prime-m}\bar{a}_{\ell_2^\prime}
+|\ell_2^{\prime}-m|^{r_0}\bar{a}_{\ell_2^{\prime}-m}\bar{a}_{\ell_1-m}\bar{a}_{\ell_2^\prime}
\right.\\
& \qquad\qquad
\left.
+ |\ell_2^\prime|^{r_0} \bar{a}_{\ell_{2}^{\prime}}\bar{a}_{\ell_1-m}\bar{a}_{\ell_2^{\prime}-m}\right]\\
&\leq
C L_1^2L_{r_0+1}\|\mathbf{a}_{1}-\mathbf{a}_2\|.
\end{align*}
Similarly, it can be shown that
$K_{41}\leq CL_1^2L_{r_0+1}\|\mathbf{a}_1-\mathbf{a}_2\|$, thus implying
$K_{4}\leq  C L_1^2L_{r_0+1} \|\mathbf{a}_1-\mathbf{a}_2\|$.
Following the same steps yields also that
\begin{equation*}
K_{5}\leq  C L_1^2L_{r_0+1} \|\mathbf{a}_1-\mathbf{a}_2\|
\end{equation*}
and hence, for some constant $C_0 > 0$, and for all $\mathbf{b} \in \mathbb{R}^{m_0}$,
\begin{equation*}
|\mathcal{R}_{\tau}(\mathbf{a}_1,\mathbf{b})-\mathcal{R}(\mathbf{a}_2,\mathbf{b})|\leq C_0 L_1^2L_{r_0+1} \|\mathbf{a}_1-\mathbf{a}_2\|.
\end{equation*}
This shows that $\mathcal{R}_{\tau}(\mathbf{a},\mathbf{b})$ is Lipschitz with with respect each variable
with a bounded Lipschitz constant $C_0$.
Thus, the equiconinuity of $\beta_{\tau,p}(z,\mathbf{a})$, for any $z$ with $\Im(z) = v > 0$, follows from
\begin{align*}
|\beta_{\tau,p}(z,\mathbf{a}_1) - \beta_{\tau,p}(z,\mathbf{a}_2)|
&\leq \|(\widetilde{\mathbf{C}}_{\tau}-zI)^{-1}(\Gamma_{\tau}(\mathbf{a}_1)-\Gamma_{\tau}(\mathbf{a}_2))\|\\
&\leq \frac{C_0}{v}\|\mathbf{a}_1-\mathbf{a}_2\|,
\end{align*}
observing that $\mathbf{\Gamma}_\tau(a)=\mathrm{diag}(\mathcal{R}_\tau(\mathbf{a},\bs\alpha_j)\colon j=1,\ldots,p)$.

\section{Auxiliary results for Section \ref{sec:proof:lin_proc}}
\label{app:lin_proc}

As a first step, an inequality is derived for bounding discrete convolutions of the sequence
$(\bar{a}_\ell\colon \ell\in\mathbb{N})$ that appears in assumptions {\bf A4} and {\bf A5}.

\begin{lemma}
\label{lem:app:D1}
Let $(\bar{a}_\ell\colon \ell\in\mathbb{N})$ be as in {\bf A4} and {\bf A5} and $r_0$ be as in {\bf A5}. Then, for
$r \leq r_0$,
\begin{equation}
\label{eq:a_sum_convolution_bound}
\bigg|\sum_{\ell=k}^{\infty}\bar{a}_{\ell}\bar{a}_{\ell+u}\bigg|
\leq
\frac{2^{r+1}}{1+|u|^{r}}
\bigg(\sum_{\ell=k}^{\infty}\ell^{r}\bar{a}_{\ell}\bigg)
\bigg(\sum_{\ell=k}^{\infty}\bar{a}_{\ell}\bigg)
\end{equation}
for any $k\in\mathbb{N}_0$ and $u\in\mathbb{Z}$.
\end{lemma}

\begin{proof}
Fix $k\in\mathbb{N}_0$ and $u\in\mathbb{Z}$. Plancherel's identity and integration by parts ($r$ times) yields that
\begin{align*}
\sum_{\ell=k}^{\infty}\bar{a}_{\ell}\bar{a}_{\ell+u}
&=\frac{1}{2\pi} \int_{0}^{2\pi} \bigg|\sum_{\ell=k}^{\infty}e^{\ic \ell \theta}\bar{a}_{\ell}\bigg|^2 e^{\ic u\theta}d\theta \\
&= \frac{1}{2\pi}\int_{0}^{2\pi}\psi_{[k]}(\theta)e^{\ic u \theta}d\theta\\
&= \frac{\ic^{r}}{u^r} \frac{1}{2\pi}\int_{0}^{2\pi} \psi_{[k]}^{(r)}(\theta)e^{\ic u\theta}d\theta,
\end{align*}
where
\begin{equation}
\psi_{[k]}(\theta)
= \bigg|\sum_{\ell=k}^{\infty}e^{\ic \ell \theta}\bar{a}_{\ell}\bigg|^2
= \sum_{\ell_1=q+1}^{\infty}\sum_{\ell_2=q+1}^{\infty} e^{\ic (\ell_1-\ell_2)\theta}\bar{a}_{\ell_1}\bar{a}_{\ell_2}
\end{equation}
and, for $r\geq 0$,
\begin{equation*}
\psi_{[k]}^{(r)}(\theta)
= \ic^{r}\sum_{\ell_1=k}^{\infty}\sum_{\ell_2=k}^{\infty}(\ell_1-\ell_2)^{r} e^{\ic (\ell_1-\ell_2)\theta}\bar{a}_{\ell_1}\bar{a}_{\ell_2}.
\end{equation*}
Since
\begin{equation*}
\sup_{\theta \in [0,2\pi]}|\psi_{[k]}^{(r)}(\theta)|
\leq 2^{r-1} \sum_{\ell_1=k}^{\infty}\sum_{\ell_2=k}^{\infty}(\ell_1^{r}+\ell^{r}_2)\bar{a}_{\ell_1} \bar{a}_{\ell_2}
\leq 2^{r} \bigg(\sum_{\ell=k}^{\infty}\ell^{r}\bar{a}_{\ell}\bigg)
\bigg(\sum_{\ell=k}^{\infty}\bar{a}_{\ell}\bigg),
\end{equation*}
the assertion of the lemma follows.
\end{proof}

\subsection{Bounding $\mE[\|\mathbf{S}_{\tau, 1}-\mE [\mathbf{S}_{\tau, 1}]\|_{F}^{2}]$
and $\mE[\|\mathbf{S}_{\tau, 2}-\mE [\mathbf{S}_{\tau, 2}]\|_{F}^{2}]$}
\label{app:lin_proc:1}

Notice first that, for $i=1,2,3$, $\mE[\|\mathbf{S}_{\tau, i}-\mE [\mathbf{S}_{\tau, i}]\|_{F}^{2}]= \mE[\|\mathbf{S}_{\tau, i}\|^{2}_{F}]-\|\mE[\mathbf{S}_{\tau,i}]\|^2_{F}$. Moreover $\mE [\mathbf{S}_{\tau,1}] = (2n)^{-1}\mE[ \sum_{t=1}^{n-\tau}(\bar{X}_{t}{X^\mathrm{tr}_{t+\tau}}^*+{X^\mathrm{tr}_{t+\tau}}\bar{X}^{*}_{t})]$ and
\begin{align*}
\|\mE [\mathbf{S}_{\tau,1}]\|^2_{F}
&= \left(1-\frac{\tau}{n}\right)^2
\sum_{m'=\max\{-\tau, q+1\}}^{q-\tau}\sum_{m=\max\{-\tau, q+1\}}^{q-\tau}
\tr(\mathbf{A}_{m'+\tau}\mathbf{A}_{m+\tau}\mathbf{A}_{m}\mathbf{A}_{m'}) = 0, \\
\|\mE [\mathbf{S}_{\tau,2}]\|^2_{F}
&= \left(1-\frac{\tau}{n}\right)^2
\sum_{m'=\max\{\tau, q+1\}}^{q+\tau}\sum_{m=\max\{\tau, q+1\}}^{q+\tau}
\tr(\mathbf{A}_{m'-\tau}\mathbf{A}_{m-\tau}\mathbf{A}_{m}\mathbf{A}_{m'}).
\end{align*}
Since the arguments for bounding $\mE[\|\mathbf{S}_{\tau,2}\|^2_{F}]$ are similar, the focus is here on bounding $\mE[\|\mathbf{S}_{\tau,1}\|^2_{F}$]. The key decomposition is
\begin{align*}
\mE\big[\|\mathbf{S}_{\tau,1}\|^2_{F}\big]
=& \frac{1}{2n^2}\sum_{t=1}^{n-\tau}\sum_{s=1}^{n-\tau}
\Re\mE \big[{X^\mathrm{tr}_{t+\tau}}^{*}X^\mathrm{tr}_{s+\tau}\bar{X}^{*}_{s}\bar{X}_{t}\big]
+ \frac{1}{2n^2} \sum_{t=1}^{n-\tau}\sum_{s=1}^{n-\tau}
\Re\mE \big[{X^\mathrm{tr}_{t+\tau}}^{*}X^\mathrm{tr}_{s}\bar{X}^{*}_{s+\tau}\bar{X}_{t}\big] \\
=&\frac{1}{2n^2}
\sum_{t=1}^{n-\tau}\sum_{s=1}^{n-\tau}\sum_{\ell=0}^{q}\sum_{\ell^\prime=0}^{q}\sum_{m=q+1}^{\infty}\sum_{m^\prime=q+1}^{\infty}
\Re\mE\big[
Z^{*}_{t+\tau-\ell}\mathbf{A}_{\ell}\mathbf{A}_{\ell^\prime}Z_{s+\tau-\ell^\prime}Z_{s-m}^{*}\mathbf{A}_{m}\mathbf{A}_{m^\prime}Z_{t-m^\prime}
\big] \\
&+ \frac{1}{2n^2}
\sum_{t=1}^{n-\tau}\sum_{s=1}^{n-\tau}\sum_{\ell=0}^{q}\sum_{\ell^\prime=0}^{q}\sum_{m=q+1}^{\infty}\sum_{m^\prime=q+1}^{\infty}
\Re\mE \big[
Z^{*}_{t+\tau-\ell}\mathbf{A}_{\ell}\mathbf{A}_{\ell^\prime}Z_{s-\ell^\prime}Z_{s+\tau-m}^{*}\mathbf{A}_{m}\mathbf{A}_{m^\prime}Z_{t-m^\prime}
\big] \\
=& Q_1 + Q_2.
\end{align*}
By independence of the $Z_{ij}$, the summands in $Q_1$ and $Q_2$ are non-zero only if the indices of the $Z_{j}$ pair up, giving four types of pairs that contribute to the summands in $Q_1$, namely
\begin{enumerate}\itemsep-.2cm
\item $t+\tau-\ell = s+\tau-\ell^\prime\ne s-m =t-m^\prime$, that is, $t=s+\ell-\ell^\prime=s+m^\prime-m$ and $\ell^\prime \ne m+\tau$;
\item $t+\tau-\ell = s-m \ne s+\tau-\ell^\prime = t-m^\prime$, that is, $t=s+\ell-m-\tau =s+m^\prime-\ell^\prime+\tau$ and $\ell^\prime \ne m+\tau$;
\item $t+\tau-\ell = t-m^\prime \ne s-m = s+\tau -\ell^\prime$;
\item $t+\tau-\ell = s+\tau-\ell^\prime= s-m =t-m^\prime$, that is, $t=s+\ell-\ell^\prime=s+m^\prime-m$ and $\ell^\prime = m+\tau$.
\end{enumerate}
The corresponding terms are labeled $K_{1,1}, K_{1,2}, K_{1,3}$ and $K_{1,4}$. The individual contributions of these terms can be given as follows. First, setting $u_1=\ell-\ell^\prime=m^\prime-m$
\[
K_{1,1}
= \frac{1}{2n^2}\sum_{s=1}^{n-\tau}\sum_{\ell^\prime=0}^{q}\sum_{\stackrel{m=q+1}{m\neq\ell^\prime-\tau}}^{\infty}
\sum_{u_1=\max\{1-s, -\ell^\prime,q+1-m\}}^{\min\{n-\tau-s,q-\ell^\prime\}}
\tr(\mathbf{A}_{\ell^\prime}\mathbf{A}_{\ell^\prime+u_1})\tr(\mathbf{A}_{m}\mathbf{A}_{m+u_1}).
\]
Second, setting $u_2 = \ell-m-\tau = m^\prime-\ell^\prime+\tau$,
\begin{align*}
K_{1,2}
=& \frac{1}{2n^2}\sum_{s=1}^{n-\tau}\sum_{\ell^\prime=0}^{q}\sum_{\stackrel{m=q+1}{m\neq \ell^\prime-\tau}}^{\infty}
\sum_{u_2 = \max\{1-s,-(m+\tau),q+1+\tau-\ell^\prime\}}^{\min\{n-\tau-s,q-(m+\tau)\}} \\
& \qquad\times\tr(\mathbf{A}_{m+u_{2}+\tau}\mathbf{A}_{m}\mathbf{A}_{\ell^\prime}\mathbf{A}_{\ell^\prime+u_{2}-\tau})
\mE[\bar{Z}_{s-m,1}^2]\mE[Z_{s+\tau-\ell^\prime,1}^2].
\end{align*}
Third,
\begin{align*}
K_{1,3}
=&  \frac{1}{2n^{2}}
\sum_{s=1}^{n-\tau}\sum_{t=1}^{n-\tau}\sum_{m=\max\{-\tau, q+1\}}^{q-\tau}
\sum_{\stackrel{m^\prime=\max\{-\tau, q+1\}}{m^\prime\ne m+t-s}}^{q-\tau}
\tr(\mathbf{A}_{m^\prime+\tau}\mathbf{A}_{m^\prime}\mathbf{A}_{m+\tau}\mathbf{A}_{m}) \\
=& \frac{1}{2n^{2}}\sum_{s=1}^{n-\tau}\sum_{t=1}^{n-\tau}
\sum_{m=\max\{-\tau, q+1\}}^{q-\tau}\sum_{m^\prime=\max\{-\tau, q+1\}}^{q-\tau}
\tr(\mathbf{A}_{m^\prime+\tau}\mathbf{A}_{m+\tau}\mathbf{A}_{m}\mathbf{A}_{m^\prime}) \\
& \qquad+ \frac{1}{2n^{2}}\sum_{s=1}^{n-\tau}\sum_{t=1}^{n-\tau}\sum_{m=\max\{-\tau, q+1\}}^{q-\tau}
\tr(\mathbf{A}_{m+t-s+\tau}\mathbf{A}_{m+\tau}\mathbf{A}_{m}\mathbf{A}_{m+t-s}) \\
=& K_{1,3}^{(1)} + K_{1,3}^{(2)}.
\end{align*}
Observe that $K_{1,3}^{(1)}$ coincides with $(1/2)\mE[\|S_{\tau,1}\|_{F}^{2}]$. Finally,
\begin{align*}
K_{1,4}
=& \frac{1}{2n^2}\sum_{s=1}^{n-\tau}
\sum_{m=\max\{-\tau,q+1\}}^{q-\tau}\sum_{u=\max\{1-s,q+1-m,-(m+\tau)\}}^{\min\{n-\tau-s,q-(m+\tau)\}}
\bigg[\tr(\mathbf{A}_{m+\tau}\mathbf{A}_{m+\tau+u})\tr(\mathbf{A}_{m}\mathbf{A}_{m+u}) \\
& \qquad + \sum_{k=1}^{p}(\mathbf{A}_{m+\tau}\mathbf{A}_{m+\tau+u})_{kk}(\mathbf{A}_{m}\mathbf{A}_{m+u})_{kk}
\big(\mE\big[|Z_{s-m, k}|^4\big]-1\big)\bigg]
\end{align*}
These quantities are bounded using the basic bound $\|\mathbf{A}_\ell\|\leq\bar{a}_\ell$. In the
following, let $\bar{a}_\ell = 0$ for $\ell < 0$ and denote by $C$ a generic positive constant. For $\tau \geq 0$,
\begin{align*}
|K_{1,1}|
\leq& \frac{1}{2n^2}\sum_{s=1}^{n-\tau}\sum_{\ell^\prime=0}^{q}\sum_{\stackrel{m=q+1}{m\neq \ell^\prime-\tau}}^{\infty}
\sum_{u_1=\max\{1-s, -\ell^\prime,q+1-m\}}^{\min\{n-\tau-s,q-\ell^\prime\}}
p^2 \bar{a}_{\ell^\prime}\bar{a}_{\ell^\prime+u_1} \bar{a}_{m}\bar{a}_{m+u_1} \\
\leq& \frac{p^2}{2n} \sum_{u=-\infty}^\infty \bigg(\sum_{\ell^\prime=0}^{\infty} \bar{a}_{\ell^\prime}\bar{a}_{\ell^\prime+u}\bigg)
\bigg(\sum_{m=q+1}^\infty \bar{a}_{m}\bar{a}_{m+u}\bigg)\\
\leq&  2^{r_0+1}\frac{p^2}{n} \sum_{u=-\infty}^\infty  \frac{1}{(1+|u|^{r_0})^2}
\bigg(\sum_{\ell=0}^{\infty}\ell^{r_0}\bar{a}_{\ell}\bigg)\bigg(\sum_{\ell=0}^{\infty}\bar{a}_{\ell}\bigg)
\bigg(\sum_{m=q+1}^{\infty}m^{r_0}\bar{a}_{m}\bigg)\bigg(\sum_{m=q+1}^{\infty}\bar{a}_{m}\bigg) \\
\leq& C L_1 L_{r_0+1} \frac{p^2}{n}
\bigg(\sum_{m=q+1}^{\infty}m^{r_0}\bar{a}_{m}\bigg)\bigg(\sum_{m=q+1}^{\infty}\bar{a}_{m}\bigg),
\end{align*}
where the third inequality follows from (\ref{eq:a_sum_convolution_bound}).  Next,
\begin{align*}
|K_{1,2}|
\leq& \frac{1}{2n^2}\sum_{s=1}^{n-\tau}\sum_{\ell^\prime=0}^{q}\sum_{\stackrel{m=q+1}{m\neq \ell^\prime-\tau}}^{\infty}
\sum_{u_2 = \max\{1-s,-(m+\tau),q+1+\tau-\ell^\prime\}}^{\min\{n-\tau-s,q-(m+\tau)\}}
p \bar{a}_{m+u_{2}+\tau}\bar{a}_{m}\bar{a}_{\ell^\prime}\bar{a}_{\ell^\prime+u_{2}-\tau} \\
\leq& \frac{p}{2n} \sum_{u=-\infty}^\infty \bigg(\sum_{\ell^\prime=0}^{\infty} \bar{a}_{\ell^\prime}\bar{a}_{\ell^\prime+u-\tau}\bigg)
\bigg(\sum_{m=q+1}^\infty \bar{a}_{m}\bar{a}_{m+u+\tau}\bigg)\\
\leq& C L_1 L_{r_0+1} \frac{p}{n} \bigg(\sum_{m=q+1}^{\infty}m^{r_0}\bar{a}_{m}\bigg)\bigg(\sum_{m=q+1}^{\infty}\bar{a}_{m}\bigg).
\end{align*}
Also,
\begin{align*}
|K_{1,3}^{(2)}|
\leq& \frac{1}{2n^{2}}\sum_{s=1}^{n-\tau}\sum_{t=1}^{n-\tau}\sum_{m=\max\{-\tau, q+1\}}^{q-\tau}
p\bar{a}_{m+t-s+\tau}\bar{a}_{m+\tau}\bar{a}_{m}\bar{a}_{m+t-s} \\
\leq& \frac{p}{2n} \sum_{u=-\infty}^\infty \sum_{m = q+1}^\infty \bar{a}_{m+\tau+u}\bar{a}_{m+\tau}\bar{a}_{m}\bar{a}_{m+u} \\
\leq& \frac{p}{2n} \sum_{u=-\infty}^\infty \bigg(\sum_{m = q+1}^\infty  \bar{a}_{m+\tau+u}\bar{a}_{m+\tau}\bigg)
\bigg(\sum_{m = q+1}^\infty  \bar{a}_{m+u}\bar{a}_{m}\bigg) \\
\leq& C \frac{p}{n} \bigg(\sum_{m=q+1}^\infty m^{r_0} \bar{a}_m\bigg)^2\bigg(\sum_{m=q+1}^\infty \bar{a}_m\bigg)^2.
\end{align*}
Finally,
\begin{align*}
|K_{1,4}|
\leq& \frac{1}{2n^2}\sum_{s=1}^{n-\tau} \sum_{m=\max\{-\tau,q+1\}}^{q-\tau}\sum_{u=\max\{1-s,q+1-m,-(m+\tau)\}}^{\min\{n-\tau-s,q-(m+\tau)\}}
\big[p^2 + p \big(\mE\big[|Z_{11}|^4\big]-1\big)\big] \bar{a}_{m+\tau+u}\bar{a}_{m+\tau}\bar{a}_{m}\bar{a}_{m+u}  \\
\leq& C\frac{p^2}{n} \sum_{u=-\infty}^\infty \sum_{m=q+1}^\infty \bar{a}_{m+\tau+u}\bar{a}_{m+\tau}\bar{a}_{m}\bar{a}_{m+u} \\
\leq& C \frac{p^2}{n} \bigg(\sum_{m=q+1}^\infty m^{r_0} \bar{a}_m\bigg)^2\bigg(\sum_{m=q+1}^\infty \bar{a}_m\bigg)^2.
\end{align*}
For any $\tau \geq 0$, the above calculations yield the bound
\begin{equation}\label{eq:S_1_diff_F_norm_bound}
\mE\big[\|\mathbf{S}_{\tau,1} - \mE[ \mathbf{S}_{\tau,1}]\|_F^2\big] \leq
C \frac{p^2}{n} \bigg(\sum_{m=q+1}^\infty m^{r_0} \bar{a}_m\bigg)\bigg(\sum_{m=q+1}^\infty \bar{a}_m\bigg)
\bigg[1+\bigg(\sum_{m=q+1}^\infty m^{r_0} \bar{a}_m\bigg)\bigg(\sum_{m=q+1}^\infty \bar{a}_m\bigg) \bigg],
\end{equation}
for some constant $C$. In can be checked that the same bound applies
to $\mE[\| \mathbf{S}_{\tau,2} - \mE[\mathbf{S}_{\tau,2}]\|_F^2]$ as well.

\subsection{Bounding $\mE[\|\mathbf{S}_{\tau, 3}-\mE [\mathbf{S}_{\tau, 3}]\|_{F}^{2}]$}
\label{app:lin_proc:2}

Note first that $\mE[ \mathbf{S}_{\tau,3}] = (2n)^{-1} \sum_{t=1}^{n-\tau}\mE[\bar{X}_{t}\bar{X}_{t+\tau}^*
+\bar{X}_{t+\tau}\bar{X}^{*}_{t}]$
and
\[
\|\mE [\mathbf{S}_{\tau,3}]\|^2_{F}
= \left(1-\frac{\tau}{n}\right)^2
\sum_{m^\prime=\max\{q+1-\tau,q+1\}}^{\infty}\sum_{m=\max\{q+1-\tau,q+1\}}^{\infty}
\tr(\mathbf{A}_{m^\prime+\tau}\mathbf{A}_{m+\tau}\mathbf{A}_{m}\mathbf{A}_{m^\prime}).
\]
Moreover,
\begin{align*}
\mE\big[\|\mathbf{S}_{\tau,3}\|^2_{F}\big]
=& \frac{1}{2n^2}\sum_{t=1}^{n-\tau}\sum_{s=1}^{n-\tau}
\Re\mE\big[{\bar{X}_{t+\tau}}^{*}\bar{X}_{s+\tau}\bar{X}^{*}_{s}\bar{X}_{t}\big]
+ \frac{1}{2n^2}\sum_{t=1}^{n-\tau}\sum_{s=1}^{n-\tau}
\Re\mE \big[{\bar{X}_{t+\tau}}^{*}\bar{X}_{s}\bar{X}^{*}_{s+\tau}\bar{X}_{t}\big] \\
=&\frac{1}{n^2}\sum_{t=1}^{n-\tau}\sum_{s=1}^{n-\tau}
\sum_{\ell=q+1}^{\infty}\sum_{\ell^\prime=q+1}^{\infty}\sum_{m=q+1}^{\infty}\sum_{m^\prime=q+1}^{\infty}
\Re\mE \big[Z^{*}_{t+\tau-\ell}\mathbf{A}_{\ell}\mathbf{A}_{\ell^\prime}Z_{s+\tau-\ell^\prime}
Z_{s-m}^{*}\mathbf{A}_{m}\mathbf{A}_{m^\prime}Z_{t-m^\prime}\big] \\
& + \frac{1}{2n^2} \sum_{t=1}^{n-\tau}\sum_{s=1}^{n-\tau}
\sum_{\ell=q+1}^{\infty}\sum_{\ell^\prime=q+1}^{\infty}\sum_{m=q+1}^{\infty}\sum_{m^\prime=q+1}^{\infty}
\Re\mE \big[Z^{*}_{t+\tau-\ell}\mathbf{A}_{\ell}\mathbf{A}_{\ell^\prime}Z_{s-\ell^\prime}
Z_{s+\tau-m}^{*}\mathbf{A}_{m}\mathbf{A}_{m^\prime}Z_{t-m^\prime}\big] \\
= & R_1 + R_2,
\end{align*}
where $R_j = T_{j,1}+T_{j,2}+T_{j,3}+T_{j,4}$, $j=1,2$, with
\begin{align*}
T_{1,1}
=& \frac{1}{2n^2}\sum_{s=1}^{n-\tau}\sum_{\ell^\prime=q+1}^{\infty}\sum_{\stackrel{m=q+1}{m\neq \ell^\prime-\tau}}^{\infty}
\sum_{u_1=\max\{1-s, q+1-\ell^\prime,q+1-m\}}^{n-\tau-s}
\tr(\mathbf{A}_{\ell^\prime}\mathbf{A}_{\ell^\prime+u_{1}})\tr(\mathbf{A}_{m}\mathbf{A}_{m+u_1}), \\[.2cm]
T_{1,2}
=& \frac{1}{2n^2}\sum_{s=1}^{n-\tau}\sum_{\ell^\prime=q+1}^{\infty}\sum_{\stackrel{m=q+1}{m\neq \ell^\prime-\tau}}^{\infty}
\sum_{u_{2}=\max\{1-s, q+1-(m+\tau),q+1+\tau-\ell^\prime\}}^{n-\tau-s} \\
&  \qquad\times\tr(\mathbf{A}_{m+u_{2}+\tau}\mathbf{A}_{m}\mathbf{A}_{\ell^\prime}\mathbf{A}_{\ell^\prime+u_{2}-\tau})
\mE[\bar{Z}_{s-m,1}^2]\mE[Z_{s+\tau-l',1}^2], \\[.2cm]
T_{1,3}
=& \frac{1}{2n^{2}}\sum_{s=1}^{n-\tau}\sum_{t=1}^{n-\tau}\sum_{m=\max\{q+1-\tau, q+1\}}^{\infty}
\sum_{\stackrel{m^\prime=\max\{q+1-\tau, q+1\}}{m^\prime\ne m+t-s}}^{\infty}
\tr(\mathbf{A}_{m^\prime+\tau}\mathbf{A}_{m^\prime}\mathbf{A}_{m+\tau}\mathbf{A}_{m}) \\
=& \frac{1}{2n^{2}}\sum_{s=1}^{n-\tau}\sum_{t=1}^{n-\tau}\sum_{m=\max\{q+1-\tau, q+1\}}^{\infty}
\sum_{m^\prime=\max\{q+1-\tau, q+1\}}^{\infty}
\tr(\mathbf{A}_{m^\prime+\tau}\mathbf{A}_{m+\tau}\mathbf{A}_{m}\mathbf{A}_{m^\prime}) \\
& + \frac{1}{2n^{2}}\sum_{s=1}^{n-\tau}\sum_{t=1}^{n-\tau}\sum_{m=\max\{q+1-\tau, q+1\}}^{\infty}
\tr(\mathbf{A}_{m+t-s+\tau}\mathbf{A}_{m+\tau}\mathbf{A}_{m}\mathbf{A}_{m+t-s}) \\
=& T_{1,3}^{(1)} + T_{1,3}^{(2)}, \\[.2cm]
T_{1,4}
=& \frac{1}{2n^2}\sum_{s=1}^{n-\tau}\sum_{m=\max\{-\tau,q+1\}}^{q-\tau}\sum_{u=\max\{1-s,q+1-m,q+1-(m+\tau)\}}^{n-\tau-s}
\bigg[\tr(\mathbf{A}_{m+\tau}\mathbf{A}_{m+\tau+u})\tr(\mathbf{A}_{m}\mathbf{A}_{m+u}) \\
& + \sum_{k=1}^{p}(\mathbf{A}_{m+\tau}\mathbf{A}_{m+\tau+u})_{kk}(\mathbf{A}_{m}\mathbf{A}_{m+u})_{kk}
\big(\mE[|Z_{s-m, k}|^4]-1\big)\bigg]
\end{align*}
and $T_{1,3}^{(1)} = \frac{1}{2}\|\mE[\mathbf{S}_{\tau,3}]\|^2$. The corresponding quantities $T_{2,j}$, $j=1,\ldots,4$, can be expressed similarly.

Using calculations as in case of $K_{1,j}$, $j=1,\ldots,4$, it follows that
\begin{align*}
|T_{1,1}| \leq& C\frac{p^2}{n}\bigg( \sum_{m=q+1}^{\infty}m^{r_0}\bar{a}_{m}\bigg)\bigg(\sum_{m=q+1}^{\infty}\bar{a}_{m}\bigg) \\
|T_{1,2}| \leq& C\frac{p}{n} \bigg(\sum_{m=q+1}^{\infty}m^{r_0}\bar{a}_{m}\bigg)\bigg(\sum_{m=q+1}^{\infty}\bar{a}_{m}\bigg) \\
|T_{1,3}^{(2)}| \leq& C \frac{p}{n} \bigg(\sum_{m=q+1}^{\infty}m^{r_0}\bar{a}_{m}\bigg)^2\bigg(\sum_{m=q+1}^{\infty}\bar{a}_{m}\bigg)^2 \\
|T_{1,4}| \leq& C \frac{p^2}{n} \bigg(\sum_{m=q+1}^{\infty}m^{r_0}\bar{a}_{m}\bigg)^2\bigg(\sum_{m=q+1}^{\infty}\bar{a}_{m}\bigg)^2,
\end{align*}
with similar bounds for $T_{2,j}$, $j=1,\ldots,4$. Therefore, for any $\tau \geq 0$,
\begin{equation}
\label{eq:S_3_diff_F_norm_bound}
\mE\big[\|\mathbf{S}_{\tau,3} - \mE[ \mathbf{S}_{\tau,3}]\|_F^2\big]
\leq C \frac{p^2}{n} \bigg(\sum_{m=q+1}^\infty m^{r_0} \bar{a}_m\bigg)\bigg(\sum_{m=q+1}^\infty \bar{a}_m\bigg)
\bigg[1+ \bigg(\sum_{m=q+1}^\infty m^{r_0} \bar{a}_m\bigg)\bigg(\sum_{m=q+1}^\infty \bar{a}_m\bigg) \bigg]
\end{equation}
for some constant $C$.

Finally, observe that $(\sum_{m=q+1}^\infty m^{r_0} \bar{a}_m)(\sum_{m=q+1}^\infty \bar{a}_m)
\leq L_1 L_{r_0+1}$. Then, using that $q = \lceil p^{1/4} \rceil$ the bound
\begin{align*}
\sum_{p=1}^\infty
& \frac{n}{p^2} \frac{p^2}{n} \bigg(\sum_{m=q+1}^\infty m^{r_0} \bar{a}_m\bigg)\bigg(\sum_{m=q+1}^\infty \bar{a}_m\bigg) \\
&\leq L_{r_0+1} \sum_{p=1}^\infty \sum_{m=q+1}^\infty \bar{a}_m \\
&\leq L_{r_0+1} \sum_{m=0}^\infty \bar{a}_m \sum_{p=1}^\infty \mathbf{1}_{\{p^{1/4} \leq m\}} \\
&\leq L_{r_0+1} \sum_{m=0}^\infty m^4 \bar{a}_m ~\leq~ L_{r_0+1}L_{5} < \infty.
\end{align*}
This completes the proof of (\ref{eq:series_summation}) by virtue
of (\ref{eq:S_1_diff_F_norm_bound}) and (\ref{eq:S_3_diff_F_norm_bound}).

\section{Proving that the expression in (\ref{eq:difference_gaussian_stieltjes_nongaussian_stieltjes})
converges to zero}
\label{app:non-gauss}

Let $\mathbf{Z}=[Z_{1-q}\colon\cdots\colon Z_n]$ be the $p\times(n+q)$ matrix of innovations $Z_t$ with truncated, centered and rescaled Gaussian variables $Z_{t}$. Denote the real and imaginary parts
\begin{equation*}
Z^{\mathbf{R}}_{1,1-q}, Z^{\mathbf{R}}_{2, 1-q}, \ldots, Z^{\mathbf{R}}_{p,1-q}, Z^{\mathbf{R}}_{1,2-q}, \ldots, Z^{\mathbf{R}}_{p, 2-q},\ldots Z^{\mathbf{R}}_{1,n}, \ldots, Z^{\mathbf{R}}_{p,n}~~~ \mbox{by} ~~~ Y^{\mathbf{R}}_{1}, \ldots, Y^{\mathbf{R}}_{p\times(n+q)},
\end{equation*}
\begin{equation*}
Z^{\mathbf{I}}_{1,1-q}, Z^{\mathbf{I}}_{2, 1-q}, \ldots Z^{\mathbf{I}}_{p,1-q}, Z^{\mathbf{I}}_{1,2-q}, \ldots, Z^{\mathbf{I}}_{p, 2-q},\ldots, Z^{\mathbf{I}}_{1,n}, \ldots, Z^{\mathbf{I}}_{p,n}~~~ \mbox{by} ~~~ Y^{\mathbf{I}}_{1}, \ldots, Y^{\mathbf{I}}_{p\times(n+q)},
\end{equation*}
respectively. Also denote
\begin{equation*}
W^{\mathbf{R}}_{1,1-q}, W^{\mathbf{R}}_{2, 1-q}, \ldots, W^{\mathbf{R}}_{p,1-q}, W^{\mathbf{R}}_{1,2-q}, \ldots, W^{\mathbf{R}}_{p, 2-q},\ldots, W^{\mathbf{R}}_{1,n}, \ldots, W^{\mathbf{R}}_{p,n}~~~ \mbox{by} ~~~ \Y^{\mathbf{R}}_{1}, \ldots, \Y^{\mathbf{R}}_{p\times(n+q)},
\end{equation*}
\begin{equation*}
W^{\mathbf{I}}_{1,1-q}, W^{\mathbf{I}}_{2, 1-q}, \ldots, W^{\mathbf{I}}_{p,1-q}, W^{\mathbf{I}}_{1,2-q}, \ldots, W^{\mathbf{I}}_{p, 2-q},\ldots, W^{\mathbf{I}}_{1,n}, \ldots, W^{\mathbf{I}}_{p,n}~~~ \mbox{by} ~~~ \Y^{\mathbf{I}}_{1}, \ldots, \Y^{\mathbf{I}}_{p\times(n+q)}.
\end{equation*}
Let $\bar{m}_n = p(n+q)$. Note that $((Y_k^{\mathbf{R}}, Y_k^{\mathbf{I}}) : k=1,\ldots,\bar{m_n})$, is a reordering of
the variables $((Z_{jt}^\mathbf{R},Z_{jt}^\mathbf{I}): j=1,\ldots,p; t=1-q,\ldots,n)$ by stacking the columns of the matrix,
and similarly for $\{(\tilde Y_k^{\mathbf{R}}, Y_k^{\mathbf{I}}) : k=1,\ldots,\bar{m_n}\}$. This order relationship
is assumed throughout. Define,
\[
 T_{k} = (Y^{\mathbf{R}}_1, \ic Y^{\mathbf{I}}_1, \ldots, Y^{\mathbf{R}}_{k}, \ic Y^{\mathbf{I}}_{k}, \Y^{\mathbf{R}}_{k+1} ,\ic \Y^{\mathbf{I}}_{k+1},\ldots, \Y^{\mathbf{R}}_{\bar{m}_n}, \ic \Y^{\mathbf{I}}_{\bar{m}_n}), ~~~\mbox{for}~k=1,\ldots,\bar{m}_n-1,
\]
and
\[
T_0 = (Y_1^{\mathbf{R}},Y_1^{\mathbf{I}},\ldots,Y_{\bar{m}_n}^{\mathbf{R}},Y_{\bar{m}_n}^{\mathbf{I}}),
\qquad T_{\bar{m}_n} = (\tilde Y_1^{\mathbf{R}},\tilde Y_1^{\mathbf{I}},\ldots,\tilde Y_{\bar{m}_n}^{\mathbf{R}},\tilde Y_{\bar{m}_n}^{\mathbf{I}}).
\]
Introduce
\begin{eqnarray*}
 T^{\text{bridge}}_{k} &=& (Y^{\mathbf{R}}_1,\ic Y^{\mathbf{I}}_1, \ldots, Y^{\mathbf{R}}_{k-1},\ic Y^{\mathbf{I}}_{k-1}, 0, \ic Y^{\mathbf{I}}_{k}, \Y^{\mathbf{R}}_{k+1} ,\ic \Y^{\mathbf{I}}_{k+1},\ldots, \Y^{\mathbf{R}}_{\bar{m}_n}, \ic \Y^{\mathbf{I}}_{\bar{m}_n}),\\
 \hat{T}^{\text{bridge}}_{k} &=& (Y^{\mathbf{R}}_1,\ic Y^{\mathbf{I}}_1, \ldots, Y^{\mathbf{R}}_{k-1},\ic Y^{\mathbf{I}}_{k-1}, 0, \ic \tilde Y^{\mathbf{I}}_{k}, \Y^{\mathbf{R}}_{k+1} ,\ic \Y^{\mathbf{I}}_{k+1},\ldots, \Y^{\mathbf{R}}_{\bar{m}_n}, \ic \Y^{\mathbf{I}}_{\bar{m}_n}),\\
T^{0}_{k} &=&  (Y^{\mathbf{R}}_1, \ic Y^{\mathbf{I}}_1, \ldots, Y^{\mathbf{R}}_{k-1},\ic Y^{\mathbf{I}}_{k-1}, 0, 0, \Y^{\mathbf{R}}_{k+1} ,\ic \Y^{\mathbf{I}}_{k+1},\ldots, \Y^{\mathbf{R}}_{\bar{m}_n}, \ic \Y^{\mathbf{I}}_{\bar{m}_n}).
\end{eqnarray*}
Suppose that, for a fixed $z\in \mC^+$,  $f$ is a function of $2\bar{m}_n$ variables defined as
\begin{equation}
f(\mathbf{y}) = \frac{1}{p}\tr(\mathbf{C}_{\tau}(\mathbf{y})-zI)^{-1},
\end{equation}
where we loosely use $\mathbf{C}_{\tau}(\mathbf{y})$ to mean the symmetrized lag-$\tau$ sample autocovariance obtained
by the columns of the $p\times(n+q)$ matrix constructed by appropriately reorganizing the elements of the $2\bar{m}_n \times 1$ vector
$\mathbf{y}$ so that $(2k-1)$-th and $(2k)$-th coordinates form the real and ($\ic$ times) imaginary part of the
entries of the data matrix for each $k=1,\ldots,\bar{m}_n$. With an appropriate reorganization scheme, we can write
$f(T_{0})=p^{-1}\tr(\mathbf{C}^{\prime}_{\tau}-zI)^{-1}$ and $f(T_{\bar{m}_n})=p^{-1}\tr(\mathbf{C}_{\tau}-zI)^{-1}$. Therefore, (\ref{eq:difference_gaussian_stieltjes_nongaussian_stieltjes}) can be written as a telescoping sum involving one-by-one replacements of random variables $(Y_k^{\mathbf{R}},Y_k^{\mathbf{I}})$ with $(\tilde Y_k^{\mathbf{R}},\tilde Y_k^{\mathbf{I}})$, that is,
\begin{equation}\label{eq:mean_diff_telescoping}
\mE\left[\frac{1}{p}\tr(\mathbf{C}_\tau-zI)^{-1}\right]
-\mE\left[\frac{1}{p}\tr(\mathbf{C}_{\tau}^{\prime}-zI)^{-1}\right]
= \sum_{k= 1}^{\bar{m}_n}\mE\big[f(T_{k})-f(T_{k-1})\big].
\end{equation}
In the following, we use $\p_k^{r}$ and $\bar\p_k^{r}$
to denote the $r$-th order partial derivative with respect to the $(2k-1)$-th coordinate and $(2k)$-th coordinate,
respectively.

Define, for $\xi \in [0,1]$,
\begin{eqnarray*}
T_{k}^{(1)}(\xi) &=&  (Y^{\mathbf{R}}_1,\ic Y^{\mathbf{I}}_1, \ldots, \xi Y^{\mathbf{R}}_{k}, \ic Y^{\mathbf{I}}_{k}, \Y^{\mathbf{R}}_{k+1} ,\ic \Y^{\mathbf{I}}_{k+1},\ldots, \Y^{\mathbf{R}}_{\bar{m}_n}, \ic \Y^{\mathbf{I}}_{\bar{m}_n})\\
T_{k}^{(2)}(\xi) &=& (Y^{\mathbf{R}}_1,\ic Y^{\mathbf{I}}_1, \ldots, 0, \ic \xi Y^{\mathbf{I}}_{k}, \Y^{\mathbf{R}}_{k+1} ,\ic \Y^{\mathbf{I}}_{k+1},\ldots, \Y^{\mathbf{R}}_{\bar{m}_n}, \ic \Y^{\mathbf{I}}_{\bar{m}_n})\\
\hat{T}_{k}^{(1)}(\xi) &=&  (Y^{\mathbf{R}}_1,\ic Y^{\mathbf{I}}_1, \ldots, \xi \tilde Y^{\mathbf{R}}_{k}, \ic \Y^{\mathbf{I}}_{k},
\Y^{\mathbf{R}}_{k+1} ,\ic \Y^{\mathbf{I}}_{k+1},\ldots, \Y^{\mathbf{R}}_{\bar{m}_n}, \ic \Y^{\mathbf{I}}_{\bar{m}_n}) \\
\hat{T}_{k}^{(2)}(\xi) &=& (Y^{\mathbf{R}}_1,\ic Y^{\mathbf{I}}_1, \ldots, 0, \ic \xi \Y^{\mathbf{I}}_{k}, \Y^{\mathbf{R}}_{k+1} ,\ic \Y^{\mathbf{I}}_{k+1},\ldots, \Y^{\mathbf{R}}_{\bar{m}_n}, \ic \Y^{\mathbf{I}}_{\bar{m}_n}).
\end{eqnarray*}
Since $f$ is a smooth function of its arguments (being a Stieltjes transform evaluated at $z\in \mC^+$),
a third-order Taylor expansion gives
\begin{eqnarray}\label{eq:f_T_k_bridge}
f(T_{k}) &=& f(T_{k}^{\text{bridge}})+Y^{\mathbf{R}}_{k}
\p_{k} f(T^{\text{bridge}}_{k})+\frac{1}{2}(Y^{\mathbf{R}}_{k})^{2}\p_{k}^{2}f(T_k^{\text{bridge}})
+\frac{1}{6}(Y^{\mathbf{R}}_{k})^{3}\int_{0}^{1}(1-\xi)^2\p_{k}^{3}f\big(T_{k}^{(1)}(\xi)\big)d\xi, \nonumber\\
 f(T_{k}^{\text{bridge}}) &=& f(T_{k}^{0})+(\ic Y^{\mathbf{I}}_{k})\bar\p_{k} f(T_{k}^{0})+ \frac{1}{2}(\ic Y^{\mathbf{I}}_{k})^{2}\bar\p_{k}^{2}f(T_k^{0})+\frac{1}{6}(\ic Y^{\mathbf{I}}_{k})^{3}\int_{0}^{1}(1-\xi)^2\bar\p_{k}^{3}f(T_{k}^{(2)}(\xi))d\xi,
\end{eqnarray}
and
\begin{eqnarray}\label{eq:f_T_k_bridge_mixed}
\p_{k} f(T^{\text{bridge}}_{k}) &=& \p_{k}f(T_k^0) + (\ic Y_k^{\mathbf{I}})\bar\p_{k}\p_{k}f(T_k^0)
+ \frac{1}{2}(\ic Y_k^{\mathbf{I}})^2 \int_0^1 (1-\xi) \bar\p_{k}^2\p_{k}f(T_k^{(2)}(\xi)) d\xi \nonumber\\
\p_{k}^2 f(T^{\text{bridge}}_{k}) &=& \p_{k}^2f(T_k^0) + (\ic Y_k^{\mathbf{I}}) \int_0^1
\bar\p_{k}\p_{k}^2f(T_k^{(2)}(\xi)) d\xi.
\end{eqnarray}
Similarly, one derives the expansion for $f(T_{k-1})$ as
\begin{eqnarray}\label{eq:f_T_k_hat_bridge}
f(T_{k-1}) &=& f(\hat{T}_{k}^{\text{bridge}})+\Y^{\mathbf{R}}_{k}\p_{k} f(\hat{T}^{\text{bridge}}_{k})+\frac{1}{2}(\Y^{\mathbf{R}}_{k})^{2}\p_{k}^{2}f(\hat{T}_k^{\text{bridge}})
+\frac{1}{6}(\Y^{\mathbf{R}}_{k})^{3}\int_{0}^{1}(1-\xi)^2\p_{k}^{3}f(\hat{T}_{k}^{(1)}(\xi))d\xi,\nonumber\\
f(\hat{T}_{k}^{\text{bridge}}) &=&  f(T_{k}^{0})+(\ic \Y^{\mathbf{I}}_{k})\bar\p_{k} f(T_{k}^{0})+ \frac{1}{2}(\ic \Y^{\mathbf{I}}_{k})^{2}\bar\p_{k}^{2}f(T_k^{0})+\frac{1}{6}
(\ic \Y^{\mathbf{I}}_{k})^{3}\int_{0}^{1}(1-\xi)^2\bar\p_{k}^{3}f(\hat{T}_{k}^{(2)}(\xi))d\xi,
\end{eqnarray}
and
\begin{eqnarray}\label{eq:f_T_k_hat_bridge_mixed}
\p_{k} f(\hat{T}^{\text{bridge}}_{k}) &=& \p_{k}f(T_k^0) + (\ic \tilde Y_k^{\mathbf{I}})\bar\p_{k}\p_{k}f(T_k^0)
+ \frac{1}{2}(\ic \tilde Y_k^{\mathbf{I}})^2 \int_0^1 (1-\xi) \bar\p_{k}^2\p_{k}f(\hat{T}_k^{(2)}(\xi)) d\xi \nonumber\\
\p_{k}^2 f(\hat{T}^{\text{bridge}}_{k}) &=& \p_{k}^2f(T_k^0) + (\ic \tilde Y_k^{\mathbf{I}}) \int_0^1
\bar\p_{k}\p_{k}^2f(\hat{T}_k^{(2)}(\xi)) d\xi.
\end{eqnarray}

By {\bf T2}, the $(Y^{\mathbf{R}}_{k}, Y^{\mathbf{I}}_{k})$ and the $(\Y^{\mathbf{R}}_{k},\Y^{\mathbf{I}}_{k})$
are independent and each has independent real and imaginary parts with zero mean and equal variance. Therefore,
from the expansions in (\ref{eq:f_T_k_bridge}), (\ref{eq:f_T_k_bridge_mixed}), (\ref{eq:f_T_k_hat_bridge})
and (\ref{eq:f_T_k_hat_bridge_mixed}), it follows that bounding
(\ref{eq:mean_diff_telescoping}) is equivalent to bounding
\begin{align}
\label{eq:remainder_complex}
&\sum_{k=1}^{\bar{m}_n} \frac{1}{6}\int_{0}^{1}(1-\xi)^2\mE\left[(Y^{\mathbf{R}}_{k})^3\p_{k}^{3}(f(T_{k}^{(1)}(\xi))
-(\Y^{\mathbf{R}}_{k})^{3}\p_{k}^3f(\hat{T}_{k}^{(1)}(\xi))\right]d\xi \nonumber\\
&\quad + \sum_{k=1}^{\bar{m}_n} \frac{1}{6}\int_{0}^{1}(1-\xi)^2\mE\left[(\ic Y^{\mathbf{I}}_{k})^3\p_{k}^{3}(f(T_{k}^{(2)}(\xi))-(\ic \Y^{\mathbf{I}}_{k})^{3}\p_{k}^3f(\hat{T}_{k}^{(2)}(\xi))\right]d\xi  + \Delta_n,
\end{align}
where
\begin{eqnarray*}
\Delta_n &=& \sum_{k=1}^{\bar{m}_n} \frac{1}{2} \int_0^1 (1-\xi) \mathbb{E}\left[Y_k^{\mathbf{R}}(\ic Y_k^{\mathbf{I}})^2 \bar\p_{k}^2\p_{k}f(T_k^{(2)}(\xi)) - \tilde Y_k^{\mathbf{R}}(\ic \tilde Y_k^{\mathbf{I}})^2 \bar\p_{k}^2\p_{k}f(\hat{T}_k^{(2)}(\xi))\right] d\xi
\\
&& + \sum_{k=1}^{\bar{m}_n} \frac{1}{2} \int_0^1 \mathbb{E}\left[(Y_k^{\mathbf{R}})^2(\ic Y_k^{\mathbf{I}})
\bar\p_{k}\p_{k}^2f(T_k^{(2)}(\xi))
- (\tilde Y_k^{\mathbf{R}})^2(\ic \tilde Y_k^{\mathbf{I}}) \bar\p_{k}\p_{k}^2f(\hat{T}_k^{(2)}(\xi))\right] d\xi.
\end{eqnarray*}
Derivation of upper bounds for each of the above terms follows the same pattern and, for simplicity, only arguments
for the real valued case are provided, whereupon the mixed derivative terms are absent.
It should, moreover, be emphasized that the Gaussianity of the $Z_{jt}$ is not used in the proofs of this section
as only moment conditions are invoked, so that the notation $Z_{jt}$ could be used for either $Z_{jt}$ or $W_{jt}$,
noticing that their role will be the same when using the bounds for expected values of $\mE[(Y_{k}^{c})^3\p_{k}^{3}(f(T_{k}^{(1)}(\xi))]$ and $\mE[(\tilde{Y}_{k}^c)^{3}\p_{k}^3f(\hat{T}_{k}^{(1)}(\xi))]$, where $c$ is either $\mathbf{R}$ or $\mathbf{I}$.
Due to the simplification afforded by the expansion (\ref{eq:remainder_complex}) in terms of the real and
imaginary parts of the random variables, in the following, without loss of generality, we treat $Z_{jt}$'s to be real valued
and focus on bounding the expression on the first line of (\ref{eq:remainder_complex}).
This will require straightforward modification of the definitions of $T_k^{(1)}(\xi)$ and $\hat{T}_k^{(1)}(\xi)$.
The corresponding versions for the real valued case are $T_k^{(1)}(\xi) = (Y_1,\ldots,Y_{k-1},\xi Y_k,\tilde
Y_{k+1},\ldots,\tilde Y_{\bar{m}_n})$ and $\hat{T}_k^{(1)}(\xi) = (Y_1,\ldots,Y_{k-1},\xi \tilde Y_k,\tilde
Y_{k+1},\ldots,\tilde Y_{\bar{m}_n})$, where we omit the superscript $\mathbf{R}$ since it is superfluous.

Thus, it remains to obtain an expression for $\p_i^3 f(T_{i}^{(1)}(\xi))$, where $f$ is treated as
a function over $\mathbb{R}^{\bar{m}_n}$ and $\p_i$ denotes partial derivative with respect to the $i$-th coordinate.
For the rest of this section, $(j,k)$ denotes the pair of indices such that $Z_{jk}$ is mapped into $Y_i$ in the mapping from
$\{Z_{lt}:l=1,\ldots,p;t=1-q,\ldots,n\}$ to $(Y_1,\ldots,Y_{\bar{m}_n})$,
Throughout, unless otherwise specified, index $i$ and hence $(j,k)$, are kept fixed. We also redefine $T_i =
(Y_1,\ldots,Y_i,\tilde Y_{i+1},\ldots,Y_{\bar{m}_n})$.
Let the resolvent of $\mathbf{C}_{\tau}^{(i)} \equiv \mathbf{C}_{\tau}(T_i)$ be denoted by $G_{\tau}^{(i)}(z) = (\mathbf{C}_{\tau}^{(i)}-zI)^{-1}$.
Thus, we can write
\begin{equation}\label{eq:third_order_derivative_resolvent}
\p^3 f(T_i^{(1)}(\xi)) = \frac{1}{p}\tr\left(\frac{\p^3 G_{\tau}^{(i)}}{\p Z_{jk}^3}\right)\left|_{Z_{jk}=\xi Y_i}\right.,
\end{equation}
while recalling that $Y_i = Z_{jk}$. In the following, we drop the superscript from $G_\tau^{(i)}$
for notational simplicity.

By direct computation, we obtain
\begin{align}
\label{eq:expression_third_derivative_C_tau_Z}
\frac{\p^3 G_{\tau}}{\p Z_{jk}^{3}}
=& -\frac{\p^3 \mathbf{C}_{\tau}}{\p Z^3_{jk}}G_{\tau}^{2}+2\frac{\p^2 \mathbf{C}_{\tau}}{\p Z_{jk}^2}G_{\tau}\frac{\p \mathbf{C}_{\tau}}{\p Z_{jk}}G_{\tau}^2+2\frac{\p^2 \mathbf{C}_{\tau}}{\p Z^2_{jk}}G_{\tau}\frac{\p \mathbf{C}_{\tau}}{\p Z_{jk}}G^{2}_{\tau}\nonumber\\
&- 2\frac{\p \mathbf{C}_{\tau}}{\p Z_{jk}}\frac{\p \mathbf{C}_{\tau}}{\p Z_{jk}}G^{2}_{\tau}\frac{\p \mathbf{C}_{\tau}}{\p Z_{jk}}G^{2}_{\tau}+2\frac{\p \mathbf{C}_{\tau}}{\p Z_{jk}}G_{\tau}\frac{\p^2 \mathbf{C}_{\tau}}{\p Z_{jk}^2}G^{2}_{\tau}- 4 \frac{\p \mathbf{C}_{\tau}}{\p Z_{jk}}G_{\tau}\frac{\p \mathbf{C}_{\tau}}{\p Z_{jk}}G_{\tau}\frac{\p \mathbf{C}_{n}}{\p Z_{jk}}G^{2}_{\tau}\nonumber\\
=& 6\frac{\p^2 \mathbf{C}_{\tau}}{\p Z^2_{jk}}G_{\tau}\frac{\p \mathbf{C}_{\tau}}{\p Z_{jk}}G^{2}_{\tau}
- 4\frac{\p \mathbf{C}_{\tau}}{\p Z_{jk}}G_{\tau}\frac{\p \mathbf{C}_{\tau}}{\p Z_{jk}}G_{\tau}\frac{\p
\mathbf{C}_{\tau}}{\p Z_{jk}}G^{2}_{\tau} -2\frac{\p \mathbf{C}_{\tau}}{\p Z_{jk}}\frac{\p \mathbf{C}_{\tau}}{\p Z_{jk}}G^{2}_{\tau}\frac{\p \mathbf{C}_{\tau}}{\p Z_{jk}}G^{2}_{\tau}.
\end{align}
Then, defining $\mathcal{L}^{(1)}_{\tau,k} := \{\ell\colon \max(0, 1-k-\tau)\leq \ell \leq \min(q, n-\tau-k)\}$ and $\mathcal{L}^{(2)}_{\tau,k} := \{\ell\colon \max(0, 1-k+\tau)\leq \ell \leq \min(q, n-k)\}$,
\begin{align*}
\frac{\p \mathbf{C}_{\tau}}{\p Z_{jk}}
&= \frac{1}{2\sqrt{pn}}
\sum_{\ell\in\mathcal{L}^{(1)}_{\tau,k}}\left(\mathbf{A}_{\ell}e_{j}X_{\ell+k+\tau}^{*}+X_{\ell+k+\tau}e_{j}^{*}\mathbf{A}_{\ell}\right)
+\frac{1}{2\sqrt{pn}}
\sum_{l\in\mathcal{L}^{(2)}_{\tau,k}}\left(\mathbf{A}_{\ell}e_{j}X_{\ell+k-\tau}^{*}+X_{\ell+k-\tau}e_{j}^{*}\mathbf{A}_{\ell}\right), \\
\frac{\p^2 \mathbf{C}_{\tau}}{\p Z^2_{jk}}
&= \frac{1}{\sqrt{pn}}\sum_{\ell\in \mathcal{I}_{\tau,k}^{(1)}}\left(\mathbf{A}_{\ell}e_{j}e_{j}^{*}\mathbf{A}_{\ell+\tau}\right)
+\frac{1}{\sqrt{pn}}\sum_{\ell\in \mathcal{I}_{\tau,k}^{(2)}}\left(\mathbf{A}_{\ell-\tau}e_{j}e_{j}^{*}\mathbf{A}_{\ell}\right), \\
\frac{\p^3 \mathbf{C}_{\tau}}{\p Z^3_{jk}}
&=0,
\end{align*}
in which $ \mathcal{I}_{\tau,k}^{(1)} := \mathcal{L}^{(1)}_{\tau,k}\cap \{\ell\colon 0\leq \ell+\tau \leq q\}$ and
$\mathcal{I}_{\tau,k}^{(2)} := \mathcal{L}^{(2)}_{\tau,k}\cap \{\ell\colon 0\leq \ell-\tau \leq q\}$.
Notice that the size of the index set
$\mathcal{I}_{\tau,k}^{(i)}, i =1,2$, is at most $q+1$. Define
\[
L^{(1)}_{\tau,k}
=\sum_{\ell\in\mathcal{L}^{(1)}_{\tau,k}}(\xi_{\ell}X_{\ell+k+\tau}^{*}+X_{\ell+k+\tau}\xi_{\ell}^{*}),
\qquad\mbox{and}\qquad
L^{(2)}_{\tau,k}
=\sum_{\ell\in\mathcal{L}^{(2)}_{\tau,k}}(\xi_{\ell}X_{\ell+k-\tau}^{*}+X_{\ell+k-\tau}\xi_{\ell}^{*}),
\]
where $\xi_{\ell} = \xi_{\ell,j} = \mathbf{A}_{\ell}e_{j},$ the $j$th column of $\mathbf{A}_{\ell}$. Then
\begin{equation*}
\frac{\p \mathbf{C}_{\tau}}{\p Z_{jk}} = \frac{1}{2\sqrt{pn}}(L^{(1)}_{\tau,k}+L^{(2)}_{\tau,k}).
\end{equation*}
It follows that
\begin{align}
\label{eq:tr_deriv_C_tau_first}
\frac{1}{p}\tr\bigg(\frac{\p C_{\tau}}{\p Z_{jk}}G_{\tau}\frac{\p^2 C_{\tau}}{\p Z^2_{jk}}G^2_{\tau}\bigg)
= \eta_{1}(n)+\eta_{2}(n)+\eta_{3}(n)+\eta_{4}(n),
\end{align}
where
\begin{align*}
\eta_{1}(n)
&=\frac{1}{2np^2}\sum_{\ell^\prime\in \mathcal{I}_{\tau,k}^{(1)}}\xi_{\ell^\prime+\tau}^{*}G_{\tau}^2L_{\tau,k}^{(1)}G_{\tau}\xi_{\ell^\prime},\\
\eta_{2}(n)
&=\frac{1}{2np^2}\sum_{\ell^\prime\in \mathcal{I}_{\tau,k}^{(1)}}\xi_{\ell^\prime+\tau}^{*}G_{\tau}^2L_{\tau,k}^{(2)}G_{\tau}\xi_{\ell^\prime},\\
\eta_{3}(n)
&=\frac{1}{2np^2}\sum_{\ell^\prime\in \mathcal{I}_{\tau,k}^{(2)}}\xi_{\ell^\prime-\tau}^{*}G_{\tau}^2L_{\tau,k}^{(1)}G_{\tau}\xi_{\ell^\prime},\\
\eta_{4}(n)
&=\frac{1}{2np^2}\sum_{\ell^\prime\in \mathcal{I}_{\tau,k}^{(2)}}\xi_{\ell^\prime-\tau}^{*}G_{\tau}^2L_{\tau,k}^{(2)}G_{\tau}\xi_{\ell^\prime}.
\end{align*}
We bound $|\eta_l(n)|$ by using the fact that for any matrix $\mathbf{B}$ and vectors $a$ and $b$ such that $a^{*}\mathbf{B}b$, $|a^{*}\mathbf{B}b|\leq \|\mathbf{B}\|(a^{*}a)^{1/2}(b^{*}b)^{1/2}$, and moreover that, $\|\xi_{\ell}\|= \|\xi_{\ell,j}\|=\|\mathbf{A}_{\ell}e_{j}\|\leq \|\mathbf{A}_{\ell}\|\leq \bar{a}_{\ell}$ and $\sum_{\ell=0}^{\infty}\bar{a}_{\ell}^r\leq L_{1+r}<\infty$, for $r=0,1$. Then,
\begin{align*}
|\eta_{1}(n)|
=& \bigg|\frac{1}{2np^2}\sum_{\ell\in \mathcal{L}^{(1)}_{\tau,k}}\sum_{\ell^\prime\in\mathcal{I}_{\tau,k}^{(1)}}
\xi_{\ell^\prime+\tau}^{*}G_{\tau}^2(\xi_{\ell}X_{\ell+k+\tau}^{*}+X_{\ell+k+\tau}\xi_{\ell}^{*})G_{\tau}\xi_{\ell^\prime}\bigg| \\
\leq& \bigg|\frac{1}{2np^2}\sum_{\ell\in \mathcal{L}^{(1)}_{\tau,k}}\sum_{\ell^\prime\in\mathcal{I}_{\tau,k}^{(1)}}
\xi_{\ell^\prime+\tau}^{*}G_{\tau}^2\xi_{\ell}X_{\ell+k+\tau}^{*}G_{\tau}\xi_{\ell^\prime}\bigg|
+\bigg|\frac{1}{2np^2}\sum_{\ell\in \mathcal{L}^{(1)}_{\tau,k}}\sum_{\ell^\prime\in\mathcal{I}_{\tau,k}^{(1)}}
\xi_{\ell^\prime+\tau}^{*}G_{\tau}^2X_{\ell+k+\tau}\xi_{\ell}^{*}G_{\tau}\xi_{\ell^\prime}\bigg|\\
\leq& \frac{1}{2np^2}\sum_{\ell\in \mathcal{L}^{(1)}_{\tau,k}}\sum_{\ell^\prime\in\mathcal{I}_{\tau,k}^{(1)}}
\big|\xi_{\ell^\prime+\tau}^{*}G_{\tau}^2\xi_{\ell}\big|\big|X_{\ell+k+\tau}^{*}G_{\tau}\xi_{\ell^\prime}\big|
+\frac{1}{2np^2}\sum_{\ell\in \mathcal{L}^{(1)}_{\tau,k}}\sum_{\ell^\prime\in\mathcal{I}_{\tau,k}^{(1)}}
\big|\xi_{\ell^\prime+\tau}^{*}G_{\tau}^2X_{\ell+k+\tau}\big|\big|\xi_{\ell}^{*}G_{\tau}\xi_{\ell^\prime}\big| \\
\leq & \frac{1}{2np^2v^2}\sum_{\ell\in \mathcal{L}^{(1)}_{\tau,k}}\sum_{\ell^\prime\in\mathcal{I}_{\tau,k}^{(1)}}
\bar{a_{\ell}}\bar{a}_{\ell^\prime+\tau}|X_{\ell+k+\tau}^{*}G_{\tau}\xi_{\ell^\prime}|
+\frac{1}{2np^2v}\sum_{\ell\in \mathcal{L}^{(1)}_{\tau,k}}\sum_{\ell^\prime\in\mathcal{I}_{\tau,k}^{(1)}}
\bar{a}_{\ell}\bar{a}_{\ell^\prime}|\xi_{\ell^\prime+\tau}^{*}G_{\tau}^2X_{\ell+k+\tau}| \\
\leq& \frac{1}{np^2v^{3}}\sum_{\ell\in \mathcal{L}^{(1)}_{\tau,k}}\sum_{\ell^\prime\in\mathcal{I}_{\tau,k}^{(1)}}
\bar{a}_{\ell}\bar{a}_{\ell^\prime}\bar{a}_{\ell^\prime+\tau}\|X_{\ell+k+\tau}\| \\
\leq& \frac{L^2_{1}}{np^2v^3}\sum_{\ell\in\mathcal{L}^{(1)}_{\tau,k}}\bar{a}_{\ell}\|X_{\ell+k+\tau}\|,
\end{align*}
where the last inequality holds since $\sum_{\ell^\prime}\bar{a}_{\ell^\prime}\bar{a}_{\ell^\prime+\tau}\leq (\sum_{\ell^\prime}\bar{a}_{\ell^\prime})(\sum_{\ell^\prime}\bar{a}_{\ell^\prime+\tau})\leq L^2_{1}$.  Similar calculations
show that for $l=1,2,3,4$,
\begin{equation}\label{eq:eta_1to4_bound}
|\eta_l(n)| \leq \frac{L^2_{1}}{np^2v^3}\sum_{\ell\in\mathcal{L}^{(s)}_{\tau,k}}\bar{a}_{\ell}\|X_{\ell+k-(-1)^s\tau}\|
\end{equation}
for $s=1,2,1,2$, respectively.

The second term (without the multiplying constant) on the RHS of (\ref{eq:third_order_derivative_resolvent}) can be expressed as follows.
\begin{equation}
\label{eq:tr_deriv_C_tau_second}
\frac{1}{p}\tr\left(\frac{\p \mathbf{C}_{\tau}}{\p Z_{jk}}G_{\tau}\frac{\p \mathbf{C}_{\tau}}{\p Z_{jk}}G_{\tau}\frac{\p \mathbf{C}_{\tau}}{\p Z_{jk}}G_{\tau}^2\right)
=\frac{1}{8p^{5/2}n^{3/2}}\sum_{r,s,t \in \{1,2\}} \tr\big(G_{\tau}L^{(r)}_{\tau,k}G_{\tau}L^{(s)}_{\tau,k}G_{\tau}L^{(t)}_{\tau,k}G_{\tau}\big)
=\sum_{l=5}^{12}\eta_l(n),
\end{equation}
where, for each $l=5,\ldots,12$, $\eta_l(n)$ is of the form
\[
\eta_{l}(n)
=\frac{1}{8p^{5/2}n^{3/2}}\tr(G_{\tau}L^{(r)}_{\tau,k}G_{\tau}L^{(s)}_{\tau,k}G_{\tau}L^{(t)}_{\tau,k}G_{\tau})
\]
where $r,s,t \in \{1,2\}$.
Similarly, the third term (without the multiplying constant) on the RHS of (\ref{eq:third_order_derivative_resolvent}) is
\begin{equation}
\label{eq:tr_deriv_C_tau_third}
\frac{1}{p}\tr\left(\frac{\p \mathbf{C}_{\tau}}{\p Z_{jk}}\frac{\p \mathbf{C}_{\tau}}{\p Z_{jk}}G_{\tau}^2\frac{\p \mathbf{C}_{\tau}}{\p Z_{jk}}G_{\tau}^2\right)
=\frac{1}{8p^{5/2}n^{3/2}}\sum_{r,s,t \in \{1,2\}} \tr\big(L^{(r)}_{\tau,k}L^{(s)}_{\tau,k}G_{\tau}^2L^{(t)}_{\tau,k}G_{\tau}^2\big)
=\sum_{l=13}^{20} \eta_l(n),
\end{equation}
where for each $l=13,\ldots,20$, $\eta_{l}(n)$ is of the form
\[
\eta_{l}(n) = \frac{1}{8p^{5/2}n^{3/2}} \tr\big(L^{(r)}_{\tau,k}L^{(s)}_{\tau,k}G_{\tau}^2L^{(t)}_{\tau,k}G_{\tau}^2\big)
\]
where $r,s,t \in \{1,2\}$.

Since rank$(L_{\tau,k}^{(s)}) \leq 2(q+1)$, by using the fact that for any $p\times p$ matrix
$\mathbf{B}$, $|\tr(\mathbf{B})| \leq \mbox{rank}(\mathbf{B})\|\mathbf{B}\|$, we obtain that
for each $l=5,\ldots,20$,
\begin{equation}\label{eq:eta_5to20_bound}
|\eta_{l}(n)|
\leq\frac{(q+1)}{p^{5/2}n^{3/2}v^{4}}
\bigg(\sum_{\ell\in \mathcal{L}_{\tau,k}^{(r)}}\bar{a}_{\ell}\|X_{\ell+k-(-1)^r\tau}\|\bigg)
\bigg(\sum_{\ell\in \mathcal{L}_{\tau,k}^{(s)}}\bar{a}_{\ell}\|X_{\ell+k-(-1)^s\tau}\|\bigg)
\bigg(\sum_{\ell\in \mathcal{L}_{\tau,k}^{(t)}}\bar{a}_{\ell}\|X_{\ell+k-(-1)^t\tau}\|\bigg).
\end{equation}
for specific combinations of $r,s,t\in \{1,2\}$. To complete the proof, the following two lemmas are needed.

\begin{lemma}
\label{lem:generalized_holder_inequality}
Let $X_1,\ldots,X_m$ be random variables defined in a probability space
$(\Omega, \mathcal{F}, \mP)$. Let  $r\in (0,\infty)$ and $p_1, \ldots, p_{m}>0$ be real number such that
$\sum_{i=1}^{m} 1/p_{i} = 1/r$. Then,
\[
\bigg(\mE \prod_{i=1}^{m} |X_{i}|^{r}\bigg)^{1/r}
\leq \prod_{i=1}^{m}\big(\mE|X_{i}|^{p_{i}}\big)^{1/p_{i}}.
\]
\end{lemma}
Proof of Lemma \ref{lem:generalized_holder_inequality} is a straightforward application of H\"{o}lder's inequality.

\begin{lemma}
\label{lem:estimation_X_M}
Let $Z_{jt}$'s be independent with $\mathbb{E}(Z_{11}) = 0$, $\mathbb{E}|Z_{11}|^2=1$, $\mathbb{E}|Z_{11}|^4 \leq \mu_4 < \infty$
and $|Z_{11}| \leq n^{1/4}\epsilon_p$. Also, let $X_t = \sum_{\ell=0}^q \mathbf{A}_\ell Z_{t-\ell}$
where $\|\mathbf{A}_\ell\| \leq \bar{a}_\ell$ for all $\ell$ with $\sum_{\ell=0}^q \bar{a}_\ell \leq L_1 < \infty$.
Then, for integers $k \geq 1$,
\[
\mE\|X_{t}\|^{2k} \leq \bar{C}_k L_1^{2k} \left(p^{2k} + \mu_4^{k/2} p^{k/2} +  \mu_4 p (n^{1/4}\epsilon_p)^{(2k-4)_+}\right)
\]
where $\bar{C}_k$'s are positive constants that only depend on $k$ and $(x)_+ = \max\{x,0\}$ for $x\in \mathbb{R}$.
\end{lemma}
\begin{proof}
First, we consider the case of $k=1$.
\[
\mE[\|X_{t}\|^{2}]
=\mE \bigg[X^{*}_{t}X_{t}\bigg]
=\mE \bigg[\sum_{\ell=0}^{q}\sum_{m=0}^{q}Z^{*}_{t-\ell}\mathbf{A}^{*}_{\ell}\mathbf{A}_{m}Z_{t-m}\bigg]
=\mE \bigg[\sum_{\ell=0}^{q}Z_{t-\ell}^{*}\mathbf{A}^{*}_{\ell}\mathbf{A}_{\ell}Z_{t-\ell}\bigg]
=\sum_{\ell=0}^{q}\tr(\mathbf{A}^2_{\ell}).
\]
Then, by the fact that $\sum_{\ell=0}^{q}\bar{a}_{\ell}\leq L_1<\infty$,
\begin{equation*}
\frac{1}{p}\mE[\|X_{t}\|^2]
= \sum_{\ell=0}^{q}\frac{1}{p}\tr(\mathbf{A}^2_{\ell})
\leq \sum_{\ell=0}^{q}\|\mathbf{A}_{\ell}\|^2
\leq \sum_{\ell=0}^{q}\bar{a}^2_{\ell}
\leq \bigg(\sum_{\ell=0}^{q}\bar{a}_{\ell}\bigg)^2
<L^2_{1}.
\end{equation*}
This proves the result for $k=1$.
Next,
\[
\|X_t\| \leq \sum_{\ell=0}^q \|\mathbf{A}_\ell\| \|Z_{t-\ell}| \leq \sum_{\ell=0}^q \bar{a}_{\ell}  \|Z_{t-\ell}|,
\]
and hence, for $k\geq 2$,
\begin{eqnarray}\label{eq:expectation_norm_X_t_bound}
\mathbb{E}\|X_t\|^{2k} &\leq& \sum_{\ell_1=0}^q\cdots\sum_{\ell_{2k}=0}^q\left(\prod_{j=1}^{2k} \bar{a}_{\ell_j}\right)
\mathbb{E}\left(\prod_{j=1}^{2k} \|Z_{t-\ell_j}\|\right) \nonumber\\
&\leq& L_1^{2k} \max_{1\leq j \leq 2k} \mathbb{E}\left(\prod_{j=1}^{2k} \|Z_{t-\ell_j}\|\right) ~\leq~ L_1^{2k}
\mathbb{E}\|Z_1\|^{2k},
\end{eqnarray}
where the last inequality follows from an application of
Lemma \ref{lem:generalized_holder_inequality} and the fact that $Z_j$'s are i.i.d.
Also, for $k \geq 2$, $\mathbb{E}|Z_{11}|^{2k} \leq \mu_4 (n^{1/4}\epsilon_p)^{2k-4}$. Thus,
by Lemma \ref{lemma:moments_of_quadratic_forms}, we have
\begin{equation}\label{eq:expectation_norm_Z_1_bound}
\mathbb{E}|\|Z_1\|^2 - p|^k \leq C_k \left[\mu_4 p (n^{1/4}\epsilon_p)^{2k-4} + \mu_4^{k/2} p^{k/2}\right].
\end{equation}
The result follows from (\ref{eq:expectation_norm_X_t_bound}), (\ref{eq:expectation_norm_Z_1_bound}) and
the independence of the $Z_{jt}$'s.
\end{proof}


\vskip.1in
In the following we use $M$ to indicate a generic positive finite constant whose value changes from one expression
to another.  Recalling (\ref{eq:eta_1to4_bound}) and applying Lemmas \ref{lem:estimation_X_M} and \ref{lem:generalized_holder_inequality} we
obtain that for each $l=1,\ldots,4$
\begin{align}
\label{eq:E_Z_jk_eta_1to4_bound}
\mE \left[\left|Z_{jk}^3\eta_l(n)\right|\right]
\leq& \frac{L^2_{1}}{np^2v^{3}}\mE\sum_{\ell\in\mathcal{L}^{(1)}_{\tau,k}}\bar{a}_{\ell}|Z_{jk}|^3\|X_{\ell+k-(-1)^s\tau}\| \nonumber\\
\leq&\frac{L^2_{1}}{np^2v^3}\sum_{\ell\in\mathcal{L}^{(1)}_{\tau,k}}\bar{a}_{\ell}(\mE|Z_{jk}|^4)^{3/4}(\mE\|X_{\ell+k+\tau}\|^{4})^{1/4}
~\leq~ \frac{M}{np^{3/2}v^3},
\end{align}
for some $s \in \{1,2\}$,  where the second inequality holds by $\sum_{\ell\in\mathcal{L}^{(s)}_{\tau,k}}\bar{a}_{\ell}\leq L_1$.

Next, for $l=5,\ldots,20$, by  (\ref{eq:eta_5to20_bound}) and Lemma \ref{lem:generalized_holder_inequality},
for some $r,s,t \in \{1,2\}$,
\begin{align}
\label{eq:E_Z_jk_eta_5to20_bound}
\mE&\left[|Z_{jk}^{3}\eta_{l}(n)|\right] \nonumber\\
&\leq \frac{(q+1)}{p^{5/2}n^{3/2}v^{4}}\mE\bigg[
\sum_{\ell_1\in\mathcal{L}^{(i)}_{\tau,k}}\sum_{\ell_2\in\mathcal{L}^{(j)}_{\tau,k}}\sum_{\ell_3\in\mathcal{L}^{(r)}_{\tau,k}}
\bar{a}_{\ell_1}\bar{a}_{\ell_2}\bar{a}_{\ell_3}|Z_{jk}|^3\|X_{\ell_1+k+\tau}\|\|X_{\ell_2+k+\tau}\|\|X_{\ell_3+k+\tau}\| \bigg] \nonumber\\
&\leq \frac{(q+1)}{p^{5/2}n^{3/2}v^{4}}
\sum_{\ell_1\in\mathcal{L}^{(r)}_{\tau,k}}\sum_{\ell_2\in\mathcal{L}^{(s)}_{\tau,k}}\sum_{\ell_3\in\mathcal{L}^{(t)}_{\tau,k}}
\bar{a}_{\ell_1}\bar{a}_{\ell_2}\bar{a}_{\ell_3}\big(\mE\big[|Z_{jk}|^{6}\big]\big)^{1/2} \nonumber\\
&\qquad\qquad\times
\Big[\big(\mE[\|X_{\ell_1+k-(-1)^r\tau}\|^6]\big)^{1/3}\big(\mE[\|X_{\ell_2+k-(-1)^s\tau}\|^6]\big)^{1/3}
\big(\mE[\|X_{\ell_3+k-(-1)^t\tau}\|^6]\big)^{1/3}\Big]^{1/2} \nonumber\\
&\leq \frac{M(q+1)}{p^{5/2}n^{3/2} v^4} n^{1/4}\epsilon_p\max\{p^{3/2}, p^{1/2} n^{1/4}\epsilon_{p}\}
\end{align}
where the last inequality holds because of $|Z_{jk}|\leq n^{1/4}\epsilon_{p}$ and Lemma \ref{lem:estimation_X_M}.
Finally, combining (\ref{eq:third_order_derivative_resolvent}), (\ref{eq:expression_third_derivative_C_tau_Z}),
(\ref{eq:tr_deriv_C_tau_first}), (\ref{eq:tr_deriv_C_tau_second}), (\ref{eq:tr_deriv_C_tau_third}),
(\ref{eq:E_Z_jk_eta_1to4_bound}) and (\ref{eq:E_Z_jk_eta_5to20_bound}) and using the fact that $\xi \in [0,1]$,
we can concluded that
\begin{eqnarray*}
&& \sum_{k=1}^{\bar{m}_n}\int_{0}^{1}(1-\xi)^{2}\left(\mE\big[\big|(Y_{k}^\mathbf{R})^{3}\p_{k}^{3}f(T_{k}^{(1)}(\xi))\big|\big]
+\mE\big[\big|(\tilde{Y}_{k}^{\mathbf{R}})^{3}\p_{k}^{3}f(\hat{T}_k^{(1)}(\xi))\big|\big]\right)d\xi \\
&\leq& M \max\{\frac{1}{v^3},\frac{1}{v^4}\} \max\{\epsilon_p \frac{(q+1)}{n^{1/4}},\epsilon^2_{p}\frac{(q+1)}{p},p^{-1/2} \}\to 0
\end{eqnarray*}
by the fact that $q = O(p^{1/4})$ and $p = o(n)$. This completes the proof that (\ref{eq:mean_diff_telescoping})
converges to the zero when $Z_{jt}$'s are real valued. Proof in the complex valued case follows from this fact,
and the discussion in the paragraph where equation (\ref{eq:remainder_complex}) appears.


\end{document}